\documentclass[a4,12pt]{amsart}
\usepackage{amssymb,amstext,amsmath,amscd,amsthm,amsfonts,enumerate,latexsym}
\usepackage{color}
\usepackage[dvipdfmx]{graphicx}
\usepackage[all]{xy}
\usepackage{stmaryrd} 
\usepackage{ytableau}
\usepackage[dvipdfmx]{hyperref}

\usepackage{bm}

\usepackage{mathptm}

\theoremstyle{plain}
\newtheorem{thm}{Theorem}[section]

\newtheorem*{thm*}{Theorem}
\newtheorem*{cor*}{Corollary}

\newtheorem{prop}[thm]{Proposition}

\newtheorem{lem}[thm]{Lemma}
\newtheorem{cor}[thm]{Corollary}

\newtheorem{claim}[thm]{Claim}
\newtheorem{fact}[thm]{Fact}
\newtheorem*{claim*}{Claim}

\theoremstyle{definition}
\newtheorem{defn}[thm]{Definition}

\newtheorem{ex}[thm]{Example}

\newtheorem{setup}[thm]{Setup}

\theoremstyle{remark}
\newtheorem{rem}[thm]{Remark}

\numberwithin{equation}{thm}

\newtheorem*{ac}{Acknowledgments}

\def\xx{\text{{\boldmath$x$}}}

\def\Ext{\operatorname{Ext}}

\def\Im{\operatorname{Im}}

\def\Ker{\operatorname{Ker}}

\def\bbZ{\mathbb{Z}}

\def\Assh{\operatorname{Assh}}

\def\Proj{\operatorname{Proj}}

\def\mod{\mathrm{mod}}

\def\rank{\operatorname{rank}}

\def\e{\mathrm{e}}
\def\m{\mathfrak m}
\def\n{\mathfrak n}

\def\p{\mathfrak p}
\def\q{\mathfrak q}

\newcommand{\rma}{\mathrm{a}}

\newcommand{\rmc}{\mathrm{c}}

\newcommand{\rme}{\mathrm{e}}
\newcommand{\rmf}{\mathrm{f}}

\newcommand{\rmr}{\mathrm{r}}

\newcommand{\rmI}{\mathrm{I}}

\newcommand{\rmK}{\mathrm{K}}

\newcommand{\rmQ}{\mathrm{Q}}

\newcommand{\calB}{\mathcal{B}}

\newcommand{\calF}{\mathcal{F}}

\newcommand{\calR}{\mathcal{R}}
\newcommand{\calS}{\mathcal{S}}
\newcommand{\calT}{\mathcal{T}}

\newcommand{\fka}{\mathfrak{a}}

\newcommand{\fkc}{\mathfrak{c}}

\newcommand{\fkm}{\mathfrak{m}}

\newcommand{\fkp}{\mathfrak{p}}
\newcommand{\fkq}{\mathfrak{q}}

\newcommand{\mapright}[1]{%
\smash{\mathop{%
\hbox to 1cm{\rightarrowfill}}\limits^{#1}}}

\newcommand{\mapleft}[1]{%
\smash{\mathop{%
\hbox to 1cm{\leftarrowfill}}\limits_{#1}}}

\def\depth{\operatorname{depth}}

\def\Supp{\operatorname{Supp}}

\def\Ass{\operatorname{Ass}}

\def\height{\mathrm{ht}}
\def\Spec{\operatorname{Spec}}

\def\gr{\mbox{\rm gr}}
\def\red{\operatorname{red}}

\def\Cl{\operatorname{Cl}}
\def\cl{\operatorname{cl}}

\def\yy{\text{\boldmath $y$}}

\title[Goto rings]{Goto rings}

\author[Naoki Endo]{Naoki Endo}
\address{School of Political Science and Economics, Meiji University, 1-9-1 Eifuku, Suginami-ku, Tokyo 168-8555, Japan}
\email{endo@meiji.ac.jp}
\urladdr{https://www.isc.meiji.ac.jp/~endo/}

\thanks{The author was partially supported by JSPS Grant-in-Aid for Young Scientists 20K14299 and JSPS Grant-in-Aid for Scientific Research (C) 23K03058.}

\begin{document}

\maketitle

\setlength{\baselineskip}{15.3pt}

\begin{center} 
Dedicated to the memory of Shiro Goto
 \end{center}

\begin{abstract}
As part of stratification of Cohen-Macaulay rings, we introduce and develop the theory of Goto rings, generalizing the notion of almost Gorenstein rings originally defined by V. Barucci and R. Fr{\"o}berg in 1997. 
What has dominated the series of researches on almost Gorenstein rings is the fact that the reduction numbers of extended canonical ideals are at most $2$; we define Goto rings as Cohen-Macaulay rings admitting such extended canonical ideals. We provide a characterization of Goto rings in terms of the structure of Sally modules and determine the Hilbert functions of them.
Various examples of Goto rings that come from numerical semigroups, idealizations, fiber products, and equimultiple Ulrich ideals are explored as well.
\end{abstract}




{\footnotesize \tableofcontents}


\section{Introduction}\label{sec1}

The objective of the present paper is, so as to stratify Cohen-Macaulay rings, to establish the theory of Goto rings as a new class of rings which fills in a gap between Gorenstein and Cohen-Macaulay properties. 
Underlying this we have a natural and naive query of why there are so many Cohen-Macaulay rings which are not Gorenstein. Although Cohen-Macaulay and Gorenstein rings are the two wings that have led to the rapid advancement of modern commutative ring theory, there had been little analysis on the difference between them. Yet the above question is attractive and has a significance, for a ring theoretic viewpoint in the sense of responding to the strong desire to classify local rings in detail, and for the development of other fields closely related to commutative algebras. 

The first step to attack the question should be considered to find a new class of rings, which may not be Gorenstein, but sufficiently good next to Gorenstein rings. The notion of almost Gorenstein rings is one of the candidates for such classes. 
Historically the theory of almost Gorenstein rings was introduced by V. Barucci and R. Fr{\"o}berg \cite{BF} in the case where the local rings are analytically unramified and of dimension one, e.g., numerical semigroup rings over a field. In 2013, their work inspired S. Goto, N. Matsuoka, and T. T. Phuong to extend the notion of almost Gorenstein rings for arbitrary one-dimensional Cohen-Macaulay local rings. 
More precisely, a Cohen-Macaulay local ring $R$ with $\dim R=1$ is called {\it almost Gorenstein} if $R$ admits a canonical ideal $I$ such that $\rme_1(I) \le \rmr(R)$, where $\rme_1(I)$ denotes the first Hilbert coefficients of $R$ with respect to $I$ and $\rmr(R)$ is its Cohen-Macaulay type (\cite[Definition 3.1]{GMP}). Two years later, Goto, R. Takahashi, and the author of this paper defined almost Gorenstein graded/local rings of arbitrary dimension. 
Let $R$ be a Cohen-Macaulay local ring with maximal ideal $\m$. Then $R$ is said to be an {\it almost Gorenstein ring} if $R$ admits a canonical module $\rmK_R$ and there exists an exact sequence
$$
0 \to R \to \rmK_R \to C \to 0
$$ 
of $R$-modules such that $\mu_R(C) = \rme^0_\m(C)$ (\cite[Definition 3.3]{GTT}). Here, $\mu_R(-)$ (resp. $\rme^0_\m(-)$) denotes the number of elements in a minimal system of generators (resp. the multiplicity with respect to $\m$). When $\dim R=1$, if $R$ is an almost Gorenstein local ring in the sense of \cite{GTT}, then $R$ is almost Gorenstein in the sense of \cite{GMP}. The converse does not hold in general (\cite[Remark 3.5]{GTT}, see also \cite[Remark 2.10]{GMP}); however it does when $R/\m$ is infinite (\cite[Proposition 3.4]{GTT}).
Since then and up to the present, the question of when various Cohen-Macaulay rings, including Rees algebras, determinantal rings, Stanley-Reisner rings, and others are almost Gorenstein has been studied (\cite{CELW, GMTY1, GMTY2, GMTY3, GMTY4, GRTT, Higashitani, MM, T}). 
Among them, we encounter non-almost Gorenstein rings, but some of which still have good structures; time has come to generalize almost Gorenstein rings. 



The notions of generalized Gorenstein rings (\cite[Definition 3.3]{GK}) and 2-almost Gorenstein rings (\cite[Definition 1.3]{CGKM}) have been proposed as the generalizations of almost Gorenstein rings. In the research on a series of almost Gorenstein rings starting from \cite{BF} of Barucci and Fr{\"o}berg, what has dominated these theories is the fact that the reduction numbers of extended canonical ideals are at most $2$ (\cite[Theorem 3.7]{CGKM}, \cite[Theorem 1.2]{GIKT}, \cite[Corollary 5.3]{GTT}). 
In this paper we focus on the reduction numbers of extended canonical ideals and aim to build a new theory that can be understood these theories in a unified manner. As it turns out, the present research has been strongly inspired by \cite{CGKM, GIKT}, but is not just routine generalization and requires a good deal of technical development.

\vspace{0.3em}

Let $(A, \m)$ be a Cohen-Macaulay local ring with $d = \dim A>0$ admitting a canonical module $\rmK_A$. 
Let $I~(\ne A)$ be an ideal of $A$ such that $I \cong \rmK_A$ as an $A$-module. 
An ideal $J$ is called an {\it extended canonical ideal} of $A$ if $J = I+Q$ for some parameter ideal $Q=(a_1, a_2, \ldots, a_d)$ satisfying that $a_1 \in I$ and $Q$ is a reduction of $J$.


\begin{defn}\label{1.1}
For an integer $n\ge 0$, we say that $A$ is an {\it $n$-Goto ring} if there exists a parameter ideal $Q=(a_1, a_2, \ldots, a_d)$ of $A$ such that $a_1 \in I$, $\rank \calS_Q(J) = n$, and $\calS_Q(J)$ is generated by the homogeneous elements of degree one, where $J=I+Q$ and $\calS_Q(J) = \bigoplus_{i \ge 1} J^{i+1}/JQ^i$ denotes the Sally module of $J$ with respect to $Q$. 
\end{defn}

The condition that $\calS_Q(J)$ is generated by the homogeneous elements of degree one holds if and only if $J^3 = QJ^2$, i.e., the reduction number of $J$  is at most 2 (\cite[Lemma 2.1 (5)]{GNO}). Thus the ideal $J$ appeared in Definition \ref{1.1} forms an extended canonical ideal of $A$, while the rank of Sally modules coincides with the length of $J^2/QJ$ as an $A$-module (Lemma \ref{2.4} (2)). Hence, Goto rings are none other than Cohen-Macaulay rings admitting extended canonical ideals whose reduction numbers are at most 2, and the class of Goto rings provides meticulous stratification of those kind of Cohen-Macaulay rings in terms of the rank of Sally modules. 
Indeed, every $0$-Goto (resp. $1$-Goto) is Gorenstein (resp. non-Gorenstein almost Gorenstein) and the converse holds if $d=1$ (see Section 4), or the field $A/\m$ is infinite (Theorems \ref{11.1}, \ref{11.3}). When $d=1$, every $2$-Goto ring is $2$-almost Gorenstein (\cite[Theorem 3.7]{CGKM}). By \cite[Theorem 1.2]{GIKT} if $A$ is a generalized Gorenstein ring with respect to an $\m$-primary ideal $\fka$, then $A$ is an $n$-Goto ring, where $n$ is the length of the $A$-module $A/\fka$, provided that $A/\m$ is infinite and $A$ is not Gorenstein. Note that the notion of Goto rings in our sense is different from that in \cite[Definition 4.4]{GGHHV}, and  there are no implications between Goto rings and nearly Gorenstein rings in the sense of \cite[Definition 2.2]{HHS}; see Example \ref{10.5e}.

The reduction numbers of ideals are closely related to the coefficients of Hilbert polynomials and blow-up algebras. For an $\m$-primary ideal $J$, we assume $J$ contains a parameter ideal $Q$ as a reduction. Classically one has the inequality
$$
\rme_1(J) \ge \rme_0(J) - \ell_A(A/J)
$$
(\cite[Theorem 1]{Northcott}), where $\rme_i(J)$ is the $i$-th Hilbert coefficients of $A$ with respect to $J$, and the equality $\rme_1(J) = \rme_0(J) - \ell_A(A/J)$ holds if and only if $J^2 = QJ$ (\cite[Theorem 2.1]{Huneke}, \cite[Theorem 4.3]{Ooishi}, see also \cite[Theorem 1.9]{L}). When this is the case,  both of the graded rings $\gr_J(A) = \bigoplus_{n \ge 0}J^n/J^{n+1}$ and $\calF(J) = \bigoplus_{n \ge 0} J^n/\m J^n$ are Cohen-Macaulay; moreover the Rees algebra $\calR(J)$ is Cohen-Macaulay, provided $d \ge 2$. As next border, in \cite{S3} J. Sally characterized the ideals $J$ with $\rme_1(J) = \rme_0(J) - \ell_A(A/J)+1$ and $\rme_2(J)\ne 0$. 
W. V. Vasconcelos introduced Sally modules, recovered Sally's results, and made further progress in this direction, e.g., the rank of $\calS_Q(J)$ is given by $\rme_1(J) - \rme_0(J) + \ell_A(A/J)$; see \cite[Corollary 3.3]{V1}. Thereafter, in \cite{GNO, GNO2} Goto, K. Nishida, and K. Ozeki finally brought fruit to fruition for the theory of Sally modules of rank one.
Whereas they have considered general $\m$-primary ideals, we concentrate in this paper on extended canonical ideals, raise the rank of the Sally modules, and attempt to understand the relation with the structure of base rings more deeply.

We state our results explaining how this paper is organized. In Section \ref{sec2} we summarize some preliminaries on extended canonical ideals and their Sally modules, both of which play an important role in our argument. 
Section \ref{sec3} is devoted to define Goto rings and to explore basic properties. 
As shown in Theorem \ref{3.5}, the Goto property is preserved by taking modulo super-regular sequences. This observation suggests us that the importance of analysis in the one-dimensional case for Goto rings. 

Section \ref{sec4-1} focuses on Goto rings of dimension one. For a one-dimensional Cohen-Macaulay local ring $R$ with maximal ideal $\m$, the condition that $R$ contains a canonical ideal possessing a parameter ideal as a reduction is equivalent to the existence of a {\it fractional canonical ideal} $K$, i.e., $K$ is an $R$-submodule of $\rmQ(R)$ such that $R \subseteq K \subseteq \overline{R}$ and $K \cong \rmK_R$ as an $R$-module, where $\overline{R}$ denotes the integral closure of $R$ in its total ring $\rmQ(R)$ of fractions  and $\rmK_R$ stands for the canonical module of $R$ (\cite[Corollary 2.8]{GMP}). Using this, the ring $R$ is $n$-Goto if and only if $K^2=K^3$ and $\ell_R(K^2/K) = n$, where $\ell_R(-)$ denotes the length as an $R$-module. 
Theorem \ref{4.9a} shows that, for each $n \ge 2$, the ring $R$ is $n$-Goto if and only if the blow-up $B=\bigcup_{n\geq 0}\left[\m^n: \m^n\right]$ of $\m$ is $(n-1)$-Goto, provided that $R$ has minimal multiplicity, the field $R/\m$ is algebraically closed, and $\overline{R}$ is a local ring. 
In Sections \ref{sec4} and \ref{sec6}, we investigate the question of when idealizations and fiber products are Goto rings, which allows us to enrich the theory and produce concrete examples of such rings. We prove that if $R$ is $n$-Goto and the $R$-module $K^2/K$ is cyclic, both of the trivial extension $R \ltimes \fkc$ and the fiber product $R \times_{R/\fkc} R$ are $n$-Goto rings for every $n \ge 1$, where $\fkc = R:R[K]$; see Corollary \ref{5.4}. 
Let us remark here that the cyclicity of $K^2/K$ is satisfied if $R$ is almost Gorenstein, $2$-almost Gorenstein, and generalized Gorenstein rings (Lemma \ref{4.5}, \cite[Theorem 3.11]{GMP}, \cite[Proposition 3.3]{CGKM}, \cite[Theorem 4.11]{GK}). 
When $R$ and $S$ are one-dimensional Cohen-Macaulay local rings with common residue class field $k$, for each $n \ge 2$, the fiber product $R\times_k S$ of the canonical homomorphisms $R \longrightarrow k \longleftarrow S$ is $(n+1)$-Goto provided that $R$ is $n$-Goto and $S$ is $2$-Goto (Theorem \ref{6.1}). 

In Section \ref{sec7-1} we aim to characterize the Goto property for semigroup rings $k[[H]]$ over a field $k$, especially, of three-generated numerical semigroups $H$. Besides we determine all the numerical semigroups $H$ corresponding to the Goto ring $k[[H]]$ with minimal multiplicity $3$ (Corollary \ref{8.1}).
In Section \ref{sec7}, we provide a necessary and sufficient condition for a one-dimensional Cohen-Macaulay ring $R$ of the from $R=T/\fka$ to be a Goto ring, in terms of minimal free presentations of fractional canonical ideals, where $T$ is a regular local ring and $\fka$ is an ideal of $T$. Our results Corollary \ref{10.4a} and Theorem \ref{10.9} correspond to those about almost Gorenstein (resp. $2$-almost Gorenstein) rings given by \cite[Theorem 7.8]{GTT} (resp. \cite[Theorems 2.2, 2.9]{GIT}).   
Even though we need an extra assumption on rings, Corollary \ref{10.4a} provides an explicit system of generators of defining ideals in numerical semigroup rings possessing canonical ideals with reduction number $2$. In addition Theorem \ref{10.9} leads us to obtain, for given integers $n \ge 2$ and $\ell \ge 3$, an example of $n$-Goto rings of dimension $\ell$; see Example \ref{10.14}. 

The notion of Ulrich ideals is defined for $\m$-primary ideals and is one of the modifications of that of stable maximal ideals. The definition (\cite[Definition 1.1]{GOTWY}) was brought about by S. Goto, K. Ozeki, R. Takahashi, K.-i. Watanabe, and K.-i. Yoshida in 2014, where they developed the basic theory, revealing that the behavior of Ulrich ideals has ample information about the singularities of rings. In Section \ref{sec9} we extend the notion to equimultiple ideals and provide a construction of Goto rings of higher dimension; see Theorem \ref{10.4}. 
In Section \ref{sec10}, we investigate the relation between the structure of Sally modules of extended canonical ideals and Goto rings. Theorem \ref{11.8} is a generalization of \cite[Corollary 5.3 (2)]{GTT} and \cite[Theorem 3.7]{CGKM}. As an application, we determine the Hilbert function of Goto rings as well; see Corollary \ref{10.8ab}. In the final section, we explore the existence of a certain embedding of Goto rings into its canonical module. Theorem \ref{11.2a} shows that the cokernel of such embedding is a direct sum of Ulrich modules. 

Throughout this paper, unless otherwise specified, we fix the terminology and notation as follows. 
For an arbitrary commutative ring $A$, an ideal $I$ of $A$ is called {\it regular} if it contains a non-zerodivisor on $A$. We denote by $\rmQ(A)$ the total ring of fractions of $A$. For $A$-submodules $X$ and $Y$ of $\rmQ(A)$, let $X:Y = \{a \in \rmQ(A) \mid aY \subseteq X\}$. If we consider ideals $I, J$ of $A$, we set $I:_AJ =\{a \in A \mid aJ \subseteq I\}$; hence $I:_AJ = (I:J) \cap A$.
A {\it fractional ideal} $I$ of $A$ is a finitely generated $A$-submodule of $\rmQ(A)$ satisfying $\rmQ(A)\cdot I = \rmQ(A)$. 
For an integer $t > 0$, let $\rmI_t(X)$ be the ideal of $A$ generated by all the $t \times t$ minors of a matrix $X$ with entries in $A$. 
When $(A, \m)$ is a Cohen-Macaulay local ring with $d = \dim A$,  for an $A$-module $M$, let $\ell_A(M)$ denote the length of $M$ and $\mu_A(M)$ stand for the number of elements of a minimal system of generators for $M$. We set $v(A) = \mu_A(\m)$ and $\rmr(A) = \ell_A(\Ext^d_A(A/\m, A))$.  
For an $\m$-primary ideal $I$ in $A$, let $\rme_i(I)$ be the $i$-th Hilbert coefficient of $A$ with respect to $I$. In particular, we denote by $\rme(A) = \rme_0(\m)$ the multiplicity of $A$. 
Let $\overline{A}$ be the integral closure of $A$ in $\rmQ(A)$ when $d=1$; otherwise it may denote the residue class ring of $A$ when there is no confusion. 
For a graded module $M$ over a graded ring $R$ and for an integer $\ell$, let $M(\ell)$ denote the graded $R$-module whose underlying $R$-module is the same as that of the $R$-module $M$ and the grading is given by $[M(\ell)]_m = M_{\ell + m}$ for all $m \in \Bbb Z$, where $[-]_m$ denotes the $m$-th homogeneous component. 












\section{Preliminaries}\label{sec2}


Let $(A, \m)$ be a Cohen-Macaulay local ring with $d=\dim A>0$ admitting a canonical module $\rmK_A$. An ideal $I$ of $A$ is called a {\it canonical ideal} of $A$ if $I \ne A$ and $I \cong \rmK_A$ as an $A$-module. 
Recall that $A$ possesses a canonical ideal if and only if the local ring $A_{\p}$ is Gorenstein for every $\p \in \Spec A$ with $\dim A/\p =d$; see \cite[Proposition 2.3]{EGGHIKMT}. Hence, the canonical ideal exists if the total ring of fractions $\rmQ(A)$ of $A$ is Gorenstein. 
Introduced by D. G. Northcott and D. Rees (\cite{NR}), for ideals $J$ and $Q$ of $A$ with $Q \subseteq J$, we say that $Q$ is a {\it reduction} of $J$ if $J^{r+1} = QJ^r$ for some $r \ge 0$; the least such integer $r$, denoted by $\red_Q(J)$, is called the {\it reduction number} of $J$ with respect to $Q$.  Geometrically, $Q$ is a reduction of $J$ if and only if the canonical morphism $\Proj(A[Jt])\to \Proj (A[Qt])$ is finite, provided that $J$ is an $\m$-primary ideal of $A$, where $t$ denotes an indeterminate over $A$ (\cite[Proposition 1.44]{Vasconcelos}).

We assume that $A$ admits a canonical ideal $I$ of $A$. Let us begin with the following.

\begin{defn}
We say that a parameter ideal $Q =(a_1, a_2, \ldots, a_d)$ of $A$ satisfies the {\it condition $(\sharp)$} if $a_1 \in I$ and $Q$ is a reduction of $Q+I$. An ideal $J$ of $A$ is called an {\it extended canonical ideal} of $A$ if $J=I+Q$ for some parameter ideal $Q$ with condition $(\sharp)$. 
\end{defn}

When $\dim A=1$, a parameter ideal $Q$ satisfying the condition $(\sharp)$ is nothing but a parameter ideal which forms a reduction of $I$. 
The parameter ideals exist if the field $A/\m$ is infinite, or $A$ is {\it analytically irreducible}, i.e., the $\m$-adic completion $\widehat{A}$ of $A$ is an integral domain. Thus the numerical semigroup rings over a field always possess the parameter ideals. There are examples in higher dimensions as well.

\begin{ex}
Let $\ell \ge 3$ be an integer and $T=k[[X_1, X_2, \ldots, X_{\ell}, V_1, V_2, \ldots, V_{\ell-1}]]$ the formal power series ring over a field $k$. 
For given integers $m \ge n \ge 2$, we set
$$
A=T/
\rmI_2
\begin{pmatrix}
X_1^n & X_2 + V_1 & \cdots & X_{\ell-1}+ V_{\ell -2} & X_{\ell} + V_{\ell-1}\\[3pt]
X_2 & X_3 & \cdots & X_{\ell} & X_1^m
\end{pmatrix}.
$$
Let $I = (x_1^n, x_2, \ldots, x_{\ell-1})$ and $Q = (x_1^n, v_1, \ldots, v_{\ell-1})$, where $x_i$ and $v_i~(1 \le i < \ell)$ denote the images of $X_i$ and $V_i$ in $A$, respectively. Then $A$ is a Cohen-Macaulay local ring with $\dim A= \ell$ admitting the canonical ideal $I$, and $Q$ is a parameter ideal of $A$ satisfying the condition $(\sharp)$. See Example \ref{10.14} for the proof. 
\end{ex}

We choose a parameter ideal $Q =(a_1, a_2, \ldots, a_d)$ with condition $(\sharp)$ and set $J=I+Q$. When $d \ge 2$, the ideal $\fkq = (a_2, a_3, \ldots, a_d)$ forms a parameter ideal in a Gorenstein local ring $A/I$, we see that $\fkq \cap I = \fkq I$. Hence, by setting $\overline{A}=A/\fkq$, we get $J \overline{A} \cong I/\fkq I \cong \rmK_{\overline{A}}$, so that $J \overline{A}$ forms a canonical ideal of $\overline{A}$. 
Let
$$
\calR =\calR(J) = A[Jt] \subseteq A[t]\ \ \ \text{and} \ \ \ \calT = \calR(Q) = A[Qt] \subseteq A[t]
$$
be the Rees algebras of $I$ and $Q$, respectively, where $t$ denotes an indeterminate over $A$.  Following the terminology of W. V. Vasconcelos \cite{V1}, we define
$$
\calS_Q(J) = J\calR/J\calT \cong \bigoplus_{i \ge 1} J^{i+1}/JQ^i
$$
and call it the {\it Sally module} of $J$ with respect to $Q$. Since $\calR$ is a module-finite extension of $\calT$, the Sally module $\calS_Q(J)$ is a finitely generated graded $\calT$-module. Note that $J^2 =QJ$ if and only if $\calS_Q(J) = (0)$, and $J^3 = QJ^2$ is equivalent to saying that $\calS_Q(J) = \calT\left[\calS_Q(J)\right]_1$ (\cite[Lemma 2.1 (3), (5)]{GNO}).

We summarize some basic facts on Sally modules that we will use throughout this paper.  
Let $\gr_J(A) = \calR/J\calR\cong \bigoplus_{n \ge 0}J^n/J^{n+1}$ be the associated graded ring of $J$ and set $\p = \m \calT \in \Spec \calT$. In this paper, we consider the module $(0)$ as Cohen-Macaulay.

\begin{fact}[{\cite[Lemma 2.1 (1), Proposition 2.2]{GNO}}]\label{2.3}
The following assertions hold true. 
\begin{enumerate}
\item[$(1)$] $\m^{\ell}\!\cdot\! \calS_Q(J)=(0)$ for all $\ell \gg 0$.
\item[$(2)$] $\Ass_{\calT}\calS_Q(I) \subseteq \{\p\}$; hence $\dim_{\calT} \calS_Q(J) = d$, provided $\calS_Q(J) \ne (0)$. 
\item[$(3)$] $\ell_A(A/J^{n+1}) = \rme_0(J)\binom{n+d}{d} - \left[\rme_0(J) - \ell_A(A/J)\right]\binom{n+d-1}{d-1} - \ell_A(\left[\calS_Q(J)\right]_n)$ for all $n \ge 0$. 
\item[$(4)$] $\rme_1(J) = \rme_0(J) - \ell_A(A/J) + \ell_{\calT_\p}(\left[\calS_Q(I)\right]_\p)$.
\item[$(5)$] 
$\calS_Q(J)$ is a Cohen-Macaulay graded $\calT$-module if and only if $\depth \gr_J(A) \ge d-1$. 
\end{enumerate}
\end{fact}


We define
$$
\rank \calS_Q(J) = \ell_{\calT_{\p}}([\calS_Q(J)]_{\fkp})
$$
and call it the {\it rank} of the Sally module $\calS_Q(J)$.  
Then  
$\rme_1(J) = \rme_0(J) - \ell_A(A/J) + \rank \calS_Q(J)$.


A sequence $x_1, x_2, \ldots, x_{\ell}~(\ell >0)$ of elements in $A$ is called a {\it super-regular sequence} of $A$ with respect to $J$ if $x_1t, x_2t, \ldots, x_{\ell}t \in \calR$ is a regular sequence on $\gr_J(A)$. 
We then have the following.

\begin{lem}\label{2.4}
Suppose that $\calS_Q(J) = \calT\left[\calS_Q(J)\right]_1$. Then the following assertions hold true.
\begin{enumerate}
\item[$(1)$] If $d \ge 2$, then $a_2, a_3, \ldots, a_d \in \fkq$ forms a super-regular sequence of $A$ with respect to $J$. In particular, $\depth \gr_J(A) \ge d-1$.  
\item[$(2)$] The Sally module $\calS_Q(J)$ is Cohen-Macaulay and $\rank \calS_Q(J) = \ell_A(J^2/QJ)$. 
\end{enumerate}
\end{lem}

\begin{proof}
$(1)$ Thanks to \cite[Corollary 2.7]{VV}, it suffices to show that $\fkq \cap J^{m+1} = \fkq J^m$ for all $\m \in \Bbb Z$. We may assume $m \ge 1$. Suppose $m=1$. Since $J^2 = \fkq J + I^2$, we get
$$
\fkq \cap J^2 = \fkq \cap (\fkq J + I^2) = \fkq J + \fkq \cap I^2 \subseteq \fkq J  
$$
because $\fkq \cap I = \fkq I$. Suppose that $m \ge 2$ and our assertion holds for $m-1$. As $J^3 = QJ^2$, we have $J^{m+1} = QJ^m = a_1J^m + \fkq J^m$. Hence
\begin{eqnarray*}
\fkq \cap J^{m+1} &=& \fkq \cap(a_1J^m + \fkq J^m) = \fkq J^m + \fkq \cap (a_1J^m) = \fkq J^m + a_1 (\fkq \cap J^m) \\
& = & \fkq J^m + a_1 (\fkq J^{m-1}) = \fkq J^m
\end{eqnarray*} 
where the third and fourth equalities follow from $\fkq :_A a_1 = \fkq$ and the induction hypothesis on $m$, respectively. Therefore $a_2, a_3, \ldots, a_d \in \fkq$ forms a super-regular sequence of $A$ with respect to $J$.

$(2)$ By Fact \ref{2.3} (3), the Sally module $\calS_Q(J)$ is Cohen-Macaulay as a $\calT$-module. This shows, by \cite[Theorem 2.1]{GGHV2} the equality 
$$
\rank \calS_Q(J) = \sum_{i=1}^{r-1} \ell_A(J^{i+1}/QJ^i)
$$
holds where $r = \red_Q(J)$. 
Since $\calS_Q(J) = \calT\left[\calS_Q(J)\right]_1$, we get $r \le 2$ and hence 
$\rank \calS_Q(J) = \ell_A(J^2/QJ)$.  
\end{proof}

We investigate the relation between the Sally modules and super-regular elements. 


\begin{prop}
Let $(A, \m)$ be a Cohen-Macaulay local ring with $d = \dim A \ge 2$ and $J$ an $\m$-primary ideal of $A$ admitting a parameter ideal $Q$ as a reduction. Let $a \in Q\setminus \m Q$ be a super-regular element of $A$ with respect to $J$. Then one has an exact sequence
$$
0 \to \calS_Q(J)(-1) \overset{\widehat{at}}{\longrightarrow} \calS_Q(J) \to \calS_{Q/(a)}(J/(a)) \to 0
$$
of graded $\calT$-modules, where $\calT=\calR(Q) = A[Qt]$ and $t$ is an indeterminate over $A$. Hence we have an isomorphism
$$
\calS_Q(J)/(at)\calS_Q(J) \cong \calS_{Q/(a)}(J/(a))
$$
of graded $\calT$-modules. 
\end{prop}

\begin{proof}
Let $\overline{J} = J/(a)$ and $\overline{Q} = Q/(a)$. For each $n \ge 1$, we consider the exact sequence
$$
0 \to X \to J^{n+1}/Q^nJ \to \overline{J}^{n+1}/\overline{Q}^n\overline{J} \to 0
$$
of $A$-modules. Then, because $a \in Q$ is super-regular, we get
$$
X = \frac{\left(Q^nJ + (a)\right) \cap J^{n+1}}{Q^nJ} = \frac{Q^nJ + (a) \cap J^{n+1}}{Q^nJ} = \frac{Q^nJ + a J^n}{Q^nJ} \cong \frac{aJ^n}{Q^nJ \cap aJ^n}.
$$
By choosing an ideal $\fkq$ of $A$ with $Q = (a) + \fkq$, we obtain the equalities
$$
Q^nJ \cap aJ^n = aJ^n \cap (aQ^{n-1} + \fkq^n)J = aQ^{n-1}J + (aJ^n \cap \fkq^nJ) = aQ^{n-1}J
$$
where the last comes from the fact that $(a) \cap \fkq^n = a\fkq^n$. Therefore
$$
[\calS_Q(J)]_{n-1} = J^n/Q^{n-1}J \overset{\widehat{a}}{\stackrel{\sim}{\longrightarrow}} a J^n/a Q^{n-1}J = X
$$
as desired. 
\end{proof}

We note the following. 

\begin{prop}\label{2.6}
Let $\varphi : (A, \m) \to (A_1, \m_1)$ be a flat local homomorphim of Noetherian local rings and $J$ an $\m$-primary ideal of $A$ admitting a parameter ideal $Q$ as a reduction. Then one has isomorphisms 
$$
A_1\otimes_A \calS_Q(J)  \cong \calS_{QA_1}(JA_1)  \ \ \ \text{and} \ \ \ A_1\otimes_A \calT\left[\calS_Q(J)\right]_1  \cong \calT_1\left[\calS_{QA_1}(JA_1)\right]_1
$$
of graded $\calT$-modules, where  $\calT =\calR(Q)$ and $\calT_1 = \calR(QA_1)$. 
\end{prop}

\begin{proof}
Let $J_1 = JA_1$ and $Q_1 = QA_1$. The isomorphism 
$A_1 \otimes_A A[t] \overset{\alpha}{\stackrel{\sim}{\longrightarrow}} A_1[t]$ of $A_1$-algebras induces the isomorphisms $A_1 \otimes_A\calR(J) \cong \calR(J_1)$ and $A_1 \otimes_A\calR(Q) \cong \calR(Q_1)$. Then, because $A_1 \otimes_A J\calR(J) \cong J_1\calR(J_1)$ and $A_1 \otimes_A J\calR(Q) \cong J_1\calR(Q_1)$, the flatness of $A_1$ provides a commutative diagram
$$
\xymatrix{
0 \ \ar[r] &A_1 \otimes_A J\calR(Q)  \ar[r]\ar[d]^{\cong} & A_1 \otimes_A J\calR(J)  \ar[r]\ar[d]^{\cong} &A_1 \otimes_A \calS_Q(J)  \ar[r] &  \ 0  \\
0 \ \ar[r] &J_1\calR(Q_1) \ar[r] &  J_1\calR(J_1) \ar[r] & \calS_{Q_1}(J_1) \ar[r] & \ 0
}
$$
of graded $\calT$-modules. 
This yields that $A_1\otimes_A \calS_Q(J)  \cong \calS_{QA_1}(JA_1)$. Similarly we have $A_1\otimes_A \calT\left[\calS_Q(J)\right]_1  \cong \calT_1\left[\calS_{QA_1}(JA_1)\right]_1$. 
\end{proof}

Closing this section we discuss the existence of parameter ideals with condition $(\sharp)$.



\begin{lem}\label{2.7}
Let $(A, \m)$ be a Noetherian local ring with $d = \dim A>0$ and $J$ an $\m$-primary ideal of $A$. If $g_1, g_2, \ldots, g_d \in J$ forms a superficial sequence of $A$ with respect to $J$, then the ideal of $A$ generated by the sequence $g_1, g_2, \ldots, g_d$ is a parameter ideal which is a minimal reduction of $J$. 
\end{lem}

\begin{proof}
Set $Q = (g_1, g_2, \ldots, g_d)$.
Suppose $d=1$. As $J$ is $\m$-primary, we note that $g_1 \in J$ is a system of parameter of $A$. Choose an integer $\ell > 0$ such that $J^{\ell} \subseteq (g_1)$. Since $g_1$ is a superficial element of $A$ with respect to $J$, there exists an integer $N \ge 0$ satisfying $J^{n+1}:_Ag_1 = J^n + \left[(0):_Ag_1\right]$ for all $n \ge N$. Hence the  equality 
$
J^{n+1} = g_1 J^n
$
holds for every $n \ge \max\{N, \ell-1\}$. 

We assume that $d \ge 2$ and our assertion holds for $d-1$. Let $\overline{A} = A/(g_1)$, $\overline{J} = J\overline{A}$, and $\overline{Q} = Q\overline{A}$. Note that $g_1 \in J$ is a subsystem of parameters of $A$ and $\dim \overline{A} = d-1$. For each $2 \le i \le d$, we denote by $\overline{g_i}$ the image of $g_i$ in $\overline{A}$. Since $\overline{g_2}, \overline{g_3}, \ldots, \overline{g_d} \in \overline{J}$ forms a superficial sequence of $\overline{A}$ with respect to $\overline{J}$, the hypothesis of induction on $d$ shows that $\overline{Q}$ is a reduction of $\overline{J}$, i.e., $\overline{J}^{r+1} = \overline{Q}\cdot\overline{J}^r$ for some $r \ge 0$. Besides, we can choose an integer $N \ge 0$ such that 
$J^{n+1}:_Ag_1 = J^n +\left[(0):_Ag_1\right]$ for all $n \ge N$, because $g_1 \in J$ is a superficial element of $A$ with respect to $J$. Therefore, for every $n \ge \max \{N, r\}$, we have
$$
J^{n+1} \subseteq (g_2, g_3, \ldots, g_d)J^n + (g_1)
$$
which yields that
$$
J^{n+1} \subseteq [ (g_2, g_3, \ldots, g_d)J^n + (g_1)]\cap J^{n+1} = (g_2, g_3, \ldots, g_d)J^n + (g_1) \cap J^{n+1} = QJ^n
$$
because $(g_1) \cap J^{n+1} = g_1 J^n$. Hence $J^{n+1} = QJ^n$, so that $Q$ is a reduction of $J$. 
\end{proof}

\begin{rem}
If $a_1, a_2, \ldots, a_{\ell} \in J~(\ell >0)$ forms a super-regular sequence of $A$ with respect to $J$, then the sequence $a_1, a_2, \ldots, a_{\ell}$ forms superficial of $A$ with respect to $J$. 
\end{rem}

\begin{prop}\label{2.9}
Suppose that $A$ has an infinite residue class field. Then there exists a parameter ideal $Q = (a_1, a_2, \ldots, a_d)$ of $A$ which satisfies the condition $(\sharp)$. In particular, the extended canonical ideal exists. 
\end{prop}

\begin{proof}
Since $\dim A/I = d-1$, we choose a parameter ideal $Q'=(f_1, f_2, \ldots, f_d)$ of $A$ such that $f_1 \in I$. We consider the $\m$-primary ideal $J=I+Q'=(f_2, f_3, \ldots, f_d) + I$. Since the field $A/\m$ is infinite, we can choose a superficial sequence $g_1, g_2, \ldots, g_d \in J$ of $A$ with respect to $J$. Thanks to Lemma \ref{2.7}, the ideal $Q=(g_1, g_2, \ldots, g_d)$ is a reduction of $J$. By setting $\overline{A}=A/I$ and $\overline{J} = J\overline{A}$, the ideal $(\overline{g_1}, \overline{g_2}, \ldots, \overline{g_d})$ forms a reduction of $\overline{J}$, where $\overline{(-)}$ denotes the image in $\overline{A}$. Hence, because $\overline{J}$ is a parameter ideal of $\overline{A}$, we get the equality
$$
(\overline{g_1}, \overline{g_2}, \ldots, \overline{g_d}) = (\overline{f_2}, \overline{f_3}, \ldots, \overline{f_d})
$$
inside of the ring $\overline{A}$. Thus 
$$
J=I+(g_1, g_2, \ldots, g_{i-1}, g_{i+1}, \ldots, g_d)
$$ for some $1 \le i \le d$. We set $\fkq = (g_1, g_2, \ldots, g_{i-1}, g_{i+1}, \ldots, g_d)$ and write $g_i = \xi + \eta$ with $\xi \in I$ and $\eta \in \fkq$. Then $Q=(g_1, g_2, \ldots, g_d) = (\xi) + \fkq$ is a parameter ideal of $A$ which satisfies the condition $(\sharp)$. 
\end{proof}



\section{Definition of Goto rings}\label{sec3}


Let $(A, \m)$ be a Cohen-Macaulay local ring with $d=\dim A>0$ admitting a canonical ideal $I$. Namely, $I$ is an ideal of $A$ such that $I \ne A$ and $I \cong \rmK_A$ as an $A$-module. 
Here, we remark that the existence of canonical ideals implicitly assumes the existence of the canonical module $\rmK_A$.

\begin{defn}\label{3.1}
For each integer $n\ge 0$, we say that $A$ is an {\it $n$-Goto ring} if there exists a parameter ideal $Q=(a_1, a_2, \ldots, a_d)$ of $A$ such that the following conditions are satisfied, where $J=I+Q$ and $\calT=\calR(Q)$ denotes the Rees algebra of $Q$.
\begin{enumerate}
\item[$(1)$] $a_1 \in I$.
\item[$(2)$] $\rank \calS_Q(J) = n$.
\item[$(3)$] $\calS_Q(J) = \calT \left[\calS_Q(J)\right]_1$. 
\end{enumerate}
\end{defn}

In Definition \ref{3.1}, the condition $(3)$ is satisfied if and only if $J^3 =QJ^2$; hence $Q$ is a reduction of $J$, so that the parameter ideal $Q$ of $A$ satisfies the condition $(\sharp)$. By Lemma \ref{2.4} (2), the condition $(2)$ ensures that $\ell_A(J^2/QJ) = n$. Thus, roughly speaking, the notion of Goto rings attempts to analyze and understand the structure of Cohen-Macaulay rings admitting extended canonical ideals whose reduction numbers are at most $2$, using the difference between $J^2$ and $QJ$ as a clue. 
The key to the series of researches on almost Gorenstein rings lies in the fact that 
the reduction numbers of extended canonical ideals are $2$. In fact, 
every $0$-Goto (resp. $1$-Goto) ring is Gorenstein (resp. non-Gorenstein almost Gorenstein) and the converse holds if $d=1$ (see Section 4), or the residue class field $A/\m$ of $A$ is infinite (Theorems \ref{11.1}, \ref{11.3}). When $d=1$, every $2$-Goto ring is nothing but $2$-almost Gorenstein (\cite[Theorem 3.7]{CGKM}). Besides, by \cite[Theorem 1.2]{GIKT} if $A$ is a generalized Gorenstein ring with respect to an $\m$-primary ideal $\fka$, then $A$ is an $n$-Goto ring with $n=\ell_A(A/\fka)$, provided that $A/\m$ is infinite and $A$ is not Gorenstein. 


We define Goto rings with fixed parameter ideals as follows.

\begin{defn}
Let $Q =(a_1, a_2, \ldots, a_d)$ be a parameter ideal which satisfies the condition $(\sharp)$. We set $J=I+Q$ and $\calT=\calR(Q)$. For each integer $n\ge 0$, we say that $A$ is an {\it $n$-Goto ring with respect to $Q$} if $\calS_Q(J) = \calT\left[\calS_Q(J)\right]_1$ and $\rank \calS_Q(J)=n$, i.e., $J^3 =QJ^2$ and $\ell_A(J^2/QJ) = n$. 
\end{defn}

Every almost Gorenstein ring is $1$-Goto; hence so are two-dimensional rational singularities and one-dimensional Cohen-Macaulay local rings of finite CM representation type (\cite[Corollary 11.4, Theorem 12.1]{GTT}). 
We provide a list of examples of $n$-Goto rings with $n \ge 2$ that follow from the results proved later; see Examples \ref{4.3}, \ref{5.6a}, \ref{6.3a}. 

\begin{ex}
Let $n \ge 1$ be an integer. 
Let $k[[t]]$ stand for the formal power series ring over a field $k$.
\begin{enumerate}
\item[$(1)$] The semigroup ring $k[[t^3, t^{3n+1}, t^{3n+2}]]$ is $n$-Goto and it is an integral domain. 
\item[$(2)$] The fiber product $k[[t^3, t^{3n+1}, t^{3n+2}]] \times_k k[[t]]$ is $n$-Goto and reduced, but not an integral domain. The ring is not generalized Gorenstein when $n \ge 2$. 
\item[$(3)$] The idealization $k[[t^3, t^{3n+1}, t^{3n+2}]] \ltimes k[[t]]$ is $n$-Goto and it is not a reduced ring. 
\end{enumerate}
\end{ex}


\begin{rem}\label{3.4r}
As each parameter ideal in a Gorenstein ring is an extended canonical ideal, every $n$-Goto ring with $n \ge 1$ is not Gorenstein. Note that, when $\dim A=1$, the Goto property does not depend on the choice of parameter ideals (\cite[Theorem 2.5]{CGKM}). Yet it does when $\dim A \ge 2$.  
For example, we consider the ring $A = T/\rmI_2
\left(\begin{smallmatrix}
X^2 & Y+W & Z \\
Y & Z & X^3
\end{smallmatrix}\right)
$, where $T=k[[X, Y, Z, W]]$ stands for the formal power series ring over an infinite field $k$. 
We denote by $x, y, w$ the images of $X, Y, W$ in $A$, respectively. The ring $A$ admits a canonical ideal $I=(x^2, y)$ and a parameter ideal $Q = (x^2, w)$. Then, by setting $J=I+Q$, we get $J^3 = QJ^2$ and $\ell_A(J^2/QJ) = 2$. Hence $A$ is a $2$-Goto ring with respect to $Q$. In contrast the exact sequence 
$$
0  \to A \overset{\varphi}{\longrightarrow} I \to C \to 0
$$
of $A$-modules where $\varphi(1) = x^2-y$ yields that $A$ is almost Gorenstein, because $\m C = x C$. Here $\m$ denotes the maximal ideal of $A$. Thus 
we can choose a parameter ideal $Q'=(f_1, f_2)$ of $A$ such that $f_1 \in I$ and $\m(I+Q') = \m Q'$ (\cite[Theorem 5.1]{GTT}). Hence $A$ is $1$-Goto with respect to $Q'$ (\cite[Corollary 5.3]{GTT}, see also Theorem \ref{11.3}). 
\end{rem}

Let us investigate basic properties of Goto rings. 

\begin{thm}\label{3.5}
Suppose that $d \ge 2$. Let $n \ge 0$ be an integer and $Q=(a_1, a_2, \ldots, a_d)$ a parameter ideal of $A$ satisfying the condition $(\sharp)$. We set $\fkq =(a_2, a_3, \ldots, a_d)$ and $J=I+Q$. Let $x \in \fkq \setminus \fkm \fkq$ be a super-regular element of $A$ with respect to $J$. Then the following conditions are equivalent. 
\begin{enumerate}
\item[$(1)$] $A$ is an $n$-Goro ring with respect to $Q$. 
\item[$(2)$] $A/(x)$ is an $n$-Goto ring with respect to $Q/(x)$.
\end{enumerate}
\end{thm}

\begin{proof}
Let $\overline{A} = A/(x)$, $\overline{J} = J\overline{A}$, $\overline{I} = I\overline{A}$, and $\overline{Q} = Q\overline{A}$. Then $\overline{I} = (I+(x))/(x) \cong I/xI \cong \rmK_{\overline{A}}$, so that $\overline{I}$ is a canonical ideal of $\overline{A}$. Since $\overline{Q}$ is a reduction of $\overline{J} = \overline{I} + \overline{Q}$ and $\overline{x} \in \overline{I}$, the parameter ideal $\overline{Q}$ satisfies the condition $(\sharp)$ for $\overline{A}$, where $\overline{x}$ denotes the image of $x$ in $\overline{A}$. Note that $x \in \fkq$ is a part of a minimal basis of $Q$. For each $m \in \Bbb Z$, let
$$
0 \to X \to J^{m+1}/Q^nJ \to \overline{J}^{m+1}/\overline{Q}^m\overline{J} \to 0
$$
be the exact sequence of $A$-modules. Then
$$
X = \frac{\left[QJ^m+(x)\right] \cap J^{m+1}}{QJ^m} = \frac{QJ^m + (x) \cap J^{m+1}}{QJ^m} = \frac{QJ^m + xJ^m}{QJ^m} = (0)
$$
because $x \in \fkq$ is a super-regular element of $A$ with respect to $J$. Therefore we get an isomorphism 
$$
J^{m+1}/QJ^m \cong \overline{J}^{m+1}/\overline{Q}^m\overline{J}
$$
of $A$-modules. This shows $\calS_Q(J) = \calT\left[\calS_Q(J)\right]_1$ if and only if $\calS_{\overline{Q}}(\overline{J}) = \calT'\left[\calS_{\overline{Q}}(\overline{J})\right]_1$, where $\calT=\calR(Q)$ and $\calT' = \calR(\overline{Q})$ stand for the Rees algebras of $Q$ and $\overline{Q}$, respectively. Hence, we may assume $\calS_Q(J) = \calT\left[\calS_Q(J)\right]_1$. Then, by Lemma \ref{2.4} (2) the equalities $\rank \calS_Q(J) = \ell_A(J^2/QJ) = \ell_{\overline{A}}(\overline{J}^2/\overline{Q}\overline{J}) = \rank \calS_{\overline{Q}}(\overline{J})$ hold, which complete the proof. 
\end{proof}

\begin{rem}\label{3.5a}
We maintain the notation as in the proof of Theorem \ref{3.5}. Suppose that $\calS_Q(J) = \calT\left[\calS_Q(J)\right]_1$. Let $\overline{\m} = \m \overline{A}$ be the maximal ideal of $\overline{A}$. Then, for each $\ell \in \Bbb Z$, we see that $\m^{\ell} \calS_Q(J) \ne (0)$ if and only if 
$\overline{\m}^{\ell} \calS_{\overline{Q}}(\overline{J}) \ne (0)$. 
\end{rem}

As a consequence of Theorem \ref{3.5}, for given integers $n \ge 2$ and $\ell \ge 3$, there exists an example of $n$-Goto rings of dimension $\ell$. See Example \ref{10.14} for the proof.

\begin{ex}
Let $\ell \ge 3$ be an integer and $T=k[[X_1, X_2, \ldots, X_{\ell}, V_1, V_2, \ldots, V_{\ell-1}]]$ the formal power series ring over a field $k$. 
For given integers $m \ge n \ge 2$, 
$$
A=T/
\rmI_2
\begin{pmatrix}
X_1^n & X_2 + V_1 & \cdots & X_{\ell-1}+ V_{\ell -2} & X_{\ell} + V_{\ell-1}\\[3pt]
X_2 & X_3 & \cdots & X_{\ell} & X_1^m
\end{pmatrix}
$$
is an $n$-Goto ring with $\dim A=\ell$ and $\rmr(A) = \ell -1$. 
\end{ex}



We study the question of how the Goto property is inherited under flat local homomorphisms. 
Let $(A_1, \m_1)$ be a Cohen-Macaulay local ring and $\varphi : A \to A_1$ a flat local homomorphism such that $A_1/\m A_1$ is an Artinian Gorenstein ring. 
Then $IA_1 \cong I\otimes_AA_1 \cong \rmK_{A_1}$ and $IA_1 \subseteq \m A_1 \subsetneq A_1$, that is, $IA_1$ is a canonical ideal of $A_1$. 
Let $Q$ be a parameter ideal of $A$ with condition $(\sharp)$. Then $QA_1$ is a parameter ideal of $A_1$ which satisfies the condition $(\sharp)$ for $A_1$.  
With this notation we have the following. 


\begin{thm}\label{3.7}
Let $n \ge 0$ be an integer. Then the following conditions are equivalent.
\begin{enumerate}
\item[$(1)$] $A_1$ is an $n$-Goto ring with respect to $QA_1$. 
\item[$(2)$] There exists an integer $m >0$ such that $m \mid n$, $A$ is $m$-Goto with respect to $Q$, and $\ell_{A_1}(A_1/\m A_1) = \frac{n}{m}$. 
\end{enumerate}
\end{thm}

\begin{proof}
We set $J=I+Q$, $J_1 = JA_1$, $I_1 = IA_1$, and $Q_1 = QA_1$. By Proposition \ref{2.6}, we have the isomorphisms $A_1\otimes_A \calS_Q(J)  \cong \calS_{Q_1}(J_1)$ and $A_1\otimes_A \calT\left[\calS_Q(J)\right]_1  \cong \calT_1\left[\calS_{Q_1}(J_1)\right]_1$ of graded $\calT$-modules, where $\calT =\calR(Q)$ and $\calT_1 = \calR(Q_1)$. Hence the faithful flatness of $A$ shows that $\calS_Q(J) = \calT\left[\calS_Q(J)\right]_1$ if and only if $\calS_{Q_1}(J_1) = \calT_1\left[\calS_{Q_1}(J_1)\right]_1$. For each $\ell \ge 0$, the equality
$
\ell_{A_1}(A_1/J_1^{\ell + 1}) = \ell_{A_1}(A_1/\m A_1) \cdot \ell_A(A/J^{\ell+1})
$
induces that 
\begin{eqnarray*}
\rank \calS_{Q_1}(J_1) \!\!&=&\!\! \rme_1(J_1) - \rme_0(J_1) + \ell_{A_1}(A_1/J_1) \\
\!\!&=&\!\! \ell_{A_1}(A_1/\m A_1)\left[\rme_1(J) - \rme_0(J) + \ell_A(A/J)\right] \\
\!\!&=&\!\! \ell_{A_1}(A_1/\m A_1) \cdot \rank \calS_Q(J).
\end{eqnarray*}
Hence the equivalence $(1) \Leftrightarrow (2)$ follows by choosing $m = \rank \calS_Q(J)$. 
\end{proof}

As we prove in Theorem \ref{11.3}, every $1$-Goto ring is almost Gorenstein but not a Gorenstein ring. The converse holds if $A$ possesses an infinite residue class field. 

\begin{cor}
Let $n \ge 2$ be a prime integer and suppose that $A/\m$ is infinite. Then $A_1$ is an $n$-Goto ring with respect to $QA_1$ if and only if 
one of the following conditions holds.
\begin{enumerate}
\item[$(1)$] $A$ is a non-Gorenstein almost Gorenstein ring and $\ell_{A_1}(A_1/\m A_1) = n$.
\item[$(2)$] $A$ is an $n$-Goto ring with respect to $Q$ and $\m A_1 = \m_1$.
\end{enumerate}
\end{cor}

\begin{cor}
For an integer $n>0$, $A$ is an $n$-Goto ring with respect to $Q$ if and only if the completion $\widehat{A}$ is an $n$-Goto ring with respect to $\widehat{Q}=Q\widehat{A}$. 
\end{cor}

We give examples that illustrate Theorem \ref{3.7}.

\begin{ex}[cf. {\cite[Example 3.11]{CGKM}}]
Let $n \ge 1$ be an integer. We set $A_1 = A[X]/(X^n + \alpha_1 X^{n-1} + \cdots + \alpha_n)$, where $A[X]$ denotes the polynomial ring over $A$ and $\alpha_i \in \m$ for all $1 \le i \le n$. 
Then $A_1$ is a flat local $A$-algebra with maximal ideal $\m_1 = \m A_1 + XA_1$, $A_1/\m A_1= (A/\m)[X]/(X^n)$ is an Artinian Gorenstein ring, and $\ell_{A_1}(A_1/\m A_1) = n$. 
Hence, if $n \ge 2$ is a prime integer and the field $A/\m$ is infinite, then $A_1$ is an $n$-Goto ring with respect to $Q_1 = QA_1$ if and only if $A$ is a non-Gorenstein almost Gorenstein ring. 
\end{ex}

As we often refer to examples arising from numerical semigroup rings, let us explain the notation and terminology of them. 
We denote by $\Bbb N$ the set of non-negative integers. 
Let $0 < a_1, a_2, \ldots, a_\ell \in \Bbb Z~(\ell >0)$ be integers such that $\gcd(a_1, a_2, \ldots, a_\ell)=1$. We set 
$$
H = \left<a_1, a_2, \ldots, a_\ell\right>=\left\{\sum_{i=1}^\ell c_ia_i ~\middle|~  c_i \in \Bbb N~\text{for~all}~1 \le i \le \ell \right\}
$$
which is called the {\it numerical semigroup} generated by the numbers $\{a_i\}_{1 \le i \le \ell}$. The reader may consult the book \cite{RG} for the fundamental results on numerical semigroups. 
Let $V=k[[t]]$ be the formal power series ring over a field $k$, and define 
$$
R = k[[H]] = k[[t^{a_1}, t^{a_2}, \ldots, t^{a_\ell}]] \subseteq V
$$
which we call the {\it semigroup ring} of $H$ over $k$. 
The ring $R$ is a Cohen-Macaulay local domain with $\dim R=1$ and $\m = (t^{a_1},t^{a_2}, \ldots, t^{a_\ell} )$ is the maximal ideal. 
In addition, $V$ is a birational module-finite extension of $R$, so that  $\overline{R} = V$, where $\overline{R}$ denotes the integral closure of $R$ in its quotient field $\rmQ(R)$.
Let 
$$
\rmc(H) = \min \{n \in \Bbb Z \mid m \in H~\text{for~all}~m \in \Bbb Z~\text{such~that~}m \ge n\}
$$
be the {\it conductor} of $H$ and set $\rmf(H) = \rmc(H) -1$. Hence, $\rmf(H) = \max ~(\Bbb Z \setminus H)$, which is called the {\it Frobenius number} of $H$. Let 
$$
\mathrm{PF}(H) = \{n \in \Bbb Z \setminus H \mid n + a_i \in H~\text{for~all}~1 \le  i \le \ell\}
$$
denote the set of pseudo-Frobenius numbers of $H$. Therefore, $\rmf(H)$ coincides with the $\rma$-invariant of the graded $k$-algebra $k[t^{a_1}, t^{a_2}, \ldots, t^{a_\ell}]$ and $\# \mathrm{PF}(H) = \rmr(R)$; see \cite[Example (2.1.9), Definition (3.1.4)]{GW}.  We set  $f = \rmf(H)$ and $$K = \sum_{c \in \mathrm{PF}(H)}Rt^{f-c}$$ in $V$. Then $K$ is a fractional ideal of $R$ such that $R \subseteq K \subseteq \overline{R}$ and 
$$
K \cong \rmK_R = \sum_{c \in \mathrm{PF}(H)}Rt^{-c}
$$
as an $R$-module (\cite[Example (2.1.9)]{GW}). 

\begin{ex}
Let $K/k$ be a finite extension of fields with $[K:k]\! =\! n \ge 2$. We set $\omega_1 = 1$ and choose a $k$-basis $\{\omega_1, \omega_2, \ldots, \omega_n\}$ of $K$. Let $K[[t]]$ be the formal power series ring over $K$. 
Let $a_1, a_2, \ldots, a_\ell \in \Bbb Z$ be positive integers such that $\gcd(a_1, a_2, \ldots, a_\ell)=1$. We set
$H = \left<a_1, a_2, \ldots, a_\ell\right>$ and choose $0 < a \in H$. We consider the rings
$R =k[[t^{a_1}, t^{a_2}, \ldots, t^{a_{\ell}}]]$ and $R_1=k[[t^{a_1}, t^{a_2}, \ldots, t^{a_{\ell}}, \{\omega_i t^a\}_{1 \le i \le n}]]$.  
Suppose $\rmr(R) \ge 2$. Then $R_1$ is a free $R$-module of rank $n$ and $\ell_{R_1}(R_1/\m R_1) = n$. 
Hence, if $n \ge 2$ is a prime integer, the ring $R_1$ is $n$-Goto if and only if $R$ is a non-Gorenstein almost Gorenstein ring. 
\end{ex}

\begin{ex}
Let $H=\left<a_1, a_2, \ldots, a_{\ell} \right>$ be a numerical semigroup. For an odd integer $0 < \alpha \in H$ such that $\alpha \ne a_i$ for all $1 \le i \le \ell$, we set $H_1 = \left<2a_1, 2a_2, \ldots, 2a_{\ell}, \alpha \right>$. Let $R=k[[H]]$ and $R_1=k[[H_1]]$ be, respectively, the semigroup rings of $H$ and $H_1$ over a field $k$.
Then $R_1$ is a free module of rank $2$ and $\ell_{R_1}(R_1/\m R_1) = 2$. Hence $R_1$ is $2$-Goto if and only if $R$ is a non-Gorenstein almost Gorenstein ring. 
\end{ex}

\begin{rem}
In general, the contraction of a parameter ideal via a flat local homomorphism is not necessarily a parameter ideal. In fact, let $R$ denote a Cohen-Macaulay local ring with $\dim R=1$. We denote by $R_1 = R \ltimes R$ the idealization of $R$ over $R$. We choose an Ulrich ideal $I$ of $R$ generated by two-elements, say non-zerodivisors $a$ and $b$ (see \cite[Definition 1.1]{GOTWY} for the definition of Ulrich ideals). Then $Q_1=\alpha R_1$ is a parameter ideal of $R_1$,  but $Q_1 \cap R = a((a):_Rb) \cong (a):_Rb =I = (a, b)$ is not, where $\alpha = (a, b) \in R_1$. 
\end{rem}


\section{One-dimensional Goto rings}\label{sec4-1}

This section focuses on Goto rings of dimension one. 
Let $(R, \m)$ be a Cohen-Macaulay local ring with $\dim R=1$ admitting a  canonical ideal $I$ of $R$. We fix an integer $n \ge 0$.
Recall that $R$ is an $n$-Goto ring if and only if there exists $a \in I$ such that $I^3 =aI^2$ and $\ell_R(I^2/aI) = n$. 
By \cite[Theorem 2.5]{CGKM}, the Goto property is independent of the choice of canonical ideals and their reductions.
Besides, the existence of canonical ideals $I$ of $R$ containing parameter ideals as reductions is equivalent to saying that the existence of fractional canonical ideals $K$, i.e., $K$ is an $R$-submodule of $\rmQ(R)$ such that $R \subseteq K \subseteq \overline{R}$ and $K \cong \rmK_R$ as an $R$-module, where $\overline{R}$ denotes the integral closure of $R$ in $\rmQ(R)$ and $\rmK_R$ stands for the canonical module of $R$ (\cite[Corollary 2.8]{GMP}). 

In this section, unless otherwise specified, we maintain the setup below. 

\begin{setup}\label{4.1}
Let $(R, \m)$ be a Cohen-Macaulay local ring with $\dim R=1$ possessing a canonical module $\rmK_R$. Let $I$ be a canonical ideal of $R$, that is, $I$ is an ideal of $R$ such that $I \ne R$ and $I \cong \rmK_R$ as an $R$-module. Suppose that $I$ contains a parameter ideal $Q=(a)$ as a reduction. Thus the ideal $Q$ satisfies the condition $(\sharp)$. Let 
$$
K=\frac{I}{a} = \left\{\frac{x}{a} ~\middle|~ x \in I\right\} \subseteq \rmQ(A).
$$
Hence $K$ is fractional canonical ideal of $R$. We set $S =R[K]$ and $\fkc = R:S$.  
\end{setup}



The assumption that the field $A/\m$ is infinite in Theorems \ref{11.1}, \ref{11.3} is used only to assure the existence of reductions of canonical ideals. Thus, under Setup \ref{4.1}, a $0$-Goto ring is equivalent to Gorenstein, and $1$-Goto is exactly the same as a non-Gorenstein almost Gorenstein ring. 
In addition, the notions of $2$-Goto and 2-almost Gorenstein are equivalent.
We begin with the following, which naturally extends \cite[Lemma 3.1]{CGKM}.

\begin{prop}
For each integer $n \ge 0$, the ring $R$ is $n$-Goto if and only if $K^2=K^3$  and $\ell_R(K^2/K)=n$, or, equivalently, $K^2 = K^3$ and $\ell_R(R/\fkc) = n$.
\end{prop}

\begin{proof}
Since $I=aK$, we have $I^3 =aI^2$ if and only if $K^2=K^3$. Hence the isomorphism $I^2/aI \cong K^2/K$ guarantees the first equivalence. The second equivalence follows from the fact that $\rank \calS_Q(I) = \ell_R(R/\fkc)$ (\cite[Theorem 2.5]{CGKM}). 
\end{proof}








The equality $K^2 =K^3$ holds when $\ell_A(A/\fkc) \le 2$ (\cite[Theorems 3.7, 3.16]{GMP}, \cite[Theorem 3.7]{CGKM}); otherwise, if $\ell_A(A/\fkc) \ge 3$, there is an example of a ring $R$ which possesses a fractional canonical ideal $K$ satisfying $K^2 \ne K^3$. Indeed, let $V=k[[t]]$ be the formal power series ring over a field $k$ and set $R = k[[H]]$ in $V$, where $H=\left<4, 5, 11\right>$. Then $K=R + Rt$ is a fractional canonical ideal of $R$, $K^2 \ne K^3$, and $\ell_R(R/\fkc) = \ell_R(R/t^8V) = 3$ (see also Example \ref{10.5e}).

When $K^2 = K^3$, the condition $\m^{n-1}K^2 \not\subseteq K$, which is equivalent to $\m^{n-1}\calS_Q(I) \ne (0)$, plays an important role in the theory of Goto rings; see e.g., Corollaries \ref{10.4a}, \ref{11.9}. Besides, every $2$-almost Gorenstein ring satisfies the condition (\cite[Proposition 2.3]{CGKM}). 


\begin{lem}
Let $n \ge 2$ be an integer. Suppose that $R$ is an $n$-Goto ring. Then $\m^{n-1}K^2 \not\subseteq K$ if and only if $v(R/\fkc) = 1$. When this is the case, the ring $R/\fkc$ is Gorenstein. 
\end{lem}

\begin{proof}
Let $\overline{R} = R/\fkc$ and $\overline{\m} = \m \overline{R}$. Since $K:K=R$ (\cite[Bemerkung 2.5]{HK}) and $K^2 = K^3$, the condition $\m^{n-1}K^2 \not\subseteq K$ is equivalent to saying that $\m \not\subseteq K:K^2 = (K:K):K = R:K = \fkc$, that is, $\overline{\m}^{n-1} \ne (0)$ in $\overline{R}$. We first assume that $\m^{n-1}K^2 \not\subseteq K$. Then $\overline{\m}^i \ne \overline{\m}^{i+1}$ for all $0 \le i \le n-1$, so that $\ell_{\overline{R}}(\overline{\m}^i/\overline{\m}^{i+1}) = 1$ because $\ell_R(R/\fkc) = n$. Hence $v(R/\fkc) = 1$. Conversely, suppose $v(R/\fkc) = 1$. For each $0 \le i \le n-1$, we then have $\overline{\m}^i \ne \overline{\m}^{i+1}$ is cyclic as an $\overline{R}/\overline{\m}$-module. As $R$ is $n$-Goto, we get $\ell_{\overline{R}}(\overline{\m}^i/\overline{\m}^{i+1}) = 1$. In particular, $\overline{\m}^{n-1} \ne (0)$. The last assertion follows from $v(R/\fkc) = 1$ immediately. 
\end{proof}

The Gorensteinness of $R/\fkc$ can be characterized as follows. Note that if $R$ is non-Gorenstein almost Gorenstein, $2$-almost Gorenstein, or generalized Gorenstein, the ring $R/\fkc$ is Gorenstein; see \cite[Theorem 3.11]{GMP}, \cite[Proposition 3.3]{CGKM}, \cite[Theorem 4.11]{GK}.

\begin{lem}\label{4.5}
Suppose $R$ is not a Gorenstein ring. Then $R/\fkc$ is Gorenstein if and only if $K^2 = K^3$ and $\mu_R(K^2/K) = 1$.  
\end{lem}

\begin{proof}
Suppose that $R/\fkc$ is Gorenstein. As $S/K$ is a canonical module of $R/\fkc$ (e.g., \cite[Lemma 4.10]{GK}), we get $S/K \cong R/\fkc$ as an $R$-module. Choose a non-zerodivisor $b \in \fkc$ on $R$. We set $I' = bK$ and $J=bS \subsetneq R$. Then $\mu_R(J/I') = \mu_R(S/K) = 1$ and $J^2 = (bS)^2 = b(bS) = bJ$, whence, by \cite[Proposition 2.6]{GNO2} we get $I'^3 = bI'^2$. Thus $K^2=K^3$ and $\mu_R(K^2/K) = \mu_R(S/K) = 1$. 
The converse holds because $\rmr(R/\fkc)  = \mu_R(K^2/K) = 1$. 
\end{proof}

In general, the equality $K^2 = K^3$ does not imply $\mu_R(K^2/K) = 1$ as we show next.

\begin{ex}
Let $V=k[[t]]$ be the formal power series ring over a field $k$. We consider the semigroup ring $R = k[[H]]$ in $V$, where $H=\left<5, 11, 13, 19\right>$. Since $\mathrm{PF}(H) =\{8, 14, 17\}$, the fractional canonical ideal $K$ has the form $K=R + Rt^3 + Rt^9$. Then $K^2 = K + Rt^6 + Rt^{12}= K^3$ and $\mu_R(K^2/K) = 2$. In addition, we have $\ell_R(K^2/K) = 3$, so that $R$ is $3$-Goto; while $R$ is not generalized Gorenstein because $R/\fkc$ is not a Gorenstein ring.
\end{ex}




If $R$ is an $n$-Goto ring with $\rmr(R)=2$, the length $\ell_R(K^2/K)$ can be replaced by $\ell_R(K/R)$. 

\begin{prop}\label{4.6a}
Suppose that $\rmr(R) = 2$. Then $K^2 = K^3$ if and only if $K/R \cong R/\fkc$ as an $R$-module, or, equivalently, $R$ is a generalized Gorenstein ring. Hence, for an integer $n \ge 1$, the ring $R$ is $n$-Goto if and only if $K^2=K^3$ and $\ell_R(K/R) = n$. 
\end{prop}

\begin{proof}
As $\rmr(R)=2$, we write $K = R + Rf$ with $f \in K$. Then $K^2 = K + Rf^2$, whence $\mu_R(K^2/K) = 1$ because $K^2 \ne K$. If $K^2=K^3$, the ring $R/\fkc$ is Gorenstein and $\fkc = R:K$. As $K/R$ is faithful as an $R/\fkc$-module, we have $K/R \cong R/\fkc$. Conversely, if $K/R \cong R/\fkc$, the annihilator of both sides shows that the equality $\fkc = R:K$ holds, i.e., $K^2=K^3$. In addition, by \cite[Corollary 4.14]{GK} $R$ is generalized Gorenstein if and only if $K^2 = K^3$. 
\end{proof}

By the proof of Proposition \ref{4.6a}, we remark that $\mu_R(K^2/K) = 1$ holds when $\rmr(R) = 2$. 
Recall that the ring $R$ is said to have {\it minimal multiplicity} if $\rme(R) =v(R)$. 

\begin{cor}[{cf. \cite[Proposition 3.12]{CGKM}}]\label{4.7a}
Let $n \ge 1$ be an integer. Suppose that $\rme(R) = 3$ and $R$ has minimal multiplicity. Then $R$ is an $n$-Goto ring if and only if $\ell_R(K/R) = n$. 
\end{cor}

\begin{proof}
Passing to the flat local homomorphism $R \to R[X]_{\m R[X]}$, by Theorem \ref{3.7} we may assume $R/\m$ is infinite. As $\rme(R) = v(R) = 3$, we note that $\rmr(R) = 2$. Since $I=aK$ is an $\m$-primary ideal of $R$, we get $\mu_R(I^3) \le \rme(R) = 3$ (\cite{S1}). Hence there exists $b \in I$ such that $I^3 = bI^2$ (\cite{ES}), so that $K^2 = K^3$ because the reduction number of $I$ is independent of the choice of its reductions. 
\end{proof}

Let us note some examples of Goto rings of dimension one.

\begin{ex}\label{4.3}
Let $V=k[[t]]$ be the formal power series ring over a field $k$. We consider the semigroup ring $R_i = k[[H_i]]~(i=1, 2)$ in $V$, where $H_i$ has the following form.
\begin{enumerate}
\item[$(1)$] $H_1=\left<3, 3n+1, 3n+2\right>~(n \ge 1)$.
\item[$(2)$] $H_2=\left<e, \{en-e+i\}_{3 \le i \le e-1}, en+1, en+2\right>~(n \ge 2,\, e \ge 4)$. 
\end{enumerate}
Then, because $\mathrm{PF}(H_1) =\{3n-2, 3n-1\}$, $K_1 = R_1 + R_1t$ is the fractional canonical ideal of $R_1$, so that $\ell_{R_1}(K_1/R_1) = n$. By Corollary \ref{4.7a}, the ring $R_1$ is $n$-Goto. 
On the other hand, since $\mathrm{PF}(H_2) =\{en-2e+i \mid 3 \le i \le e-1\} \cup \{en-e+1, en-e+2\}$, the fractional canonical ideal $K_2$ of $R_2$ has the form $K_2=R_2 + R_2t + R_2t^3 + R_2t^4 + \cdots + R_2t^{e-1}$. Then $K_2^2 = K_2^3 = V$ and $\ell_{R_2}(K_2^2/K_2) = \ell_{R_2}(V/K_2) = n$. Therefore $R_2$ is an $n$-Goto ring. Note that $R_2$ is not generalized Gorenstein. Indeed, if $R$ is generalized Gorenstein, then $K_2/R_2 \cong (R_2/\fkc_2)^{\oplus (r-1)}$, where $\fkc_2 = R_2:R_2[K_2]$ and $r = \rmr(R_2)$. Since $R_2$ has minimal multiplicity, we get the equalities
$$
n + (e-3)(n-1) = \ell_{R_2}(K_2/R_2) = (r-1)\ell_{R_2}(R_2/\fkc_2) = (e-1-1)n.
$$
This induces $(e-3)(n-1) = (e-3)n$ which makes a contradiction. Hence $R_2$ is not a generalized Gorenstein ring. 
\end{ex}


We study the Goto property of the blow-up $B=\bigcup_{n\geq 0}\left[\m^n: \m^n\right]$ of the maximal ideal $\m$. 
By \cite[Theorem 5.1]{GMP}, the ring $R$ is $1$-Goto having minimal multiplicity if and only if $B$ is $0$-Goto but $R$ is not. 


\begin{thm}\label{4.9a}
Let $n \ge 2$ be an integer. Set $B=\bigcup_{n\geq 0}\left[\m^n: \m^n\right]$.  
Suppose that $R$ has minimal multiplicity, $B$ is a local ring, and $R/\m \cong B/\n$, where $\n$ denotes the maximal ideal of $B$. Then the following conditions are equivalent. 
\begin{enumerate}
\item[$(1)$] $R$ is an $n$-Goto ring
\item[$(2)$] $B$ is an $(n-1)$-Goto ring.
\end{enumerate}
\end{thm}

\begin{proof}
Note that $B=\bigcup_{n\geq 0}\left[\m^n: \m^n\right]=\m:\m$, because $R$ has minimal multiplicity.  
Since every finitely generated $R$-subalgebra of $\overline{R}$ is Gorenstein if $\rme(R)\le 2$, we may assume that $R$ is not a Gorenstein ring. 
Then $L=BK$ is a $B$-submodule of $\rmQ(B)$ such that $B \subseteq L \subseteq \overline{B}$ and $L \cong \rmK_B$ as a $B$-module, so $L$ is a fractional canonical ideal of $B$ and $L=K:\m$, where $\overline{B}$ denotes the integral closure of $B$ in $\rmQ(B)$ (\cite[Proposition 5.1]{CGKM}). 

$(1) \Rightarrow (2)$ We have $S = R[K] = K^2$, because $K^2 = K^3$. Since $K \subseteq L \subseteq S$, we see that $L^2 = S = L^3$. Then $\ell_B(L^2/L) = \ell_R(L^2/L) = \ell_R(S/K) - \ell_R(L/K) = n-1$, because $R/\m \cong B/\n$ and $L=K:\m$. Hence $B$ is an $(n-1)$-Goto ring.

$(2) \Rightarrow (1)$ Since $S=B[L] = L^2$, we have $\ell_R(S/K) = \ell_R(S/L) + \ell_R(L/K) = n$. Hence it suffices to show that $K^2 = K^3$. Indeed, we can write $L = K:\m = K + Rg$ with $g \in (R:\m) \setminus K$. Then $L^2 = K^2 + Kg + Rg^2 \subseteq K^2 + L$, because $B=R:\m =  \frac{\m}{\alpha}$ and $L=\frac{\m K}{\alpha}$ for some $\alpha \in \m$. Thanks to \cite[Corollary 2.4 (1)]{CGKM}, we have $L=K:\m \subseteq K^2$. Therefore $L^2 \subseteq K^2 + L =K^2$; hence $S=L^2 = K^2$. This shows $K^2 = K^3$ and $R$ is an $n$-Goto ring. 
\end{proof}



Recall that a one-dimensional Cohen-Macaulay local ring $R$ is called {\it Arf}, if the following conditions are satisfied (\cite[Definition 2.1]{L}).
\begin{enumerate}[$(1)$]
\item Every integrally closed regular ideal in $R$ has a reduction generated by a single element. 
\item If $x, y, z \in R$ such that $x$ is a non-zerodivisor on $R$ and $y/x, z/x \in \overline{R}$, then $yz/x \in R$. 
\end{enumerate}
In \cite[Theorem 2.2]{L}, it is proved that $R$ is Arf if and only if every integrally closed regular ideal is stable; equivalently, all {\it the local rings infinitely near to $R$}, i.e., the localizations of its blow-ups, have minimal multiplicity. 

We set $R^I = \bigcup_{n\geq 0}\left[I^n: I^n\right]$ for a regular ideal $I$ of $R$. 
For an integer $m \ge 0$, we define
$$
R_m = 
\begin{cases}
\  R &  (m = 0) \\
\  R_{m-1}^{J(R_{m-1})}  & (m \geq 1)
\end{cases}
$$
where $J(R_{m-1})$ stands for the Jacobson radical of the ring $R_{m-1}$.

\begin{cor}\label{4.10a}
Let $n \ge 2$ be an integer. Suppose that $R$ is an Arf ring, $\overline{R}$ is a local ring, and $R/\m \cong R_i/\m_i$ for all $i \ge 0$, where $\m_i = J(R_i)$. Then the following conditions are equivalent. 
\begin{enumerate}
\item[$(1)$] $R$ is an $n$-Goto ring. 
\item[$(2)$] $R_i$ is an $(n-i)$-Goto ring for every $1 \le i < n$.
\item[$(3)$] $R_i$ is an $(n-i)$-Goto ring for some $1 \le i < n$.
\end{enumerate}
Moreover, if $R_{n-1}$ is not a Gorenstein ring, the condition $(1)$ is equivalent to the following. 
\begin{enumerate} 
\item[$(4)$] $R_i$ is an $(n-i)$-Goto ring for every $1 \le i \le n$.
\item[$(5)$] $R_i$ is an $(n-i)$-Goto ring for some $1 \le i \le n$.
\end{enumerate}
\end{cor}

\begin{proof}
Since $\overline{R}$ is a local ring, so is the ring $R_i$ for all $i \ge 0$. Hence, by \cite[Theorem 2.2]{L}, we have $v(R_i) = \e(R_i)$. Besides, for each $1 \le i < n$, the ring $R_i$ is not Gorenstein if it is $(n-i)$-Goto. Thus, the equivalence of the conditions $(1)$, $(2)$, and $(3)$ follows by using Theorem \ref{4.9a} recursively. When $R_{n-1}$ is not a Gorenstein ring, the required equivalence follows from \cite[Theorem 5.1]{GMP}.
\end{proof}

The assumption in Corollary \ref{4.10a} that $R/\m \cong R_i/\m_i$ for all $i \ge 0$ is satisfied, if either $R/\m$ is an algebraically closed field, or $R$ is a numerical semigroup ring over a field. 


\begin{ex}
Let $e \ge 4$, $n \ge 2$ be integers. 
Let $V=k[[t]]$ be the formal power series ring over a field $k$. 
For each $0 \le j \le n$, we set $R_j = k[[H_j]]$ in $V$, where $H_j$ has the following form.
$$
H_j=
\begin{cases}
\, \left<e, \{en-(j+1)e+i\}_{\, 3 \le i \le e-1}, en-je+1, en-je+2\right> & (0 \le j \le n-2)\\
\, \left<3, 4, 5\right>& (j=n-1) \\
\ {\Bbb N} & (j=n).
\end{cases}
$$
Then $R_j = \m_{j-1}: \m_{j-1}$ has minimal multiplicity for all $1 \le j \le n$, where $\m_{j-1}$ denotes the maximal ideal of $R_{j-1}$. Hence $R_0$ is an Arf ring. By Example \ref{4.3} (2), the ring $R_0$ is $n$-Goto. Therefore, $R_j$ is an $(n-j)$-Goto ring for every $1 \le j \le n$. 
\end{ex}


For the rest of this section, under Setup \ref{4.1}, we additionally assume that $R$ is an $n$-Goto ring with $n \ge 2$ and $v(R/\fkc) = 1$. As $\ell_R(R/\fkc) =n$, we can choose a minimal system $x_1, x_2, \ldots, x_{\ell}$ of generators of $\m$ such that $\fkc =(x_1^n, x_2, \ldots, x_{\ell})$, where $\ell = v(R)$. 
For each $1 \le i \le n$, we define
$$
I_i = (x_1^i, x_2, \ldots, x_{\ell}).
$$
Then we get a chain $\fkc = I_n \subsetneq I_{n-1} \subsetneq \cdots \subsetneq I_1 =\m$ of ideals in $R$ and $v(R/I_i) = 1$ for every $2 \le i \le n$. We set $r = \rmr(R)$.




\begin{thm}\label{4.8}
With the notation of above, one has an isomorphism
$$
K/R \cong \bigoplus_{i=1}^{n} \left(R/I_i\right)^{\oplus \ell_i}
$$
of $R$-modules for some $\ell_n >0$ and for some $\ell_i \ge 0~(1 \le i < n)$ such that $\sum_{i=1}^{n} \ell_i = r-1$. Hence $K/R$ is free as an $R/\fkc$-module if and only if the equality $\ell_R(K/R) = n (r-1)$ holds. 
\end{thm}


\begin{proof}
The equality $K^2=K^3$ induces $\fkc = R:K$, so $K/R$ is a finite module over $R/\fkc$. Hence, because $R/\fkc$ is an Artinian Gorenstein ring, the $R/\fkc$-module $K/R$  decomposes to a direct sum of cyclic $R/\fkc$-modules. Note that a cyclic module over $R/\fkc$ is isomorphic to the module of the form $R/I_i$ with $1 \le i \le n$, and $\mu_R(K/R) = r-1$. 
Therefore we have an isomorphism
$
K/R \cong \bigoplus_{i=1}^{n} \left(R/I_i\right)^{\oplus \ell_i}
$
of $R$-modules for some $\ell_i \ge 0~(1 \le i \le n)$ such that $\sum_{i=1}^{n} \ell_i = r-1$. In particular, we get $\ell_n >0$, because the faithful $R/\fkc$-module $K/R$ contains $R/\fkc$ itself as a direct summand. 
\end{proof}

By \cite[Theorem 4.11]{GK}, the ring $R$ is generalized Gorenstein if and only if $K/R$ is free as an $R/\fkc$-module.

\begin{cor}
Suppose that $n \ge 2$, $R$ is $n$-Goto, and $v(R/\fkc) = 1$. Then $R$ is a generalized Gorenstein ring if and only if the equality $\ell_R(K/R) = n (r-1)$ holds.  
\end{cor}


As an application of Theorem \ref{4.8}, we get the following.  
Recall that $R$ is a Gorenstein ring if and only if $\rme_1(I) = 0$; while the ring $R$ is non-Gorenstein almost Gorenstein if and only if $\rme_1(I) = \rmr(R)$ (\cite[Theorem 3.16]{GMP}).
 
\begin{cor}\label{4.11}
Suppose that $n \ge 2$, $R$ is $n$-Goto, and $v(R/\fkc) = 1$. Then
$$
\rmr(R) + n \le \rme_1(I) \le  n \cdot \rmr(R)
$$
and the equality 
$\rme_1(I) = n \cdot \rmr(R)$ holds if and only if $R$ is a generalized Gorenstein ring. 
\end{cor}

\begin{proof}
Thanks to Theorem \ref{4.8}, we have an isomorphism $K/R \cong \bigoplus_{i=1}^{n} \left(R/I_i\right)^{\oplus \ell_i}$, where $\ell_n >0$ and $\ell_i \ge 0~(1 \le i < n)$ such that $\sum_{i=1}^{n} \ell_i = r-1$. Then 
\begin{eqnarray*}
\rme_1(I) \!\!&=&\!\! \rank \calS_Q(I) + \ell_R(K/R) = n + \sum_{i=1}^n \ell_R(R/I_i)\cdot\ell_i = n + \sum_{i=1}^n i\cdot \ell_i \\
\!\!&\le&\!\! n + n(r-1) = nr.
\end{eqnarray*}
As $R$ is $n$-Goto with $n \ge 2$, we have $r-1 = \mu_R(K/R) < \ell_R(K/R)$. Hence
$$
\rme_1(I) = \ell_R(K/R)+n \ge (\mu_R(K/R) + 1) + n = r+n.
$$
The equality $\rme_1(I) = n \cdot \rmr(R)$ holds if and only if $\ell_i=0$ for all $1 \le i \le n-1$. The latter condition is equivalent to saying that $K/R$ is free as an $R/\fkc$-module, or, equivalently, $R$ is a generalized Gorenstein ring.  
\end{proof}

Corollary \ref{4.11} holds even in higher dimensions. 


\begin{rem}
Let $(A, \m)$ be a Cohen-Macaulay local ring with $d = \dim A>0$ admitting a canonical ideal $I$. Let $Q=(a_1, a_2, \ldots, a_d)$ be a parameter ideal with condition $(\sharp)$. We set $J=I+Q$ and $\fkq = (a_2, a_3, \ldots, a_d)$. 
Suppose that $A$ is an $n$-Goto ring with $n \ge 2$ and $\m^{n-1} \calS_Q(J) \ne (0)$. Then
$$
\rmr(A) + n \le \rme_1(J) \le  n \cdot \rmr(A)
$$
hold and if the equality $\rme_1(J) =  n \cdot \rmr(A)$ holds, then $A$ is a generalized Gorenstein ring. 
\end{rem}

\begin{proof}
Let $\overline{A} = A/\fkq$, $\overline{\m}=\m \overline{A}$, $\overline{J} = J\overline{A}$, and $\overline{Q} = Q\overline{A}$. Since $a_2, a_3, \ldots, a_d$ is a super-regular sequence, the ring $\overline{A}$ is $n$-Goto, and
the condition $\m^{n-1} \calS_Q(J) \ne (0)$ implies that $\overline{\m}^{n-1} \calS_{\overline{Q}}(\overline{J}) \ne (0)$; equivalently, $v(\overline{A}/(\overline{Q}:_{\overline{A}}\overline{J}) = 1$. By Corollary \ref{4.11}, we have
$$
\rmr(A) + n \le \rme_1(J) \le  n \cdot \rmr(A)
$$
because $\rmr(A) = \rmr(\overline{A})$ and $\rme_1(J) = \rme_1(\overline{J})$ (remember that the sequence $a_2, a_3, \ldots, a_d$ is superficial). The last assertion follows from \cite[Theorem 3.9 (2)]{GK}.
\end{proof}


We close this section to describe the generators of the $R/\fkc$-module $K/R$, when $R$ is a numerical semigroup ring over a field.  

\begin{prop}\label{4.9}
Let $R = \bigoplus_{n \ge 0}R_n$ be a $\Bbb Z$-graded ring such that $R_0 =k$ is a field. Let $M$ be a non-zero finitely generated graded $R$-module with $\dim_kM_n \le 1$ for all $n \ge 0$. We set $I = (0):_RM$ and assume that $R/I$ is Gorenstein. Choose a homogeneous minimal system $x_1, x_2, \ldots, x_{\ell}$ of generators of $M$ where $\ell = \mu_R(M)$. 
Then the equality
$$
M = Rx_1 \oplus Rx_2 \oplus \cdots \oplus Rx_{\ell}
$$
holds. 
\end{prop}

\begin{proof}
The assertion is obvious when $\ell = 1$. We assume $\ell \ge 2$ and the assertion holds for $\ell -1$. Note that $I=(0):_RM = \bigcap_{i=1}^{\ell}\left[(0):_Rx_i\right]$. As $R/I$ is Gorenstein, we may assume $I=(0):_Rx_1$. Then the canonical exact sequence
$
0  \to Rx_1 \overset{i}{\to} M \overset{\varepsilon}{\to} C \to 0
$
of graded $R$-modules splits, because $Rx_1 \cong R/I$. Thus there is a graded $R$-linear map $\varphi : C \to M$ of degree $0$ such that the composite map $\varepsilon \circ \varphi$ is identity on $C$. For each $1 \le i \le n$, we can write $\varphi(\varepsilon(x_i)) = c_i x_i$ for some unit $c_i \in k$, because $\dim_kM_n \le 1$ for every $n \ge 0$. Therefore we get
\begin{eqnarray*}
M \!\!&=&\!\! Rx_1 \oplus \varphi (C) = Rx_1 \oplus \left[R\varphi(\varepsilon(x_2))+R\varphi(\varepsilon(x_3)) + \cdots + R\varphi(\varepsilon(x_{\ell})) \right] \\
\!\!&=&\!\! Rx_1 \oplus \left[R\varphi(\varepsilon(x_2))\oplus R\varphi(\varepsilon(x_3)) \oplus \cdots \oplus R\varphi(\varepsilon(x_{\ell})) \right] \\
\!\!&=&\!\! Rx_1 \oplus Rx_2 \oplus \cdots \oplus Rx_{\ell}
\end{eqnarray*}
where the third equality follows from the induction argument on $\ell$. 
\end{proof}

Hence we have the following which is a natural generalization of \cite[Proposition 2.4]{GIT}. 

\begin{cor}
Let $V=k[[t]]$ be the formal power series ring over a field $k$ and $R=k[[H]]$ the semigroup ring of a numerical semigroup $H$. We set $r=\rmr(R)$, $f=f(H)=c_r$, and write $\mathrm{PF}(H) =\{c_1, c_2, \ldots, c_r\}$. If $R/\fkc$ is Gorenstein, then the equality
$$
K/R = \bigoplus_{i=1}^{r-1}R \, \overline{t^{f-c_i}}
$$
holds, where $\overline{(-)}$ denotes the image in $K/R$. 
\end{cor}

\begin{proof}
Since $K=\sum_{i=1}^rRt^{f-c_i}$, we have $K/R = \sum_{i=1}^{r-1}R \, \overline{t^{f-c_i}}$. 
Without loss of generality, we may assume $R=k[H]$ in the polynomial ring $k[t]$. Hence, by Proposition \ref{4.9} we get the required equality.   
\end{proof}

\begin{ex}
Let $V=k[[t]]$ denote the formal power series ring over a field $k$. We consider $R = k[[H]]$ in $V$, where $H=\left<e, \{en-e+i\}_{3 \le i \le e-1}, en+1, en+2\right>~(n \ge 2,\, e \ge 4)$. Then $R$ is an $n$-Goto ring and $K=R + Rt + Rt^3 + Rt^4 + \cdots + Rt^{e-1}$ is a fractional canonical ideal of $R$, so that 
$$
K/R = R\overline{t} \oplus R\overline{t^3} \oplus R\overline{t^4} \oplus \cdots \oplus R\overline{t^{e-1}}
$$
because $R/\fkc$ is Gorenstein, where $\fkc = R:R[K]$ and $\overline{(-)}$ denotes the image in $K/R$. Since $(0):_R \overline{t} = \fkc$ and $(0):_R \overline{t^m} = \fkc + \left(t^{e(n-1)}\right) = I_{n-1}$ for all $3 \le m \le e-1$, we conclude that 
$$
K/R \cong (R/\fkc) \oplus (R/I_{n-1})^{\oplus (e-3)}.
$$
In particular, $K/R$ is not free as an $R/\fkc$-module. 
\end{ex}






\section{Examples arising from quasi-trivial extensions}\label{sec4}

In this section we investigate Goto rings obtained by quasi-trivial extensions. The aim is to provide concrete examples of Goto rings and enrich the theory.
We start by recalling the definition of quasi-trivial extensions which was recently introduced in \cite[Section 3]{GIT}.

Let $R$ denote an arbitrary commutative ring. For an ideal $J$ of $R$ and $\alpha \in R$, we set $A(\alpha)=R\oplus J$ as an additive group and define the multiplication on $A(\alpha)$ by
\begin{center}
$
(a, x) \cdot (b, y) = \left(ab, ay + xb + \alpha (xy)\right)
$ \ \ for all  \ $(a, x), (b, y) \in A(\alpha)$. 
\end{center}
Then $A(\alpha)$ is a commutative ring which we denote by 
$$
A(\alpha) = R \overset{\alpha}{\ltimes} J
$$
and call it the {\it quasi-trivial extension} of $R$ by $J$ with respect to $\alpha$. We consider $A(\alpha)$ to be an $R$-algebra via the homomorphism $\xi : R \to A(\alpha), ~a \mapsto (a, 0)$. Then $A(\alpha)$ is a ring extension of $R$, and $A(\alpha)$ is module-finite provided $J$ is finitely generated. If $\alpha = 0$, then $A(0) = R \ltimes J$ is the idealization of $J$ over $R$, introduced by M. Nagata (\cite[Page 2]{N}), and $[(0) \times J]^2=(0)$ in $A(0)$. If $\alpha = 1$, the ring $A(1)$ is called the amalgamated duplication of $R$ along $J$ (\cite{marco}), and 
$$
A(1) \cong R \times_{R/J} R, \ (a,j) \mapsto (a, a+j)
$$
the fiber product of the two copies of the canonical surjection $R \to R/J$. Hence, if $R$ is a reduced ring, then so is $A(1)$.

In this section we maintain Setup \ref{4.1}. Let $T$ be a birational module-finite extension of $R$, i.e., $T$ is an intermediate ring between $R$ and $\rmQ(R)$ which is finitely generated as an $R$-module. We assume that $K \subseteq T$ but $R \ne T$. We set $J = R:T$. Then $J=K:T$ and $K:J=T$ (\cite[Lemma 3.5 (1)]{GMP}, \cite[Definition 2.4]{HK}). Let $\alpha \in R$ and set $A(\alpha) = R \overset{\alpha}{\ltimes} J$. 
Since $J \ne R$, by \cite[Lemma 3.1]{GIT} the ring $A(\alpha)$ is Cohen-Macaulay and of dimension one, possessing the unique maximal ideal $\n = \m \times J$. 
We have the extensions of rings below
$$
A(\alpha) \subseteq T \overset{\alpha}{\ltimes} T \subseteq \overline{R} \overset{\alpha}{\ltimes} \overline{R} \subseteq \overline{A(\alpha)} \subseteq \rmQ(R) \overset{\alpha}{\ltimes} \rmQ(R)= \rmQ(A(\alpha)).
$$
Set $L = T \times K$ in $T \overset{\alpha}{\ltimes} T$. Then $L$ is an $A(\alpha)$-submodule of $\rmQ(A(\alpha))$ and  $A(\alpha) \subseteq L \subseteq \overline{A(\alpha)}$. 

We summarize some basic properties of quasi-trivial extensions, where $\rmr_R(-)$ denotes the Cohen-Macaulay type as an $R$-module. 

\begin{fact}[{\cite[Propositions 3.3, 3.4]{GIT}}]\label{5.1}
The following assertions hold true. 
\begin{enumerate}
\item[$(1)$] $T/K$ is a canonical module of $R/J$. Hence the equality $\ell_R(T/K)= \ell_R(R/J)$ holds. 
\item[$(2)$] $L$ is a fractional canonical ideal of $A(\alpha)$ and $L^m = T \overset{\alpha}{\ltimes} T$ for all $m \ge 2$. 
\item[$(3)$] The equalities $\rmr(A(\alpha))= \mu_R(T)+\rmr(R)= \rmr_R(J) + \mu_R(K/J)$ hold. Hence the Cohen-Macaulay type of $A(\alpha)$ is independent of the choice of $\alpha \in R$. 
\end{enumerate}
\end{fact}


The next naturally generalizes \cite[Theorem 3.6]{GIT}. 

\begin{thm}\label{5.2}
Let $n \ge 1$ be an integer. Then the following conditions are equivalent. 
\begin{enumerate}
\item[$(1)$] $A(\alpha)$ is an $n$-Goto ring for every $\alpha \in R$.
\item[$(2)$] $A(\alpha)$ is an $n$-Goto ring for some $\alpha \in R$.
\item[$(3)$] The fiber product $R\times_{R/J}R$ is an $n$-Goto ring.
\item[$(4)$] The idealization $R\ltimes J$ is an $n$-Goto ring.
\item[$(5)$] $\ell_R(T/K) = n$.
\item[$(6)$] $\ell_R(R/J) = n$.
\end{enumerate} 
\end{thm}

\begin{proof}
Fact \ref{5.1} (1) guarantees the equivalence of the conditions $(5)$ and $(6)$. 
Since $L^m = T \overset{\alpha}{\ltimes} T$ for all $m \ge 2$, we have $L^2 = L^3$. Then, for each $\alpha \in R$, the ring $A(\alpha)$ is $n$-Goto if and only if $\ell_{A(\alpha)}(L^2/L)=n$; equivalently $\ell_R(T/K)=n$, because $L^2/L \cong T/K$ and $A(\alpha)/\n \cong R/\m$ as $R$-modules. This completes the proof. 
\end{proof}

\begin{cor}
Let $n \ge 1$ be an integer. If $R$ and $A(\alpha)$ are $n$-Goto rings for some $\alpha \in R$, then $S=T$ and $J=\fkc$. 
\end{cor}

\begin{proof}
As $\ell_R(S/K) = \ell_R(R/\fkc)$ (\cite[Bemerkung 2.5]{HK}), we have $\ell_R(T/K) = \ell_R(S/K) = n$. Hence the assertion follows from $S=R[K] \subseteq T$. 
\end{proof}

The condition $\mu_R(K^2/K) = 1$ holds if $R/\fkc$ is Gorenstein (Lemma \ref{4.5}), so the following gives a generalization of \cite[Theorem 6.5]{GMP}, \cite[Theorem 4.2]{CGKM}, \cite[Corollary 3.8]{GIT}, and \cite[Corollary 4.38]{GK} as well.

\begin{cor}\label{5.4}
Let $n \ge 1$ be an integer and set $L=S \times K$. Then the following conditions are equivalent. 
\begin{enumerate}
\item[$(1)$] $R$ is an $n$-Goto ring and $\mu_R(K^2/K) = 1$. 
\item[$(2)$] The fiber product $A=R\times_{R/\fkc} R$ is an $n$-Goto ring and $\mu_A(L^2/L) = 1$. 
\item[$(3)$] The idealization $A=R\ltimes \fkc$ is an $n$-Goto ring and $\mu_A(L^2/L) = 1$. 
\end{enumerate} 
\end{cor}

\begin{proof}
$(1) \Leftrightarrow (3)$ We set $A=R\ltimes \fkc$ and $\fkc_A=A:A[L]$. By \cite[Proposition 4.1]{CGKM}, we have $L^2 = L^3$, $A[L]=S \ltimes S$, and $\fkc_A = \fkc \times \fkc$. Hence $A/\fkc_A \cong R/\fkc$ as an $R$-module and
$\ell_A(L^2/L)=\ell_A(A[L]/L)=\ell_R(S/K)$. 
Therefore, if $R$ is $n$-Goto and $\mu_R(K^2/K) = 1$, then $R/\fkc \cong A/\fkc_A$ is Gorenstein and $\ell_A(L^2/L)=\ell_R(S/K) = n$. Hence $A$ is an $n$-Goto ring and $\mu_A(L^2/L) = 1$. Conversely, we assume $A$ is $n$-Goto and $\mu_A(L^2/L) = 1$. Then, because $A/\fkc_A\cong R/\fkc$ is Gorenstein, we get $K^2=K^3$ and $\mu_R(K^2/K) = 1$. By Theorem \ref{5.2}, we conclude that $R$ is an $n$-Goto ring because $\ell_R(K^2/K) = \ell_R(S/K) =n$.

$(2) \Leftrightarrow (3)$ This follows from Theorem \ref{5.2} by choosing $T=S=R[K]$.
\end{proof}

By the proof of Corollary \ref{5.4}, if $R$ is an $n$-Goto ring, then so is the idealization $A=R\ltimes \fkc$, but the converse does not hold when $n \ge 3$.

\begin{ex}
Let $V=k[[t]]$ be the formal power series ring over a field $k$. 
We consider the semigroup ring $R = k[[H]]$ in $V$, where $H=\left<4, 13, 22, 27\right>$. Then $K=R + Rt^5 + Rt^{14}$ is the fractional canonical ideal of $R$, so that $\mu_R(K^2/K) = 2$, $K^2 \ne K^3$, and $K^3 =R[K] = S$. Since $\ell_A(A[L]/L) = 4$ and $L^2=L^3$, the ring $A=R\ltimes \fkc$ is $4$-Goto, but $R$ is not. 
\end{ex}


\begin{ex}[{cf. \cite[Example 4.3]{CGKM}}]\label{5.6a}
Let $n \ge 1$ be an integer. Suppose that $R$ is $n$-Goto and $\mu_R(K^2/K) = 1$ (see e.g., Example \ref{4.3}). For each $\ell \ge 0$, we define recursively
$$
A_{\ell} = 
\begin{cases}
R & (\ell =0) \\
A_{\ell-1} \ltimes \fkc_{\ell-1} & (\ell \ge 1)
\end{cases}
$$
where $\fkc_{\ell-1} = A_{\ell-1} :A_{\ell-1}[K_{\ell-1}]$ and $K_{\ell-1}$ is the fractional canonical ideal of $A_{\ell-1}$. 
Hence we have an infinite family $\{A_{\ell}\}_{\ell \ge 0}$ of $n$-Goto rings with $\mu_{A_{\ell}}(K_{\ell}^2/K_{\ell}) = 1$ and $\rme(A_{\ell}) = 2^{\ell}\cdot \rme(R)$ for every $\ell \ge 0$.  The ring $k[[t^3, t^{3n+1}, t^{3n+2}]] \ltimes k[[t]]$ is $n$-Goto, because $\fkc = R:k[[t]] = t^{3n} k[[t]] \cong k[[t]]$, where $k[[t]]$ denotes the formal power series ring over a field $k$.
\end{ex}



We note an example of Goto rings obtained from quasi-trivial extensions in the case where $J=R$ and $\alpha \in \m$. 

\begin{ex}
For each $\alpha \in \m$, we set $R_1 = R\overset{\alpha}{\ltimes} R$.  Then, by \cite[Lemma 3.1]{GIT} $R_1$ is a Cohen-Macaulay local ring with $\dim R_1=1$ possessing the maximal ideal $\m_1=\m \times R$. Note that $R_1$ is a finitely generated free $R$-module of rank $2$. Since $\m R_1 = \m \times \m$ and $\ell_{R_1}(R_1/\m R_1) = \ell_R((R\times R)/(\m \times \m)) = 2$, the ring $R_1=R\overset{\alpha}{\ltimes} R$ is $2$-Goto if and only if $R$ is almost Gorenstein but not a Gorenstein ring.
\end{ex}


\section{Examples arising from fiber products}\label{sec6}



Let $(R, \m), (S, \n)$ be one-dimensional Cohen-Macaulay local rings with a common residue class field $k = R/\m = S/\n$, and $f : R \to k$, $g : S \to k$ the canonical surjective maps. In this section we study the question of when the fiber product 
$$
A = R\times_k S = \{(a, b) \in R \times S \mid f(a) = g(b)\}
$$ 
of $R$ and $S$ over $k$ with respect to $f$ and $g$ is a Goto ring. Then $A$ is a one-dimensional Cohen-Macaulay local ring with maximal ideal $J = \m \times \n$. Since $B=R\times S$ is a module-finite birational extension of $A$, we get $\rmQ(A) = \rmQ(R) \times \rmQ(S)$ and $\overline{A} = \overline{R} \times \overline{S}$, where $\overline{(-)}$ denotes the integral closure in its total ring of fractions. 
We furthermore assume that $\rmQ(A)$ is a Gorenstein ring, $A$ admits a canonical module $\rmK_A$, and the field $k=A/J$ is infinite. Hence, all the rings $A, R$, and $S$ possess fractional canonical ideals (see \cite[Corollary 2.9]{GMP}).

By \cite{Ogoma} (see also \cite[Proposition 2.2 (3)]{EGI}), $A$ is Gorenstein if and only if $R$ and $S$ are discrete valuation rings (abbr. DVRs). Besides,  $A$ is almost Gorenstein if and only if so are $R$ and $S$; equivalently, $A$ is a generalized Gorenstein ring (\cite[Theorem 4.17]{EGI}).  

The main result of this section is stated as follows, which gives a generalization of \cite[Theorem 5.1]{EGI}.

\begin{thm}\label{6.1}
Let $n \ge 2$ be an integer. Then the following conditions are equivalent. 
\begin{enumerate}
\item[$(1)$] The fiber product $A = R\times_k S$ is an $n$-Goto ring.
\item[$(2)$] One of the following conditions holds.
\begin{itemize}
\item[$\rm (i)$] $R$ is Gorenstein and $S$ is $n$-Goto.
\item[$\rm (ii)$] $R$ is $n$-Goto and  $S$ is Gorenstein. 
\item[$\rm (iii)$] $R$ is $p$-Goto and  $S$ is $q$-Goto for some integers $p, q >0$ such that $n+1 = p+q$. 
\end{itemize}
\end{enumerate}
\end{thm}

\begin{proof}
We denote by $K$ and $L$ the fractional canonical ideals of $R$ and $S$, respectively. 
We may assume that $A$ is not a Gorenstein ring. Hence, either $R$ or $S$ is not a DVR. 
We first consider the case where $R$ is a DVR and $S$ is not a DVR.
Let $X$ be a fractional canonical ideal of $A$.
Then, \cite[Corollary 4.15]{GK} we have $X^2 = X^3$ if and only if $L^2 = L^3$; while the equality $\ell_A(A/(A:X)) = \ell_S(S/(S:L))$ holds once we get this equivalence. Hence $A$ is $n$-Goto if and only if $S$ is $n$-Goto, as claimed. 
It remains to consider the case where $R$ and $S$ are not DVRs. 
Then $K \ne \overline{R}$, because $K:K = R$ (\cite[Bemerkung 2.5]{HK}). Since $\ell_R([K:\m]/K) = \ell_R(\Ext^1_R(R/\m, K)) =1$, we see that $K:\m \subseteq \overline{R}$. In addition, we have $R : \m \not\subseteq K$. Indeed, if $R : \m \subseteq K$, then $K \supseteq R:\m = (K:K):\m = K:\m K$.
 Thus $R = K:K \subseteq K:( K:\m K) = \m K \subseteq \m \overline{R}$. This makes a contradiction.  Hence $R : \m \not\subseteq K$ and therefore
$
K:\m = K + R\cdot g_1
$
for some $g_1 \in (R:\m) \setminus K$. Similarly, because $S$ is not a DVR, we can choose $g_2 \in (S:\n) \setminus L$ such that $L:\n = L + S \cdot g_2$. We set 
$$
X = (K\times L) + A \cdot \psi
$$
with $\psi = (g_1, g_2) \in  \overline{A}$. 
Then $X$ is a fractional canonical ideal of $A$ (\cite[Lemma 4.2]{EGI}). 
With this notation we  have the following. 

\begin{claim}\label{6.2}
$X^2 = X^3$ if and only if $K^2 = K^3$ and $L^2 = L^3$. 
\end{claim}

\begin{proof}[Proof of Claim \ref{6.2}]
We first assume that $R$ and $S$ are not Gorenstein. Then, because $A:X = (R:K) \times (S:L)$ (\cite[Lemma 4.8 (i)]{EGI}), we then have 
\begin{eqnarray*}
(A:X)X \!\!&=&\!\! \left[(R:K) \times (S:L)\right]\left[(K\times L) + A \cdot \psi\right] \\
\!\!&=&\!\! \left[(R:K)K + (R:K)g_1\right] \times \left[(S:L)L + (S:L)g_2\right] \\
\!\!&=&\!\! (R:K)K \times (S:L)L
\end{eqnarray*}
where the last equality follows from the fact that $(R:K)(R:\m) = R:K$ and $(S:L)(S:\n) = (S:L)$ (\cite[Lemma 4.14 (i)]{EGI}). Therefore, by \cite[Lemma 4.14 (ii)]{EGI} we get the required equivalence, because $(A:X)X = A:X$ if and only if $(R:K)K = R:K$ and $(S:L)L = S:L$. Next, we consider the case where $R$ is Gorenstein but $S$ is not a Gorenstein ring. Then $R=K$. Thanks to \cite[Lemma 4.8 (ii)]{EGI}, we have $A:X = \m \times (S:L)$ which yields that the equalities
\begin{eqnarray*}
(A:X)X \!\!&=&\!\! \left[\m \times (S:L)\right]\left[(K\times L) + A \cdot \psi\right] \\
\!\!&=&\!\! \left[\m + \m g_1\right] \times \left[ (S:L)L + (S:L)g_2 \right] \\
\!\!&=&\!\! \m(R:\m) \times (S:L)L = \m \times  (S:L)L
\end{eqnarray*}
hold, because $\m (R:\m) = \m$ (remember that $R$ is not a DVR). This shows that $(A:X)X = A:X$ if and only if $(S:L)L = S:L$; hence $X^2 = X^3$ if and only if $L^2 = L^3$.  
\end{proof}
We continue to prove Theorem \ref{6.1}. Since $n \ge 2$, we may assume that either $R$ or $S$ is not a Gorenstein ring. Suppose that both $R$ and $S$ are not Gorenstein. Since $A:X = (R:K) \times (S:L)$, we have the equalities
$$
\ell_A(A/[A:X]) = \ell_A(B/[A:X]) - 1 = \ell_R(R/[R:K]) + \ell_S(S/[S:L]) -1. 
$$
Thanks to Claim \ref{6.2}, the condition $(1)$ is equivalent to the condition $(2)$ $\rm (iii)$, because $\ell_A(A/[A:X]) = n$ if and only if there exist integers $p, q >0$ such that $n+1 = p+q$, $\ell_R(R/[R:K])=p$, and $\ell_S(S/[S:L]) = q$. 
We may finally assume that $R$ is Gorenstein but $S$ is not a Gorenstein ring. Then $A:X = \m \times (S:L)$, so that
$$
\ell_A(A/[A:X]) = \ell_R(R/\m) + \ell_S(S/[S:L]) -1 = \ell_S(S/[S:L]).
$$
This yields that the ring $A$ is $n$-Goto if and only if so is $S$, as desired.
\end{proof}

As a direct consequence of Theorem \ref{6.1}, if $R$ is $n$-Goto and $S$ is $2$-Goto, the fiber product $A= R\times_k S$ is an $(n+1)$-Goto ring which is not a generalized Gorenstein ring. 

\begin{ex}\label{6.3a}
Let $k[[t]]$ be the formal power series ring over a field $k$. Since the semigroup ring $k[[t^3, t^{3n+1}, t^{3n+2}]]$ is an $n$-Goto ring (see Example \ref{4.3}), the fiber product $k[[t^3, t^{3n+1}, t^{3n+2}]] \times_k k[[t]]$ is $n$-Goto; while 
the ring $k[[t^3, t^{3n+1}, t^{3n+2}]] \times_k k[[t^3, t^7, t^8]]$ is an $(n+1)$-Goto ring, where $n \ge 2$ is an integer. 
\end{ex}




\section{Three-generated numerical semigroup rings}\label{sec7-1}

We investigate Goto rings arising from numerical semigroups generated by three elements. 
We first recall the results in \cite[Section 4]{GMP} (see also \cite[Section 6]{CGKM}). 

Let $0<a_1, a_2, a_3 \in \bbZ$ be integers such that $\gcd(a_1, a_2, a_3) = 1$. We consider the numerical semigroup $H=\left<a_1, a_2, a_3\right>$ which is minimally generated by $a_1, a_2, a_3$.  
Let $k[t]$ denote the polynomial ring over a field $k$ and set $T=k[H] = k[t^{a_1}, t^{a_2}, t^{a_3}]$. Then $T$ is a one-dimensional Cohen-Macaulay graded domain and $\overline{T} = k[t]$, where $\overline{T}$ denotes the normalization of $T$. Let $M=(t^{a_1}, t^{a_2}, t^{a_3})T$ be the maximal ideal of $T$. The ring $R =k[[t^{a_1}, t^{a_2}, t^{a_3}]]$ is obtained by the completion $R=\widehat{T_M}$ of the local ring $T_M$. 

Let $U=k[X, Y, Z]$ be the polynomial ring and we regard $U$ as a $\Bbb Z$-graded ring with $U_0=k$, $\deg X=a_1$, $\deg Y=a_2$, and $\deg Z = a_3$. Look at the $k$-algebra map 
$$
\varphi : U=k[X, Y, Z] \to T=k[t^{a_1}, t^{a_2}, t^{a_3}]
$$
defined by $\varphi(X) = t^{a_1}$, $\varphi(Y) = t^{a_2}$, and $\varphi(Z) = t^{a_3}$. 

Throughout this section, we assume that $T$ is not a Gorenstein ring. By \cite{H} the defining ideal $\Ker \varphi$ of $T$ is generated by the maximal minors of the matrix 
$\left[\begin{smallmatrix}
X^{\alpha} & Y^{\beta} & Z^{\gamma} \\
Y^{\beta'} & Z^{\gamma'} & X^{\alpha'}
\end{smallmatrix}\right]$, i.e.,  
$$
\Ker \varphi = {\rm I}_2
\begin{pmatrix}
X^{\alpha} & Y^{\beta} & Z^{\gamma} \\
Y^{\beta'} & Z^{\gamma'} & X^{\alpha'}
\end{pmatrix}
$$
for some $0<\alpha, \beta, \gamma, \alpha', \beta', \gamma' \in \Bbb Z$. Let $\Delta_1 = Z^{\gamma + \gamma'}-X^{\alpha'}Y^{\beta}$, $\Delta_2 = X^{\alpha + \alpha'} - Y^{\beta'}Z^{\gamma}$, and $\Delta_3 = Y^{\beta + \beta'} - X^{\alpha}Z^{\gamma'}$. Then $\Ker \varphi = (\Delta_1, \Delta_2, \Delta_3)$ and the ring $T\cong U/\Ker \varphi$ admits a graded minimal $U$-free resolution of the form
\begin{equation*}
0 \longrightarrow
\begin{matrix}
U\left(-m\right) \\
\oplus \\
U\left(-m'\right) 
\end{matrix}
\overset{\ 
\left[\begin{smallmatrix}
X^{\alpha} & Y^{\beta'} \\
Y^{\beta} & Z^{\gamma'} \\
Z^{\gamma} & X^{\alpha'} \\[2pt]
\end{smallmatrix}\right] \ 
}{\longrightarrow}
\begin{matrix}
U\left(-d_1\right) \\
\oplus \\
U\left(-d_2\right) \\
\oplus \\
U\left(-d_3\right) 
\end{matrix}
\overset{[\Delta_1 \ \Delta_2 \ \Delta_3]}{\longrightarrow}
U \overset{\varepsilon}{\longrightarrow} T \longrightarrow 0
\end{equation*}
where $d_1 = \deg \Delta_1 = a_3(\gamma + \gamma')$, $d_2 = \deg \Delta_2 = a_1(\alpha + \alpha')$, $d_3 = \deg \Delta_3 = a_2(\beta + \beta')$, $m = a_1 \alpha + d_1 = a_2 \beta + d_2 = a_3 \gamma + d_3$, and $m' = a_1\alpha' + d_3 = a_2 \beta' + d_1 = a_3 \gamma' + d_2$. Hence $m'-m = a_2 \beta' -a_1 \alpha = a_3 \gamma' -a_2 \beta = a_1 \alpha' -a_3 \gamma$. 

Let $\rmK_U=U(-d)$ be the graded canonical module of $U$ where $d = a_1 + a_2 + a_3$. Taking $\rmK_U$-dual, we get the graded minimal $U$-free presentation
 \begin{equation*}
\begin{matrix}
U\left(d_1-d\right) \\
\oplus \\
U\left(d_2-d\right) \\
\oplus \\
U\left(d_3-d\right) 
\end{matrix}
\overset{\ 
\left[\begin{smallmatrix}
X^{\alpha} & Y^{\beta} & Z^{\gamma} \\
Y^{\beta'} & Z^{\gamma'} & X^{\alpha'}\\[2pt]
\end{smallmatrix}\right] 
}{\longrightarrow}
\begin{matrix}
U\left(m-d\right)  \\
\oplus \\
U\left(m'-d\right) 
\end{matrix} \overset{\varepsilon}{\longrightarrow} \rmK_T \longrightarrow 0
\end{equation*}
of the graded canonical module $\rmK_T$. Remember that $\rmK_T = \sum_{c \in \mathrm{PF}(H)}Tt^{-c}$ as a $T$-module; see \cite[Example (2.1.9)]{GW}. Thus
$\dim_k\left([\rmK_T]_i\right) \le 1$ for all $i \in \Bbb Z$ and $\mathrm{PF}(H) =\{m-d, m'-d\}$. Therefore we conclude that $m \ne m'$. 

Let $b = |m-m'|$ be the absolute value of $m-m'$. We then have the following. 

\begin{thm}\label{7.1}
Let $n \ge 1$ be an integer. Suppose that $R$ is not a Gorenstein ring. 
Then the following conditions are equivalent. 
\begin{enumerate}
\item[$(1)$] The semigroup ring $R=k[[H]]$ is an $n$-Goto ring. 
\item[$(2)$] $3b \in H$, and $n = \alpha \beta \gamma$ $($resp. $n = \alpha' \beta' \gamma'$$)$ if $m' > m$ $($resp. $m > m'$$)$.
\end{enumerate} 
\end{thm}

\begin{proof}
After a suitable permutation of $a_2$ and $a_3$ if necessary, without loss of generality we may assume $m' > m$. Note that $K=R + Rt^b$ is a fractional canonical ideal of $R$. As $R$ is the $MT_M$-adic completion of the local ring $T_M$ where $M=(t^{a_1}, t^{a_2}, t^{a_3})$ is the graded maximal ideal of $T$, we see that $\ell_R(K/R) = \alpha\beta\gamma$ (\cite[Theorem 4.1]{GMP}). Besides, the equality $K^2 = K^3$ holds if and only if $t^{3b} \in K^2$. Equivalently, $3b \in H$, because $b, 2b \not\in H$. Indeed, since $\mu_R(K)=2$, we have $b \notin H$. If $2b \in H$, then $K^2 = K + Rt^{2b} = K$. This shows $K \subseteq K:K = R$, and hence $R=K$. This is impossible, because $R$ is not a Gorenstein ring. Consequently, $R$ is an $n$-Goto ring if and only if $3b \in H$ and $n = \alpha\beta\gamma$. 
\end{proof}



We note an example. 

\begin{ex}
Let $k[[t]]$ be the formal power series ring over a field $k$. We set $R = k[[H]]$ in $V$, where $H=\left<7, 10, 22\right>$. Then $\mathrm{PF}(H) =\{25, 33\}$ and $b = |m-m'| = 8$; hence $3b \in H$, $K = R + Rt^8$ is a fractional canonical ideal of $R$, and $K^2 = K^3$. 
The $k$-algebra map $\varphi : k[[X, Y, Z]] \to R$ defined by $\varphi(X) = t^{10}$, $\varphi(Y) = t^7$, and $\varphi(Z) = t^{22}$ induces the isomorphism
$$
R \cong k[[X, Y, Z]]/
\rmI_2
\begin{pmatrix}
X^2 & Y^2 & Z \\
Y^4 & Z & X^3
\end{pmatrix}
$$
where $k[[X, Y, Z]]$ denotes the formal power series ring over $k$. 
Then $m'-m = a_2 \beta' -a_1 \alpha = 7 \cdot 4 - 10 \cdot 2 = 8$. This
yields that $\ell_R(K/R) = 4$. Hence $R$ is a $4$-Goto ring.  
\end{ex}

The following generalizes the result \cite[Corollary 6.7 (1)]{CGKM}.


\begin{cor}\label{8.1}
Let $n \ge 1$ be an integer. Suppose that $\rme(R) = 3$ and $R$ has minimal multiplicity. Then the following conditions are equivalent. 
\begin{enumerate}
\item[$(1)$] $R=k[[H]]$ is an $n$-Goto ring. 
\item[$(2)$] $H=\left< 3, 2n+\alpha, n + 2\alpha \right>$ for some $\alpha \ge n+1$ such that $\alpha \not \equiv n$ \, \!\!\!\! $\mod \,\, 3$. 
\end{enumerate} 
When this is the case, one has the isomorphism either
$$
R \cong k[[X, Y, Z]]/{\rm I}_2
\begin{pmatrix}
X^{n} & Y & Z \\
Y & Z & X^{\alpha}
\end{pmatrix}, \  \ \text{or}  \ \ \ 
R \cong k[[X, Y, Z]]/{\rm I}_2
\begin{pmatrix}
X^{\alpha} & Y & Z \\
Y & Z & X^n
\end{pmatrix}
$$
where $k[[X, Y, Z]]$ denotes the formal power series ring over $k$. 
\end{cor}

\begin{proof}
As $R$ has minimal multiplcity, we have $\rmr(R) = \rme(R)-1 =2$, so $R$ is not Gorenstein.

$(1) \Rightarrow (2)$ 
Let us write $H=\left<3, a_2, a_3\right>$ where $0<a_2, a_3 \in \Bbb Z$ such that $\gcd(3, a_2, a_3) = 1$, and assume that $H$ is minimally generated by three elements. Consider the ring
$$
T=k[t^{3}, t^{a_2}, t^{a_3}] \cong k[X, Y, Z]/
 {\rm I}_2
\begin{pmatrix}
X^{\alpha} & Y^{\beta} & Z^{\gamma} \\
Y^{\beta'} & Z^{\gamma'} & X^{\alpha'}
\end{pmatrix}
$$
with $0<\alpha, \beta, \gamma, \alpha', \beta', \gamma' \in \Bbb Z$. 
We then have 
$$
3 =  \ell_R(R/t^3R) = \ell_k\left(k[Y, Z]/(Y^{\beta + \beta'}, Y^{\beta'}Z^{\gamma}, Z^{\gamma+\gamma'})\right) = \beta\gamma + \beta'\gamma + \beta\gamma' \ge 3.
$$
This yields that $\beta = \beta' = \gamma = \gamma' = 1$. Similarly we get
\begin{eqnarray*}
a_2 \!\!\!&=&\!\!\! \ell_R(R/t^{a_2}R) = \ell_k\left(k[X, Z]/(X^{\alpha + \alpha'}, X^{\alpha}Z^{\gamma'}, Z^{\gamma+\gamma'})\right) = \alpha \gamma + \alpha \gamma' + \alpha' \gamma' = 2\alpha + \alpha'\\
a_3 \!\!\!&=&\!\!\! \ell_R(R/t^{a_3}R) = \ell_k\left(k[X, Y]/(X^{\alpha + \alpha'}, X^{\alpha'}Y^{\beta}, Y^{\beta+\beta'})\right) = \alpha \beta + \alpha' \beta + \alpha' \beta' = \alpha + 2 \alpha'.
\end{eqnarray*}
If $m' > m$, then $n = \alpha$ by Theorem \ref{7.1}. Thus $a_2 = 2n + \alpha'$ and $a_3 = n + 2\alpha'$. Therefore, $\alpha' - n = a_2 \beta' - 3 \alpha =m'-m \ge 1$. As $H$ is minimally generated by three elements, we see that $\alpha' \not \equiv n$ \, \!\!\!\! $\mod \,\, 3$. We similarly have that, if $m > m'$, then $n=\alpha'$; hence $a_2 = 2\alpha + n$, $a_3 = \alpha + 2n$, and $\alpha - n = m-m' \ge 1$. We also have $\alpha \not \equiv n$ \, \!\!\!\! $\mod \,\, 3$.

$(2) \Rightarrow (1)$ Since $\alpha \not \equiv n$ \, \!\!\!\! $\mod \,\, 3$,  the complement of $H$ in $\Bbb N$ is finite and $H$ is minimally generated by three elements. Thus $\gcd(3, 2n+\alpha, n+2\alpha)=1$. Thanks to Corollary \ref{4.7a}, it suffices to show that $\ell_R(K/R)=n$. Indeed, the $k$-algebra map $\varphi : k[[X, Y, Z]] \to R$
defined by $\varphi(X) = t^{3}$, $\varphi(Y) = t^{2n+\alpha}$, and $\varphi(Z) = t^{n+2\alpha}$ induces the isomorphism
$$
R \cong k[[X, Y, Z]]/
\rmI_2
\begin{pmatrix}
X^n & Y & Z \\
Y & Z & X^{\alpha}
\end{pmatrix}
$$
of rings. Then, because $(2n + \alpha) - 3 n = \alpha - n > 0$, we get $\ell_R(K/R) = n$. Therefore $R$ is an $n$-Goto ring. 
\end{proof}


\section{Minimal free presentations of fractional canonical ideals}\label{sec7}

Throughout this section, unless otherwise specified, we maintain the setup below.

\begin{setup}\label{10.1}
Let $(T, \n)$ be a regular local ring with $\ell = \dim T \ge 3$, $\fka \subsetneq T$ an ideal of $T$. Let $n \ge 2$ be an integer. 
Set $R=T/\fka$ and $\m = \n/\fka$. Suppose that $R$ is a Cohen-Macaulay local ring with $\dim R=1$, but not a Gorenstein ring. Let $K$ be a fractional canonical ideal of $R$. We set $I = R:K$, $\fkc = R:R[K]$, and $r=\rmr(R) \ge 2$. 
\end{setup}

We first assume the existence of the isomorphism
$$
K/R \cong \bigoplus_{i=1}^n (R/I_i)^{\oplus \ell_i}
$$
of $R$-modules for some $\ell_n>0$, $\ell_i \ge 0$ ($1 \le i < n$) such that $\sum_{i=1}^n \ell_i = r-1$, where, for each $1 \le i \le n$, the ideal $I_i$ is generated by an $R$-regular sequence $x_1^{(i)}, x_2^{(i)}, \ldots, x_{m_i}^{(i)}$ and $m_i = \mu_R(I_i)$. Without loss of generality, we may assume $I_n = I$. Then $I_n \subseteq I_i$ for each $1 \le i \le n$. 
We choose $X_j^{(i)} \in \n$ so that $x_j^{(i)} = \overline{X_j^{(i)}}$ in $R$, where $1 \le i \le n$ and $1 \le j \le m_i$. Here, for each $a \in T$, we denote by $\overline{a}$ the image of $a$ in $R$. 
By setting 
$$
J_i =(X_1^{(i)}, X_2^{(i)}, \ldots, X_{m_i}^{(i)}) \subseteq \n \ \ \text{for each} \ \ 1 \le i \le n,
$$
we then have an surjective map $T/J_i \to R/I_i$. We furthermore assume the map $T/J_i \to R/I_i$ is an isomorphism. Thus $\fka \subseteq J_i$ for each $1 \le i \le n$, and $X_1^{(i)}, X_2^{(i)}, \ldots, X_{m_i}^{(i)}$ forms a regular sequence on $T$. 
We now choose generators $f_{ij} \in K$ of $K$ satisfying
$$
K=R + \sum_{i=1}^n\sum_{j=1}^{\ell_i}Rf_{ij} \ \ \ \text{and} \ \ \ (R/I_i)^{\oplus \ell_i} \cong \sum_{j=1}^{\ell_i} (R/I) {\cdot} \overline{f_{ij}} \ \ \text{for each} \ \ 1 \le i \le n
$$
where $\overline{f_{ij}}$ denotes the image of $f_{ij}$ in $K/R$.

With this notation we have the following. 

\begin{thm}\label{MFR1}
Suppose that $IK = I$. Then the $T$-module $K$ has a minimal free presentation of the form
$
F_1 \overset{\Bbb M}{\longrightarrow} F_0 \overset{\Bbb N}{\longrightarrow} K \to 0,
$
where
$$
\Bbb N= \left[
\begin{smallmatrix}
-1 & f_{n1} \cdots f_{n \ell_n} & f_{n-1,1} \cdots f_{n-1, \ell_{n-1}} & \cdots & f_{11} \cdots f_{1\ell_1}
\end{smallmatrix}\right]
$$
and 
$$
\! \Bbb M = \!\! \left[ 
\begin{smallmatrix}
a_{11}^{(n)} \cdots a_{1m_n}^{(n)} & \cdots  & a_{\ell_n1}^{(n)} \cdots a_{\ell_n m_n}^{(n)} & a_{11}^{(n-1)} \cdots a_{1 m_{n-1}}^{(n-1)}&  & \cdots &   &  \hspace{-0.5em}a_{11}^{(1)}  \cdots a_{1m_1}^{(1)} & \hspace{-0.5em} \cdots & \hspace{-0.5em}a_{\ell_11}^{(1)} \cdots a_{\ell_1 m_1}^{(1)}  &  c_1 \cdots c_q \\\vspace{-0.5em} 
\hspace{-0.2em} X^{(n)}_1  \cdots X^{(n)}_{m_n} &  &  &  & &    & & &  & & 0 \\
 & \ddots &  &   &  &   &  & & & &  0 \\ \vspace{-0.5em}
 &  & \hspace{-0.5em} X^{(n)}_1  \cdots X^{(n)}_{m_n} &  &  &   \\ \vspace{-0.5em}
 &  &  & \hspace{-0.5em} X^{(n-1)}_1 \cdots X^{(n-1)}_{m_{n-1}} &   &  & & &  &   &  \vdots\\ 
 &  &  &  & \hspace{-1.5em} \ddots &  &  \\ \vspace{-0.5em}
 &  &  &  &  & \hspace{-1em} X^{(n-1)}_1  \cdots X^{(n-1)}_{m_{n-1}} &   \\
 &  &  &  &  &  & \ddots &  &  & \\
 &  &  &  &  &  &  & X_1^{(1)}  \cdots X_{m_1}^{(1)} &  &  &   \\
 &  &  &  &  &  &  &  & \ddots &  & 0\\
 &  &  &  &  &  &  &  &  & X_1^{(1)}  \cdots X_{m_1}^{(1)} & 0\\
\end{smallmatrix}\right]
$$
with $a_{ij}^{(n)} \in J_n~(1 \le i \le \ell_n, \, 1 \le j \le m_n)$, $a_{ij}^{(k)} \in J_n~(1\le k < n, \, 1 \le i \le \ell_k, \, 2 \le j \le m_k)$, $a_{i1}^{(k)}\in \n ~(1 \le k < n, \, 1 \le i \le \ell_k)$, and $c_q \in \n~(q \ge 0)$. 
Moreover, one has the equality
$$
\fka = \sum_{i=1}^{n}\sum_{j=1}^{\ell_i} \rmI_2
\begin{pmatrix}
a_{j1}^{(i)} & a_{j2}^{(i)} & \cdots & a_{j m_i}^{(i)} \\[4pt]
X_1^{(i)} & X_2^{(i)} & \cdots  & X_{m_i}^{(i)}
\end{pmatrix} + (c_1, c_2, \ldots, c_q).
$$
\end{thm}

\begin{proof}
Let 
$
F_1 \overset{\Bbb A}{\longrightarrow} F_0 \overset{\Bbb N}{\longrightarrow} K \longrightarrow 0
$
be a part of a minimal $T$-free resolution of $K$ with $F_0 = T \oplus T^{\oplus \ell_n} \oplus T^{\oplus \ell_{n-1}} \oplus \cdots \oplus T^{\oplus \ell_1}$, where 
$\Bbb N= \left[
\begin{smallmatrix}
-1 & f_{n1} \cdots f_{n \ell_n} & f_{n-1,1} \cdots f_{n-1, \ell_{n-1}} & \cdots & f_{11} \cdots f_{1\ell_1}
\end{smallmatrix}\right]$. This induces a presentation of $K/R$
$$
F_1 \overset{\Bbb A'}{\longrightarrow} G_0 \overset{\Bbb N'}{\longrightarrow} K/R \longrightarrow 0
$$
as a $T$-module, where $G_0 = T^{\oplus \ell_n} \oplus T^{\oplus \ell_{n-1}} \oplus \cdots \oplus T^{\oplus \ell_1}$, $\Bbb N'=[
\begin{smallmatrix}
\overline{f_{n1}} \cdots \overline{f_{n \ell_n}} &  \cdots & \overline{f_{11}} \cdots \overline{f_{1\ell_1}}
\end{smallmatrix}]$, $s = \rank_TF_1$, and $\Bbb A'$ denotes the $\left(\sum_{i=1}^n \ell_i\right) \times s$ matrix obtained from the matrix $\Bbb A$ by deleting its first row.

By our assumption, we have isomorphisms
$$
K/R \cong \bigoplus_{i=1}^n (R/I_i)^{\oplus \ell_i} \cong \bigoplus_{i=1}^n (T/J_i)^{\oplus \ell_i}
$$
so that the $T$-module $K/R$ has a minimal presentation of the form
$$
G_1  \overset{\Bbb B}{\longrightarrow} G_0 \overset{\Bbb N'}{\longrightarrow} K/R \longrightarrow 0
$$
where $G_1 = T^{\oplus \ell_n \cdot m_n} \oplus T^{\oplus \ell_{n-1}\cdot m_{n-1}} \oplus \cdots \oplus T^{\oplus \ell_1\cdot m_1}$ and 
the matrix $\Bbb B$ is given by the form \vspace{0.5em}
$$
\Bbb B =
\left[ 
\begin{smallmatrix}
X^{(n)}_1  \cdots X^{(n)}_{m_n} & 0 & 0 & 0 & 0  & 0 & 0   \\
0 & \ddots & 0 & 0  & 0 & 0 & 0   \\
\vdots & \vdots & X^{(n)}_1  \cdots X^{(n)}_{m_n} & \vdots & \vdots &  \vdots & \vdots  \\
\vdots & \vdots & \vdots & \ddots & \vdots & \vdots & \vdots \\[5pt]
0 & 0 & 0 & 0 & X^{(1)}_1  \cdots X^{(1)}_{m_1} & 0 & 0\\
0 & 0 & 0 & 0  & 0  & \ddots & 0  \\
0 & 0 & 0 & 0  & 0 & 0 & X^{(1)}_1  \cdots X^{(1)}_{m_1} 
\end{smallmatrix}\right].\vspace{0.5em}
$$
The comparison between two presentations of $K/R$ gives rise to a commutative diagram
$$
\xymatrix{
G_1 \ar[r]^{\Bbb B}\ar[d]^{\xi} & G_0 \ar[r]\ar[d]^\cong & K/R\ar[r]\ar[d]^\cong & 0\\
 F_1 \ar[r]^{\Bbb A'}\ar[d]^{\eta} & G_0 \ar[r]\ar[d]^\cong & K/R \ar[r] \ar[d]^\cong & 0 \\
G_1 \ar[r]^{\Bbb B} & G_0 \ar[r] & K/R \ar[r] & 0}
$$
of $T$-modules, so the composite map $\eta \circ \xi$ is an isomorphism.
Thus, there exists an $s \times s$ invertible matrix $Q$ with entries in $T$ such that the equality 
$$
\Bbb A'\cdot Q= \left[\, \Bbb B \mid  O \, \right]
$$
holds, where $O$ denotes the null matrix. We now consider the matrix 
$$
\Bbb M = \Bbb A \cdot Q 
=
\arraycolsep5pt
\left[
\begin{array}{@{\,}ccc@{\,}}
~ & * &~\\
\hline
&
\multicolumn{1}{c}{\raisebox{-10pt}[0pt][0pt]{\large $\Bbb A'$}}\\
&&\\
\end{array}
\right]{\text{\large $Q$}} \ 
=
\arraycolsep5pt
\left[
\begin{array}{@{\,}c|c@{\,}}
~ * ~ & ~ * ~\\
\hline
\raisebox{-10pt}[0pt][0pt]{\large $\Bbb B$}&\raisebox{-10pt}[0pt][0pt]{\large $O$}\\
&\\
\end{array}
\right].
$$
Hence, $K$ has a minimal free presentation 
$$
F_1 \overset{\Bbb M}{\longrightarrow} F_0 \overset{\Bbb N}{\longrightarrow} K \longrightarrow 0
$$
of $T$-modules. 
Since $\Bbb N \cdot \Bbb M = O$, the equality
$$
a_{ij}^{(n)}\cdot (-1) + X_j^{(n)} \cdot f_{ni} = 0 
$$
holds in $K$ for every $1 \le i \le \ell_n$ and $1 \le j \le m_n$. Then
$$
\overline{a_{ij}^{(n)}} = a_{ij}^{(n)} \cdot 1 = X_j^{(n)} \cdot f_{ni} \in J_n K = I_n K = IK = I = J_n/\fka
$$
whence $a_{ij}^{(n)} \in J_n$. In addition, for each $1 \le k < n$, we have
$$
a_{ij}^{(k)}\cdot (-1) + X_j^{(k)}\cdot f_{ki} = 0  \ \ \text{for every} \ 1 \le i \le \ell_k, \ 2 \le j \le m_k
$$
which shows $a_{ij}^{(k)} \in J_n$, as desired.

It remains to show the last equality of ideals. To do this, we need the following fact. 

\begin{lem}[{\cite[Lemma 7.9]{GTT}}]\label{10.3}
Let $(T, \n)$ be a Noetherian local ring, $\bm{x} = x_1, x_2, \ldots, x_{\ell}$ a regular sequence on $T$ and $\bm{a} = a_1, a_2, \ldots, a_{\ell}$ a sequence of elements in $T$. We denote by $\Bbb K = {\Bbb K}_\bullet(\xx ; T)$ the Koszul complex of $T$ associated to $\xx$ and by $\Bbb L = {\Bbb L}_\bullet(\bm{a} ; T)$ the Koszul complex of $T$ associated to $\bm{a}$. Then the equality 
$$
\partial_1^{\Bbb L}(\Ker \partial_1^{\Bbb K}) = \rmI_2
\begin{pmatrix}
a_1 & a_2 & \cdots & a_{\ell} \\
x_1 & x_2 & \cdots & x_{\ell}
\end{pmatrix} 
$$
holds, where $\partial_1^{\Bbb K}$ and $\partial_1^{\Bbb L}$ denote the first differential maps of $\Bbb K$ and $\Bbb L$, respectively. 
\end{lem}

Let $\bm{e}_i$ be the standard basis of $F_0 = T^{\oplus r}$ and set $L=\Im \Bbb M$. For each $a \in \fka$, we see that $a \bm{e}_1 \in L$, whence the equality
\begin{eqnarray*}
a \bm{e}_1 \!\! &=& \!\! \sum_{i=1}^{\ell_n} \sum_{j=1}^{m_n} c_{ij}^{(n)}\left( a_{ij}^{(n)} \bm{e}_1 + X_j^{(n)} \bm{e}_{i+1}\right) +  \sum_{i=1}^{\ell_{n-1}} \sum_{j=1}^{m_{n-1}} c_{ij}^{(n-1)}\left( a_{ij}^{(n-1)} \bm{e}_1 + X_j^{(n-1)} \bm{e}_{i+\ell_n+1}\right) \\
&+& \!\!\! \cdots + \sum_{i=1}^{\ell_1} \sum_{j=1}^{m_1} c_{ij}^{(1)}\left( a_{ij}^{(1)} \bm{e}_1 + X_j^{(1)} \bm{e}_{i+\ell_n + \cdots + \ell_2 + 1}\right)  + \sum_{i=1}^q d_i \left(c_i \bm{e}_1\right)
\end{eqnarray*}
holds, where $c_{ij}^{k} \in T$ and $d_i \in T$. By comparing with components of both sides, we get
$$
a = \sum_{i=1}^{\ell_n} \sum_{j=1}^{m_n}\left( c_{ij}^{(n)}a_{ij}^{(n)}\right) + \sum_{i=1}^{\ell_{n-1}} \sum_{j=1}^{m_{n-1}}\left( c_{ij}^{(n-1)}a_{ij}^{(n-1)}\right) + \cdots + \sum_{i=1}^{\ell_1} \sum_{j=1}^{m_1}\left( c_{ij}^{(1)}a_{ij}^{(1)}\right) + \sum_{i=1}^q d_i c_i
$$
and 
\begin{eqnarray*}
&&\sum_{j=1}^{m_n}\left(c_{ij}^{(n)}X_j^{(n)}\right) = 0 \ \  \ \text{for every} \ 1 \le i \le \ell_n, \\
&&\sum_{j=1}^{m_{n-1}}\left(c_{ij}^{(n-1)}X_j^{(n-1)}\right) = 0 \ \  \ \text{for every} \ 1 \le i \le \ell_{n-1}, \\
&&  \hspace{2em}\vdots\\
&&\sum_{j=1}^{m_1}\left(c_{ij}^{(1)}X_j^{(1)}\right) = 0 \ \  \ \text{for every} \ 1 \le i \le \ell_1.
\end{eqnarray*}
For each integers $1 \le k \le n$ and $1 \le i \le \ell_k$, we consider the Koszul complexes $\Bbb K = {\Bbb K}_\bullet(X_1^{(k)}, X_2^{(k)}, \ldots, X_{m_k}^{(k)} ; T)$ and $\Bbb L = {\Bbb L}_\bullet(a_{i1}^{(k)}, a_{i2}^{(k)}, \ldots, a_{im_k}^{(k)} ; T)$ of $T$ associated to the sequences $X_1^{(k)}, X_2^{(k)}, \ldots, X_{m_k}^{(k)}$ and $a_{i1}^{(k)}, a_{i2}^{(k)}, \ldots, a_{im_k}^{(k)}$, respectively. 
As $\sum_{j=1}^{m_k}c_{ij}^{(k)} T_j \in \Ker \partial _1^{\Bbb K}$, we have 
$$
\sum_{j=1}^{m_k}c_{ij}^{(k)} a_{ij}^{(k)} \in \partial_1^{\Bbb L}(\Ker \partial_1^{\Bbb K}) = \rmI_2
\begin{pmatrix}
a_{i1}^{(k)} & a_{i2}^{(k)} & \cdots & a_{im_k}^{(k)} \\[4pt]
X_1^{(k)} & X_2^{(k)} & \cdots & X_{m_k}^{(k)}
\end{pmatrix} 
$$
by Lemma \ref{10.3}. Therefore, we have 
$$
\fka \subseteq \sum_{i=1}^{n}\sum_{j=1}^{\ell_i} \rmI_2
\begin{pmatrix}
a_{j1}^{(i)} & a_{j2}^{(i)} & \cdots & a_{j m_i}^{(i)} \\[4pt]
X_1^{(i)} & X_2^{(i)} & \cdots  & X_{m_i}^{(i)}
\end{pmatrix} + (c_1, c_2, \ldots, c_q).
$$
On the other hand, since $\Bbb N \cdot \Bbb M = O$, the equality $\overline{c_i} = (-1)\cdot c_i = 0$ holds in $K$ for each $1 \le i \le q$, so that $c_i \in \fka$. In addition, for each $1 \le i \le \ell_n$, we have
$$
\left(X_p^{(n)}a_{ij}^{(n)} - X_j^{(n)}a_{ip}^{(n)}\right) \bm{e}_1 = X_p^{(n)}\left(a_{ij}^{(n)}\bm{e}_1 + X_j^{(n)}\bm{e}_{i+1} \right) - X_j^{(n)}\left(a_{ip}^{(n)}\bm{e}_1 + X_p^{(n)}\bm{e}_{i+1}\right) \in L
$$
for every $1 \le p < j \le m_n$. This implies  $(-1) \left(X_p^{(n)}a_{ij}^{(n)} - X_j^{(n)}a_{ip}^{(n)} \right) = 0$ in $K$. Hence
$$
\rmI_2
\begin{pmatrix}
a_{i1}^{(n)} & a_{i2}^{(n)} & \cdots & a_{im_n}^{(n)} \\[4pt]
X_1^{(n)} & X_2^{(n)} & \cdots & X_{m_n}^{(n)}
\end{pmatrix} \subseteq \fka. 
$$
Besides, for each $1 \le i \le \ell_{n-1}$, we have
\begin{eqnarray*}
\left(X_p^{(n-1)}a_{ij}^{(n-1)} - X_j^{(n-1)}a_{ip}^{(n-1)}\right) \bm{e}_1 
\!\!\!&=&\!\!\! X_p^{(n-1)}\left(a_{ij}^{(n-1)}\bm{e}_1 + X_j^{(n-1)}\bm{e}_{i+\ell_n+1} \right) \\
&&- X_j^{(n-1)}\left(a_{ip}^{(n-1)}\bm{e}_1 + X_p^{(n-1)}\bm{e}_{i+\ell_n+1}\right) \in L
\end{eqnarray*}
for every $1 \le p < j \le m_{n-1}$, so that $X_p^{n-1}a_{ij}^{(n-1)} - X_j^{(n-1)}a_{ip}^{(n-1)} \in \fka$. Therefore
$$
\rmI_2
\begin{pmatrix}
a_{i1}^{(n-1)} & a_{i2}^{(n-1)} & \cdots & a_{im_{n-1}}^{(n-1)} \\[4pt]
X_1^{(n-1)} & X_2^{(n-1)} & \cdots & X_{m_{n-1}}^{(n-1)}
\end{pmatrix} \subseteq \fka. 
$$
Since the same technique works for $1 \le k \le n-2$, we conclude that
$$
\rmI_2
\begin{pmatrix}
a_{i1}^{(k)} & a_{i2}^{(k)} & \cdots & a_{im_k}^{(k)} \\[4pt]
X_1^{(k)} & X_2^{(k)} & \cdots & X_{m_k}^{(k)} 
\end{pmatrix} \subseteq \fka \ \ \ \text{for every} \ \ 1 \le i \le \ell_k.
$$
This finally provides the equality
$$
\fka = \sum_{i=1}^{n}\sum_{j=1}^{\ell_i} \rmI_2
\begin{pmatrix}
a_{j1}^{(i)} & a_{j2}^{(i)} & \cdots & a_{j m_i}^{(i)} \\[4pt]
X_1^{(i)} & X_2^{(i)} & \cdots  & X_{m_i}^{(i)}
\end{pmatrix} + (c_1, c_2, \ldots, c_q)
$$
as claimed. 
\end{proof}

We now apply Theorem \ref{MFR1} to Goto rings. 
Suppose that $\fka \subseteq \n^2$, $R$ is an $n$-Goto ring, and $v(R/\fkc) = 1$. Note that $\mu_R(\m) = \ell$. As $R$ is $n$-Goto and $v(R/\fkc) = 1$, we can choose a minimal system $x_1, x_2, \ldots, x_\ell$ of generators of $\m$ such that $\fkc = R:R[K] = (x_1^n, x_2, \ldots, x_{\ell})$.

For each $1 \le i \le n$, we set $I_i = (x_1^i, x_2, \ldots, x_{\ell})$. We then have a chain 
$$
R:K=\fkc = I_n \subsetneq I_{n-1} \subsetneq \cdots \subsetneq I_1 = \fkm
$$ of ideals, and by Theorem \ref{4.8} we have an isomorphism
$$
K/R \cong \bigoplus_{i=1}^n \left(R/I_i\right)^{\oplus \ell_i}
$$
of $R$-modules for some $\ell_n>0$, $\ell_i \ge 0$ ($1 \le i < n$) such that $\sum_{i=1}^n \ell_i = r-1$. We choose $X_i \in \n$ so that $x_i = \overline{X_i}$ in $R$. Letting $J_i = (X_1^i, X_2, \ldots, X_{\ell})$ for each $1 \le i \le n$, we have an isomorphism $T/J_i \cong R/I_i$ of $T$-modules, because $\ell_T(T/J_i) = \ell_R(R/I_i) = i$. Besides, since 
$
\fkc \subseteq \fkc K \subseteq \fkc R[K] = \fkc,
$
the equality $\fkc K= \fkc$ holds. Hence we have the following, which gives a necessary condition for $R=T/\fka$ to be an $n$-Goto ring.


\begin{cor}\label{10.4a}
Suppose that $\fka \subseteq \n^2$. If $R$ is an $n$-Goto ring and $v(R/\fkc) = 1$, then the $T$-module $K$ has a minimal free presentation of the form $
F_1 \overset{\Bbb M}{\longrightarrow} F_0 \overset{\Bbb N}{\longrightarrow} K \to 0,
$
where
$$
\Bbb N= \left[
\begin{smallmatrix}
-1 & f_{n1} \cdots f_{n \ell_n} & f_{n-1,1} \cdots f_{n-1, \ell_{n-1}} & \cdots & f_{11} \cdots f_{1\ell_1}
\end{smallmatrix}\right]
$$
and 
$$
\! \Bbb M = \!\! \left[ 
\begin{smallmatrix}
a_{11}^{(n)} a_{12}^{(n)} \cdots a_{1\ell}^{(n)} & \cdots & a_{\ell_n1}^{(n)} a_{\ell_n2}^{(n)} \cdots a_{\ell_n \ell}^{(n)} &  & \cdots & &  &  \hspace{-0.5em}a_{11}^{(1)} a_{12}^{(1)} \cdots a_{1\ell}^{(1)} &\hspace{-0.5em} \cdots & \hspace{-0.5em}a_{\ell_11}^{(1)} a_{\ell_12}^{(1)} \cdots a_{\ell_1 \ell}^{(1)}  &  c_1 c_2 \cdots c_q \\\vspace{-0.5em}
\hspace{-0.5em} X^n_1 X_2 \cdots X_{\ell} &  &  &   &  &  & & &  & & 0 \\
 & \ddots &  &   &  &   &  & & &  &  0 \\ \vspace{-1em}
 &  & \hspace{-0.5em} X^n_1 X_2 \cdots X_{\ell} &  &  &   \\ \vspace{-0.5em}
 &  &  & \hspace{-0.5em} X_1^{n-1} X_2 \cdots X_{\ell} &  &  & & &  &   &  \vdots\\ 
 &  &  &  &  \ddots &  &  \\ \vspace{-0.5em}
 &  &  &  &  &  X_1^{n-1} X_2 \cdots X_{\ell} &   \\
 &  &  &  &  &  & \ddots &  &  & \\
 &  &  &  &  &  &  & X_1 X_2 \cdots X_{\ell} &  &  &   \\
 &  &  &  &  &  &  &  & \ddots &  & 0\\
 &  &  &  &  &  &  &  &  & X_1 X_2 \cdots X_{\ell} & 0\\
\end{smallmatrix}\right]
$$
with $a_{ij}^{(n)} \in J_n~(1 \le i \le \ell_n, \, 1 \le j \le \ell)$, $a_{ij}^{(k)} \in J_n~(1\le k < n, \, 1 \le i \le \ell_k, \, 2 \le j \le \ell)$, $a_{i1}^{(k)}\in \n ~(1 \le k < n, \, 1 \le i \le \ell_k)$, and $c_q \in \n~(q \ge 0)$. 
Moreover, one has the equality
$$
\fka = \sum_{i=1}^{n}\sum_{j=1}^{\ell_i} \rmI_2
\begin{pmatrix}
a_{j1}^{(i)} & a_{j2}^{(i)} & \cdots & a_{j \ell}^{(i)} \\[4pt]
X_1^i & X_2 & \cdots  & X_{\ell}
\end{pmatrix} + (c_1, c_2, \ldots, c_q).
$$
\end{cor}

The next corresponds to \cite[Corollary 7.10]{GTT} (resp. \cite[Corollary 2.3]{GIT}) for almost Gorenstein (resp. $2$-almost Gorenstein) rings.



\begin{cor}
With the same notation as in Theorem \ref{MFR1} and Corollary \ref{10.4a}, the following assertions hold true. 
\begin{enumerate}
\item[$(1)$] If $\ell = 3$ and $m_i = \ell$ for all $1 \le i \le n$, then $r = 2$, $q = 0$, $\ell_n=1$, and $\ell_i=0$ for all $1 \le i < n$. Hence, if $R$ is an $n$-Goto ring with $v(R/\fkc) = 1$, $\fka \subseteq \n^2$, and $\ell = 3$, the matrix $\Bbb M$ has the form
$$\Bbb M = 
\begin{pmatrix}
a_{11}^{(n)} & a_{12}^{(n)} & a_{13}^{(n)} \\[4pt]
X_1^{n} & X_2 & X_3
\end{pmatrix}.
$$ 
\item[$(2)$] If $R$ has minimal multiplicity and $m_i = \ell$ for all $1 \le i \le n$, then $q = 0$. Hence, if $R$ is $n$-Goto with minimal multiplicity, $v(R/\fkc) = 1$, and $\fka \subseteq \n^2$, then $q=0$.
\end{enumerate}
\end{cor}

\begin{proof}
$(1)$ Note that the ring $R$ has a minimal free resolution
$$
0 \longrightarrow F_2=T^{\oplus r} \overset{{}^t \Bbb M}{\longrightarrow} F_1 = T^{\oplus (r+1)} \longrightarrow F_0=T \longrightarrow R \longrightarrow 0
$$ 
of $T$-modules, where the matrix $\Bbb M$ has the form stated in Theorem \ref{MFR1}. 
Then
$$
r +1 = \rank_TF_1 = s = \left(\sum_{i=1}^n \ell_i\right) \cdot \ell + q = 3\cdot (r-1)+q,
$$
whence $4-2\cdot r=q \ge 0$. Thus $\rmr(R)=2$ and $q=0$, because $R$ is not  Gorenstein. Since $0 < \ell_n \le \sum_{i=1}^n \ell_i = 1$, we conclude that $\ell_n=1$ and $\ell_i=0$ for all $1 \le i < n$.

$(2)$ As $R$ has minimal multiplicity, we get $r=\ell-1$ and the equalities
$$
(r-1) \ell + q  =\left(\sum_{i=1}^n \ell_i\right) \cdot \ell + q= s = \ell (\ell -2)
$$
where the last equality follows from \cite[{\sc Theorem} 1 (iii)]{S2}. Hence $q=0$. 
\end{proof}

Although we need a mild assumption on rings, Corollary \ref{10.4a} provides an explicit system of generators of defining ideals in numerical semigroup rings possessing canonical ideals with reduction number two. 

Let us note some examples illustrating Corollary \ref{10.4a}. 
The first example is a $3$-Goto ring with minimal multiplicity. We denote by ${}^t(-)$ the transpose of a matrix.

\begin{ex}
Let $V = k[[t]]$ be the formal power series ring over a field $k$, and set $R = k[[t^4 ,t^{11}, t^{13}, t^{14}]]$. Hence, $R=k[[H]]$, the semigroup ring of the numerical semigroup $H=\left<4, 11, 13, 14\right>$. Then ${\rmc}(H) = 11$ and $\mathrm{PF}(H) =\{7, 9, 10\}$, whence $K=R + Rt + Rt^3$ and $R[K] =V$. This shows $K^2 = K^3$ and $\fkc = R:V = t^{\rmc(H)}V = t^{11}V = (t^{11}, t^{12}, t^{13}, t^{14}) =  \left((t^{4})^3, t^{12}, t^{13}, t^{14}\right)$. Hence $\ell_R(R/\fkc) = 3$ and $v(R/\fkc) = 1$. So $R$ is a $3$-Goto ring with minimal multiplicity.
Let $T=k[[X,Y,Z,W]]$ be the formal power series ring and $\varphi : T \to R$ the $k$-algebra map defined by $\varphi (X) = t^4$, $\varphi(Y)=t^{11}$, $\varphi(Z) = t^{13}$, and $\varphi(W) = t^{14}$. Then $R$ has a minimal $T$-free resolution
of the form
$$
\Bbb F : \ 0 \longrightarrow T^{\oplus 3} \overset{{}^t\Bbb M}{\longrightarrow} T^{\oplus 8} \longrightarrow T^{\oplus 6} \longrightarrow T \longrightarrow R \longrightarrow 0$$
where 
$$
\Bbb M =
\begin{bmatrix}
Z & -X^3 & -W & -XY & Y & W & X^4 & XZ \\
X^3 & Y & Z & W & 0 & 0 & 0 & 0 \\
0 & 0 & 0 & 0 & X^2 & Y & Z & W 
\end{bmatrix}. 
$$
The $T$-dual of $\Bbb F$ gives rise to the minimal presentation
$$
T^{\oplus 8} \overset{\Bbb M}{\longrightarrow} T^{\oplus 3} \longrightarrow K \longrightarrow 0
$$
of the canonical fractional ideal $K$, so that $K/R \cong T/(X^3, Y, Z, W) \oplus T/(X^2, Y, Z, W)$. 
We then have 
$$
\Ker \varphi = {\rmI}_2
\begin{pmatrix}
Z & \hspace{-0.3em} -X^3 & \hspace{-0.3em} -W & \hspace{-0.3em} -XY \\
X^3 & Y & Z & W 
\end{pmatrix}
+
{\rmI}_2
\begin{pmatrix}
Y & W & X^4 & XZ \\
X^2 & Y & Z & W
\end{pmatrix}.
$$ 
\end{ex}

Note that one cannot expect $q = 0$ in general, see \cite[Example 7.11]{GTT}, \cite[Example 2.7]{GIT}. 
There is an example of a Goto ring with $q=0$ without assuming of minimal multiplicity.



\begin{ex}
Let $V = k[[t]]$ be the formal power series ring over a field $k$, and set $R = k[[t^8 ,t^{13}, t^{15}, t^{17}, t^{19}, t^{22}]] = k[[H]]$, where $H=\left<8, 13, 15, 17, 19, 22\right>$. Note that ${\rmc}(H) = 21$ and $\mathrm{PF}(H) =\{9, 11, 14, 18, 20\}$. Then 
\begin{eqnarray*}
K \!\! &=& \!\! R + Rt^2 + Rt^6 + Rt^9  + Rt^{11} \\
K^2 \!\!\! &=& \!\!\! K+ \left(Rt^2 + Rt^6 + Rt^9  + Rt^{11}\right)^2 = K + Rt^4 \\
K^3 \!\!\! &=& \!\!\! K^2 +  \left(R+Rt^2 + Rt^6 + Rt^9  + Rt^{11}\right)t^4  = K^2
\end{eqnarray*}
while $\ell_R(K^2/K) = 3$. Hence $R$ is a $3$-Goto ring which does not have minimal multiplicity. Since $\fkc = R : R[K] = R:K = \left((t^{8})^3, t^{13}, t^{15}, t^{17}, t^{19}, t^{22}\right)$, we have $v(R/\fkc) = 1$. 

We now compute the generators of defining ideal of $R$.
Let $T=k[[X,Y,Z,W, V, F]]$ denote the formal power series ring and $\varphi : T \to R$ the $k$-algebra map defined by $\varphi (X) = t^8$, $\varphi(Y)=t^{13}$, $\varphi(Z) = t^{15}$, $\varphi(W) = t^{17}$, $\varphi(V) = t^{19}$, and $\varphi(F) = t^{22}$. Then $R$ has a minimal $T$-free resolution
of the form
$$
\Bbb F : \ 0 \longrightarrow T^{\oplus 5} \overset{{}^t\Bbb M}{\longrightarrow} T^{\oplus 24} \longrightarrow T^{\oplus 45} \longrightarrow T^{\oplus 40} \longrightarrow T^{\oplus 15} \longrightarrow T \longrightarrow R \longrightarrow 0$$
where \vspace{0.3em}
$$
{\footnotesize
\! \Bbb M = \!\! \left[
\begin{smallmatrix}
Y^2 & Z & W & V & XY & X^3 & -F & -V & \hspace{-0.2em} -XY & \hspace{-0.2em} -XZ & \hspace{-0.2em} -XW & \hspace{-0.2em} -YZ & -V & -X^3 & -Y^2 & \hspace{-0.2em} -YZ & -Z^2 & \hspace{-0.2em} -X^2W & -W & -F & -X^3 & -Y^2& \hspace{-0.2em} -YZ & \hspace{-0.2em}-X^2 Z\\[3pt]
X^3 & Y & Z & W & V & F & 0 & 0 & 0 & 0  & 0 & 0 & 0 & 0 & 0 & 0 & 0 & 0  & 0 & 0 & 0 & 0 & 0 & 0 \\[3pt]
0 & 0 & 0 & 0  & 0 & 0 & X^2 & Y & Z & W & V & F  & 0 & 0 & 0 & 0 & 0 & 0  & 0 & 0 & 0 & 0 & 0 & 0 \\[3pt]
0 & 0 & 0 & 0  & 0 & 0 & 0 & 0 & 0 & 0  & 0 & 0 & X & Y & Z & W & V & F  & 0 & 0 & 0 & 0 & 0 & 0   \\[3pt]
0 & 0 & 0 & 0  & 0 & 0  & 0 & 0 & 0 & 0  & 0 & 0  & 0 & 0 & 0 & 0 & 0 & 0  & X & Y & Z & W & V & F  \\[2pt]
\end{smallmatrix}\right]. 
}\vspace{0.3em}
$$
By taking the $T$-dual of $\Bbb F$, we get the minimal presentation
$$
T^{\oplus 24} \overset{\Bbb M}{\longrightarrow} T^{\oplus 5} \longrightarrow K \longrightarrow 0
$$
of $K$. Then we have an isomorphism
$$
K/R \cong T/(X^3, Y, Z, W, V, F) \oplus T/(X^2, Y, Z, W, V, F) \oplus \left(T/(X, Y, Z, W, V, F)\right)^{\oplus 2}
$$ 
of $T$-modules and  \vspace{0.5em}
{\footnotesize
\begin{eqnarray*}
\ \  \Ker \varphi \!\!\!\! &=&  \!\!\! {\rmI}_2
\begin{pmatrix}
Y^2 & Z & W & V & XY & X^3  \\
X^3 & Y & Z & W & V & F 
\end{pmatrix}
+
{\rmI}_2
\begin{pmatrix}
-F & -V & \hspace{-0.3em} -XY & \hspace{-0.3em} -XZ & \hspace{-0.3em} -XW & \hspace{-0.3em} -YZ    \\
X^2 & Y & Z & W & V & F 
\end{pmatrix} \\[5pt]
\!\!\! &+& \!\!\!
{\rmI}_2
\begin{pmatrix}
-V & \hspace{-0.3em} -X^3 & \hspace{-0.3em} -Y^2 & \hspace{-0.3em} -YZ & \hspace{-0.3em} -Z^2 & \hspace{-0.3em} -X^2W  \\
X & Y & Z & W & V & F 
\end{pmatrix}
+ 
{\rmI}_2
\begin{pmatrix}
-W &\hspace{-0.3em} -F &\hspace{-0.3em} -X^3 &\hspace{-0.3em} -Y^2& \hspace{-0.3em} -YZ & \hspace{-0.3em}-X^2 Z  \\
X & Y & Z & W & V & F 
\end{pmatrix}.
\end{eqnarray*}
}
\end{ex}

\vspace{0.3em}

As we show next, Goto rings does not satisfy the condition $v(R/\fkc)=1$ in general. 

\begin{ex}
Let $V = k[[t]]$ be the formal power series ring over a field $k$, and set $R = k[[H]]$, where $H = \left<7, 10, 22\right>$. Then $\rmc(H) = 34$ and $\mathrm{PF}(H) =\{25, 33\}$. We have $K^2 = K+Rt^{16}$, $K^3 = K^2 + Rt^{24} = K^2$, and $\fkc=(t^{14}, t^{20}, t^{22}, t^{24})$. Hence $v(R/\fkc) = 2$, although $R$ is $4$-Goto because $\ell_R(K/R) = 4$ (see Lemma \ref{8.1}).

\end{ex}

We now investigate a sufficient condition for $R = T/\fka$ to be an $n$-Goto ring in terms of the presentation of fractional canonical ideals. In the following, we maintain Setup \ref{10.1}. 
We choose generators $f_{ij} \in K$ satisfying
$$
K=R + \sum_{i=1}^n\sum_{j=1}^{\ell_i}Rf_{ij}. 
$$
where $\ell_n > 0$, $\ell_i \ge 0$ ($1 \le i < n$) such that $\sum_{i=1}^n \ell_i = r-1$. 


The following provides a generalization of \cite[Theorem 7.8]{GTT} (resp. \cite[Theorem 2.9]{GIT}) for one-dimensional almost Gorenstein (2-almost Gorenstein) rings. 



\begin{thm}\label{10.9}
Let $X_1, X_2, \ldots, X_{\ell} \in \n$ be a regular system of parameters of $T$. Suppose that $\fka \subseteq \n^2$ and the $T$-module $K$ has a minimal free presentation of the form 
$$
F_1 \overset{\Bbb M}{\longrightarrow} F_0 \overset{\Bbb N}{\longrightarrow} K \longrightarrow 0
$$
where ${\Bbb M}$ and ${\Bbb N}$ are the matrices of the form stated in Corollary \ref{10.4a}, satisfying the conditions that $a^{(n)}_{ij} \in J_n~(1 \le i \le \ell_n, \ 1 \le j \le \ell)$, $a^{(k)}_{ij} \in J_n~(1 \le k < n, \ 1 \le i\le \ell_k, \ 2 \le j \le \ell)$, and $a^{(k)}_{i1} \in J_k~(1 \le k < n, \ 1 \le i \le \ell_k)$, where $J_i = (X_1^i, X_2, \ldots, X_{\ell})$ for each $1 \le i \le n$. 
Then $R$ is an $n$-Goto ring and $v(R/\fkc) = 1$. 
\end{thm}

\begin{proof}
Let $\overline{(-)}$ be the image in $K/R$.
Note that $K/R$ has a minimal free presentation
$$
F_1 \overset{\Bbb B'}{\longrightarrow} G_0 \overset{\Bbb L}{\longrightarrow} K/R \longrightarrow 0
$$
of $T$-modules, where $\Bbb L=[
\begin{smallmatrix}
\overline{f_{n1}} \cdots \overline{f_{n \ell_n}} &  \cdots & \overline{f_{11}} \cdots \overline{f_{1\ell_1}}
\end{smallmatrix}]$, $\Bbb B' = \left[\, \Bbb B \mid  O \, \right]$, and \vspace{0.5em}
$$
\Bbb B =
\left[ 
\begin{smallmatrix}
X_1^n X_2 \cdots X_{\ell} & 0 & 0 & 0 & 0  & 0 & 0   \\
0 & \ddots & 0 & 0  & 0 & 0 & 0   \\
\vdots & \vdots & X_1^n X_2 \cdots X_{\ell} & \vdots & \vdots &  \vdots & \vdots  \\
\vdots & \vdots & \vdots & \ddots & \vdots & \vdots & \vdots \\[5pt]
0 & 0 & 0 & 0 & X_1 X_2 \cdots X_{\ell} & 0 & 0\\
0 & 0 & 0 & 0  & 0  & \ddots & 0  \\
0 & 0 & 0 & 0  & 0 & 0 & X_1 X_2 \cdots X_{\ell}
\end{smallmatrix}\right]. 
\vspace{0.5em}
$$
Here $G_0 = T^{\oplus \ell_n} \oplus T^{\oplus \ell_{n-1}} \oplus \cdots \oplus T^{\oplus \ell_1}$ and $O$ denotes the null matrix. This induces an isomorphism 
$$
K/R \cong \bigoplus_{i=1}^n \left(T/J_i\right)^{\oplus \ell_i} 
$$ 
of $T$-modules. By setting $I=J_nR$, we get
$$
\fkc = R:R[K] \subseteq R:K = \bigcap_{i=1}^n J_i  R = J_n R. 
$$
To show $K^2 = K^3$, it suffices to check the equality $\fkc=I$. Indeed, it is enough to show $IK \subseteq I$ because $R[K] = K^m$ for all $m \gg 0$. 
As $\Bbb N \cdot \Bbb M = O$, we obtain the equality
$$
a_{ij}^{(n)}\cdot (-1) + Z_j^{(n)} \cdot f_{ni} = 0 \ \ \ (1 \le i \le \ell_n, \ 1 \le j \le \ell)
$$
in $K$, where
$$
Z_j^{(k)} = 
\begin{cases}
X_1^k & (j=1) \\
X_j & (2 \le j \le \ell)
\end{cases}
$$
for each $1 \le k \le n$. Hence,  $Z_j^{(n)} \cdot f_{ni} = a_{ij}^{(n)}\cdot 1 \in J_n R = I$. Similarly, for each $1 \le k < n$, the equality 
$$
a_{ij}^{(k)}\cdot (-1) + Z_j^{(k)}\cdot f_{ki} = 0 \ \ \ (1 \le i \le \ell_k, \ 2 \le j \le \ell)
$$
guarantees that $Z_j^{(k)}\cdot f_{ki} \in I$. In addition, we have the equality
$$
a_{i1}^{(k)}\cdot (-1) + X_1^{k} \cdot f_{ki} = 0 \ \ \ (1 \le k < n, \ 1 \le i \le \ell_k)
$$
inside of $K$, so that $X_1^k \cdot f_{ki} \in J_kR = (X_1^k, X_2, \ldots, X_\ell)R$. Therefore we get $X_1^n \cdot f_{ki} = X_1^{n-k}\cdot X_1^k\cdot f_{ki} \in X_1^{n-k}J_kR \subseteq J_nR = I$. This shows $IK \subseteq I$, i.e., $\fkc = I$. In particular, $v(R/\fkc) = 1$.

It remains to show that $\ell_R(R/\fkc) = n$. We set $p = \ell_R(R/\fkc)$. A surjective homomorphism $T/J_n \to R/\fkc$ induces that $p=\ell_R(R/\fkc) \le \ell_T(T/J_n) = n$. Since $R$ is a $p$-Goto ring and $v(R/\fkc) = 1$, we can choose a minimal system $y_1, y_2, \ldots, y_{\ell}$ of generators of $\m$ such that $\fkc = (y_1^p, y_2, \ldots, y_{\ell})$. By setting $I_i' = (y_1^i, y_2, \ldots, y_\ell)$ for each $1 \le i \le p$, we get an isomorphism
$$
K/R \cong \bigoplus_{i=1}^p \left(R/I_i'\right)^{\oplus m_i}
$$
of $R$-modules for some $m_p>0$, $m_i \ge 0$ ($1 \le i < p$) such that $\sum_{i=1}^p m_i = r-1$. Choose $Y_i \in \n$ so that $y_i$ is the image of $Y_i$ in $R$. For each $1 \le i \le p$, we have a surjective map $T/J_i' \to R/I_i'$ and $\ell_T(T/J_i') = \ell_R(R/I_i') = i$, whence $T/J_i' \cong R/I_i'$ as a $T$-module. Therefore
$$
\bigoplus_{i=1}^n \left(T/J_i\right)^{\oplus \ell_i} \cong  K/R \cong \bigoplus_{i=1}^p \left(T/J_i'\right)^{\oplus m_i}.
$$
Since $\ell_n > 0$, the $T$-module $T/J_n$ is a direct summand of $K/R$. This induces $T/{J_n} \cong T/J_i'$ for some $1 \le i \le p$. Thus $n=\ell_T(T/J_n) = \ell_T(T/J_i') = i \le p \le n$, i.e., $n= \ell_R(R/\fkc)$. Hence $R$ is an $n$-Goto ring. 
\end{proof}
 
 
\begin{rem}
Let $X_1, X_2, \ldots, X_{\ell} \in \n$ be a regular system of parameters of $T$. We set 
$$
I = \rmI_2
\begin{pmatrix}
X_1^{\alpha_1} & X_2^{\alpha_2}  & \cdots &  X_{\ell-1}^{\alpha_{\ell-1}} & X_{\ell}^{\alpha_{\ell}}  \\[3pt]
X_2^{\beta_2}  & X_3^{\beta_3} & \cdots  & X_{\ell}^{\beta_{\ell}} & X_1^{\beta_1}
\end{pmatrix}
$$
where $\alpha_i, \beta_i \ge 1$ for each $1 \le i \le \ell$. Then we have
$$
I + (X_1) = 
\rmI_2
\begin{pmatrix}
0 & X_2^{\alpha_2}  & \cdots &  X_{\ell-1}^{\alpha_{\ell-1}} & X_{\ell}^{\alpha_{\ell}}  \\[3pt]
X_2^{\beta_2}  & X_3^{\beta_3} & \cdots  & X_{\ell}^{\beta_{\ell}} & 0
\end{pmatrix} + (X_1)
$$
which shows $I + (X_1)$ is an $\n$-primary ideal of $T$.
Hence $A = T/I$ is a Cohen-Macaulay local ring with $\dim A = 1$. 
\end{rem}

The Eagon-Northcott complex associated to a $2 \times \ell$ matrix $\Bbb M$ provides the resolution of the ring $R=T/\rmI_2(\Bbb M)$ (see \cite{EN}). Hence we get the following. 

\begin{cor}[cf. {\cite[Corollary 2.10]{GIT}}]\label{10.11}
Let $X_1, X_2, \ldots, X_{\ell} \in \n$ be a regular system of parameters of $T$.  Choose integers $p_1, p_2, \ldots, p_\ell >0$ such that $p_1\ge n \ge 2$ and set 
$$
\fka = \rmI_2
\begin{pmatrix}
X_1^n & X_2 & \cdots & X_{\ell-1} & X_\ell  \\[3pt]
X_2^{p_2} & X_3^{p_3} & \cdots & X_{\ell}^{p_{\ell}} & X_1^{p_1}
\end{pmatrix}.
$$
Then $R=T/\fka$ is an $n$-Goto ring, $v(R/\fkc) = 1$, and the $R/\fkc$-module $K/R$ is free  of rank $\ell-2$.
\end{cor}

\begin{proof}
The ring $R=T/\fka$ is Cohen-Macaulay and of dimension one whose minimal free resolution as a $T$-module is given by the Eagon-Northcott complex associated to a matrix 
$$
\begin{bmatrix}
X_1^n & X_2 & \cdots & X_{\ell-1} & X_{\ell}  \\[3pt]
X_2^{p_2} & X_3^{p_3} & \cdots & X_{\ell}^{p_{\ell}} & X_1^{p_1}
\end{bmatrix}.
$$
By taking $T$-dual of the resolution, we get the minimal free presentation
$$
T^{\oplus \ell(\ell-2)} \overset{{\Bbb M}'}{\longrightarrow} T^{\oplus (\ell-1)} \overset{\varepsilon}{\longrightarrow} \rmK_R \longrightarrow 0
$$
of the canonical module $\rmK_R$ of $R$ where 
$$
{\Bbb M}'=\left[
\begin{smallmatrix}
X_2^{p_2} -X_3^{p_3} \cdots (-1)^{\ell+1}X_1^{p_1}
 & 0 &  &  &  &  \\
X_1^n -X_2 \cdots (-1)^{\ell+1}X_{\ell}
 & X_2^{p_2} -X_3^{p_3} \cdots (-1)^{\ell+1}X_1^{p_1}
 &    &  &  &  \\
 &  X_1^n -X_2 \cdots (-1)^{\ell+1}X_{\ell}
 &    &  &  &  \\
   &   & \ddots  & \\
   &   &  &  &  X_2^{p_2} -X_3^{p_3} \cdots (-1)^{\ell+1}X_1^{p_1}
 &  
\\   
   &   &  &  &  X_1^n -X_2 \cdots (-1)^{\ell+1}X_{\ell}
 &  X_2^{p_2} -X_3^{p_3} \cdots (-1)^{\ell+1}X_1^{p_1}
\\   
   &   &  &  &  0  & X_1^2 -X_2 \cdots (-1)^{\ell+1}X_{\ell} 
\end{smallmatrix}\right].
$$
For each $1 \le i \le n$, we denote by $x_i$ the image of $X_i$ in $R$. Note that 
$x_i \in \m$ is a non-zerodivisor on $R$,  $x_1^n \cdot x_1^{p_1} = x_2^{p_2}\cdot x_{\ell}$, and $x_i \cdot x_1^{p_1} = x_{i+1}^{p_{i+1}}\cdot x_{\ell}$  for all $2 \le i < \ell$.
By setting 
$$
y = \frac{x_2^{p_2}}{x_1^n}, \ \ 
f_i =
\begin{cases}
x_{i+1}^{p_{i+1}} & (1 \le i < \ell) \\
x_1^{p_1} & (i=\ell) 
\end{cases}, \ \ \text{and} \ \ 
g_i=
\begin{cases}
x_1^n & (i=1) \\
x_i & (2 \le i \le \ell) 
\end{cases},
$$
we see that $f_i = g_i \cdot y$ for every $1 \le i \le \ell$. Hence the equalities
$$
y^n = \frac{f_1}{g_1}\cdot \frac{f_2}{g_2}\cdots \frac{f_\ell}{g_\ell} = \frac{x_2^{p_2}}{x_1^n} \cdot \frac{x_3^{p_3}}{x_2} \cdots \frac{x_{\ell}^{p_{\ell}}}{x_{\ell-1}}\cdot \frac{x_1^{p_1}}{x_{\ell}} = x_1^{p_1-n}\cdot x_2^{p_2-1}\cdots x_{\ell}^{p_{\ell}-1} \in R
$$
hold, so that $y \in \overline{R}$. Let $K = \sum_{i=0}^{\ell-2}Ry^i$. Note that $K$ is an $R$-submodule of $\rmQ(R)$ and $R \subseteq K \subseteq \overline{R}$. We then have an isomorphism $K \cong \rmK_R$ of $R$-modules, i.e., $K$ is a fractional canonical ideal of $R$. 
In fact, we consider the $T$-linear map $\psi: T^{\oplus (\ell-1)} \longrightarrow K$ defined by $\psi (\bm{e}_i)= (-1)^{i}y^{i-1}$ for each $1 \le i < \ell$, where $\{\bm{e}_i\}_{1 \le i < \ell }$ denotes the standard basis of $T^{\oplus (\ell-1)}$. 
Since
$
\left[\begin{smallmatrix}
-1& y& -y^2& \cdots (-1)^{\ell-1}y^{\ell-2}
\end{smallmatrix}\right]
\cdot {\Bbb M}' = \bm{0},
$
we have a complex
$$
T^{\oplus \ell(\ell-2)} \overset{{\Bbb M}'}{\longrightarrow} T^{\oplus (\ell-1)} \overset{\psi}{\longrightarrow} K \longrightarrow 0
$$
of $T$-modules, and hence there is a surjective homomorphism $\sigma : \rmK_R \longrightarrow K$ with $\psi = \sigma \circ \varepsilon$. 
We are now assuming that $X = \Ker \sigma \ne (0)$, and seek a contradiction.  Let $\fkp \in \Ass_RX$. Then $\fkp \in \Ass_R \rmK_R$. As $K$ is a fractional ideal of $R$, we see that $K_{\fkp} \cong R_{\fkp}$. This yields an exact sequence
$$
0 \longrightarrow X_{\fkp} \longrightarrow \rmK_{R_{\fkp}} \longrightarrow R_{\fkp}  \longrightarrow 0
$$
of $R_{\fkp}$-modules, because $(\rmK_R)_\fkp \cong \rmK_{R_\fkp}$. The equality $\ell_{R_\fkp}(\rmK_{R_\fkp})= \ell_{R_\fkp}(R_\fkp)$ induces $X_\fkp=(0)$, which makes a contradiction. Therefore $\rmK_R \cong K$, so $K$ is a fractional canonical ideal of $R$, as claimed. In what follows, we identify $\rmK_R = K$ and $\varepsilon = \psi$. Hence 
$$
T^{\oplus \ell(\ell-2)} \overset{{\Bbb M}'}{\longrightarrow} T^{\oplus (\ell-1)} \overset{\psi}{\longrightarrow} K \longrightarrow 0
$$
gives a minimal $T$-free presentation of $K$. Since $(X_1^{p_1}, X_2^{p_2}, \ldots, X_{\ell}^{p_n}) \subseteq (X_1^{n}, X_2, \ldots, X_{\ell})$, after elementary column operations with coefficients in $T$ on the matrix $\Bbb M'$, we obtain the matrix $\Bbb M$ of the form \vspace{0.3em}
$$
\Bbb M =\left[ 
\begin{smallmatrix}
a_{11} a_{12} \cdots a_{1\ell} & a_{21} a_{22} \cdots a_{2\ell} &  \cdots  & \cdots &  a_{\ell-2, 1} a_{\ell-2, 2} \cdots a_{\ell-2, \ell}  \\[3pt]
X^n_1 X_2 \cdots X_\ell & 0  & \cdots & \cdots  & 0    \\[3pt]
0 & X^n_1 X_2 \cdots X_\ell & 0  & \cdots  & 0    \\[3pt]
0 & 0 &  \ddots &   & 0    \\[3pt]
\vdots & &  & \ \ddots & \vdots   \\[3pt]
0 & 0  & \cdots & \cdots & X_1^n X_2 \cdots X_\ell  \\[3pt]
\end{smallmatrix}\right] \vspace{0.3em}
$$
with $a_{ij} \in (X_1^{p_1}, X_2^{p_2}, \ldots, X_{\ell}^{p_\ell})$. By Theorem \ref{10.9}, the ring $R$ is $n$-Goto and $v(R/\fkc) = 1$. Moreover, we have the isomorphisms
$$
K/R \cong \left(T/(X_1^n, X_2, \ldots, X_\ell)\right)^{\oplus (\ell-2)} \cong (R/\fkc)^{\oplus (\ell-2)}. 
$$
This completes the proof.
\end{proof}

The simplest example of Corollary \ref{10.11} is stated as follows.  

\begin{ex}\label{10.12}
Let $\ell \ge 3$ be an integer and $T=k[[X_1, X_2, \ldots, X_{\ell}]]$ the formal power series ring over a field $k$. For given integers $m \ge n \ge 2$, the ring
$$
R=T/
\rmI_2
\begin{pmatrix}
X_1^n & X_2 & \cdots & X_{\ell-1} & X_{\ell}\\[3pt]
X_2 & X_3 & \cdots & X_{\ell} & X_1^m
\end{pmatrix}
$$
is $n$-Goto with $\dim R=1$ and $\rmr(R) = \ell -1$. 
\end{ex}



\begin{rem}
We can construct higher dimensional Cohen-Macaulay rings in this direction is just adding the indeterminates to elements in the matrix. More precisely, let $\ell \ge 2$ be an integer and $T=k[[X_1, X_2, \ldots, X_{\ell}, Y_1, Y_2, \ldots, Y_{\ell}]]$  the formal power series ring over a field $k$. Then 
$$
A = T/\rmI_2
\begin{pmatrix}
X_1^{\alpha_1} + Y_1 & X_2^{\alpha_2} + Y_2 & \cdots & X_{\ell-1}^{\alpha_{\ell-1}}+ Y_{\ell -1} & X_{\ell}^{\alpha_{\ell}} + Y_{\ell}\\[3pt]
X_2^{\beta_2} & X_3^{\beta_3} & \cdots & X_{\ell}^{\beta_{\ell}} & X_1^{\beta_1}
\end{pmatrix}
$$
is a Cohen-Macaulay local ring with $\dim A = \ell + 1$, where $\alpha_i, \beta_i \ge 1$ for each $1 \le i \le \ell$.
The above $Y_i$ can be added to any element of the matrix, while the number of adding elements does not have to be $\ell$. If we add $Y_i$ for $m$ pieces, the ring has dimension $m+1$. 
The way of this construction has also appeared in \cite[Section 4]{KMN}. 
\end{rem}


The following ensures that, for given integers $n \ge 2$ and $\ell \ge 3$, there exists an example of $n$-Goto rings of dimension $\ell$. 

\begin{ex}\label{10.14}
Let $\ell \ge 3$ be an integer and $T=k[[X_1, X_2, \ldots, X_{\ell}, V_1, V_2, \ldots, V_{\ell-1}]]$ the formal power series ring over a field $k$. 
For given integers $m \ge n \ge 2$, 
$$
A=T/
\rmI_2
\begin{pmatrix}
X_1^n & X_2 + V_1 & \cdots & X_{\ell-1}+ V_{\ell -2} & X_{\ell} + V_{\ell-1}\\[3pt]
X_2 & X_3 & \cdots & X_{\ell} & X_1^m
\end{pmatrix}
$$
is an $n$-Goto ring with $\dim A=\ell$ and $\rmr(A) = \ell -1$. 
\end{ex}

\begin{proof}
Note that $A$ is a Cohen-Macaulay local ring with $\dim A = \ell$. We denote by $x_i~(1 \le i \le \ell)$ and $v_i~(1 \le i < \ell)$ the images of $X_i$ and $V_i$ in $A$, respectively. Let $X_{\ell + 1} = X_1^m$ for convention. We then have an isomorphism
$$
A/(v_1, v_2, \ldots, v_{\ell-1})A \cong k[[X_1, X_2, \ldots, X_{\ell}]]/
\rmI_2
\begin{pmatrix}
X_1^n & X_2 & \cdots & X_{\ell-1} & X_\ell  \\[3pt]
X_2 & X_3 & \cdots & X_{\ell} & X_1^{m}
\end{pmatrix}
=B 
$$
of rings. By Example \ref{10.12}, the ring $B$ is $n$-Goto and of dimension one admitting the fractional canonical module of the form $(x_1^n, x_2, \ldots, x_{\ell-1})B\cong \rmK_{B}$; hence $\rmr(B) = \ell-1$. 

We set $I = (x_1^n, x_2, \ldots, x_{\ell-1})$, $Q = (x_1^n, v_1, \ldots, v_{\ell-1})$, $\fkq=(v_1, \ldots, v_{\ell-1})$, and $J=I+Q$. Note that $Q$ is a parameter ideal of $A$. Thus $\rmr(A) =  \ell-1$. Because of the isomorphism
$$
A/I \cong 
k[[V_1, V_2, \ldots, V_{\ell-1}]]/
\rmI_2
\begin{pmatrix}
0 & V_1 & V_2 & \cdots & V_{\ell-1} \\
0 & 0 & 0 & \cdots & 0
\end{pmatrix}
=k[[V_1, V_2, \ldots, V_{\ell-1}]],
$$
the ring $A/I$ is Cohen-Macaulay and of dimension $\ell-1$, so that $\fkq \cap I = \fkq I$. We then have 
$$
I/\fkq I \cong (I+\fkq) /\fkq = (x_1^n, x_2, \ldots, x_{\ell-1})(A/\fkq) \cong \rmK_{(A/\fkq)} \cong \rmK_A/\fkq \rmK_A
$$
which yields that $I$ is maximal Cohen-Macaulay as an $A$-module and 
$\rmr_A(I) = \rmr_A(I/\fkq I) = \rmr_A(\rmK_A/\fkq \rmK_A) = \rmr_A(\rmK_A) = 1$. As $x_1^n \in I$ is a non-zerodivisor on $A$, the ideal $I$ is faithful. Hence  $I \cong \rmK_A$ as an $A$-module. 

By Theorem \ref{3.5}, it suffices to show that $v_1, v_2, \ldots, v_{\ell-1}$ forms a super-regular sequence of $A$ with respect to $J$. We actually prove $J^3 = QJ^2$. Indeed, since $J= Q + (x_2, x_3, \ldots, x_{\ell-1})$, the equality
$J^2 = QJ + (x_2, x_3, \ldots, x_{\ell-1})^2$ holds. By induction argument, it is straightforward to check that 
\begin{center}
$x_{j+1} x_i \in QI$ \ for all \ $1 \le j \le \lceil \frac{\ell}{2} \rceil -1$ \ and \ $j+1 \le i \le \ell -(j+1)$
\end{center}
where $\lceil-\rceil$ denotes the ceiling function. By setting $\alpha = \lceil \frac{\ell}{2} \rceil$, we have
\begin{eqnarray*}
J^2 = QJ &+& \sum_{j=1}^{\alpha-1} \left(x_{j+1}x_{\ell-j}, x_{j+1}x_{\ell-j + 1}, \ldots, x_{j+1}x_{\ell-2}, x_{j+1}x_{\ell-1}\right) \\ 
&+& \left(x_{\alpha+1}, x_{\alpha+2}, \ldots, x_{\ell-1}\right)^2
\end{eqnarray*}
which indues that the equalities
\begin{eqnarray*}
J^3 \!\!&=&\!\! QJ^2 + \sum_{j=1}^{\alpha-1} x_{j+1}\left(x_{\ell-j}, x_{\ell-j + 1}, \ldots, x_{\ell-1}\right) J + \left(x_{\alpha+1}, x_{\alpha+2}, \ldots, x_{\ell-1}\right)^2J \\
\!\!&=&\!\! QJ^2 + \sum_{j=1}^{\alpha-1} x_{j+1}\left(x_{\ell-j}, x_{\ell-j+1}, \ldots, x_{\ell-1}\right)^2 + \left(x_{\alpha+1}, x_{\alpha+2}, \ldots, x_{\ell-1}\right)^3
\end{eqnarray*}
hold. Focusing on the $\ell$-th column of the matrix $\left[\begin{smallmatrix}
X_1^n & X_2 + V_1 & \cdots & X_{\ell-1}+ V_{\ell -2} & X_{\ell} + V_{\ell-1}\\[3pt]
X_2 & X_3 & \cdots & X_{\ell} & X_1^m
\end{smallmatrix}\right]$, we obtain $x_j x_{\ell} \in QI$ for every $2 \le j \le \ell -1$. By induction argument again, we get 
\begin{center}
$x_{i+1} x_{\ell-k}x_j \in QJ^2$ \ for all \ $1 \le k \le \ell-(\alpha-1)$, $2 \le i+1 \le \ell -k$, and $2 \le j \le \ell -1$.
\end{center}
This shows 
$$
\sum_{j=1}^{\alpha-1} x_{j+1}\left(x_{\ell-j}, x_{\ell-j+1}, \ldots, x_{\ell-1}\right)^2 + \left(x_{\alpha+1}, x_{\alpha+2}, \ldots, x_{\ell-1}\right)^3 \subseteq QJ^2,
$$
and hence the equality $J^3 = QJ^2$ holds. By Lemma \ref{2.4}, the sequence $v_1, v_2, \ldots, v_{\ell-1} \in \fkq$ of $A$ forms super-regular with respect to $J$. Therefore, $A$ is an $n$-Goto ring with $\dim A = \ell$ and $\rmr(A) = \ell-1$. 
\end{proof}

\begin{cor}
Let $\ell \ge 3$ and $m \ge n \ge 2$ be integers. 
For each $2 \le i \le \ell$, we denote by  $T_i=k[[X_1, X_2, \ldots, X_{\ell}, V_1, V_2, \ldots, V_{i-1}]]$ the formal power series ring over a field $k$, and set
$$
A_i=T_i/
\rmI_2
\begin{pmatrix}
X_1^n & X_2 + V_1 & \cdots & X_{i}+ V_{i-1} & X_{i+1}  & \cdots & X_{\ell-1} & X_{\ell} \\[3pt]
X_2 & X_3 & \cdots & X_{i+1} & X_{i+2}&  \cdots & X_{\ell} & X_1^m
\end{pmatrix}. 
$$
Then $A_i$ is an $n$-Goto ring with $\dim A_i=i$ and $\rmr(A_i) = \ell -1$. 
\end{cor}

\begin{proof}
The sequence of images of $V_1, V_2, \ldots, V_{\ell-1}$ in $A_{\ell}$ is super-regular with respect to $J=(X_1^n, X_2, \ldots, X_{\ell})A_{\ell} + (V_1, V_2, \ldots, V_{\ell-1})A_{\ell}$. Since $A_i \cong A_{\ell}/(V_i, V_{i+1}, \ldots, V_{\ell-1})A_{\ell}$, the assertion follows from Theorem \ref{3.5} and Example \ref{10.14}. 
\end{proof}

\section{Equimultiple Ulrich ideals and Goto rings}\label{sec9}

In this section we define Ulrich ideals for equimultiple ideals to construct Goto rings. Let $(A, \m)$ be a Cohen-Macaulay local ring with $d = \dim A$ and $I ~(\ne A)$ an ideal of $A$. The analytic spread $\ell(I)$ of $I$ is the Krull dimesnion of the fiber cone $\calF(I) = \bigoplus_{n \ge 0}I^n/\m I^n$. 
A proper ideal $I$ is called {\it equimultiple} if $\ell(I) = \height_AI$, where $\height_AI$ denotes the height of $I$. 

The notion of Ulrich ideals is defined for $\m$-primary ideals and is one of the modifications of that of stable maximal ideals introduced in 1971 by his monumental paper \cite{L} of J. Lipman. The present modification (\cite[Definition 1.1]{GOTWY}) was formulated by S. Goto, K. Ozeki, R. Takahashi, K.-i. Watanabe, and K.-i. Yoshida \cite{GOTWY} in 2014, where the authors developed the basic theory, revealing that the Ulrich ideals of Cohen-Macaulay local rings enjoy a beautiful structure theorem  for minimal free resolutions. 



We extend the notion of Ulrich ideals to equimultiple ideals as follows, where $\rma(\gr_I(A))$ denotes the $\rma$-invariant of $\gr_I(A)$ (\cite[Definition (3.1.4)]{GW}).

\begin{defn}[{cf. \cite[Definition 1.1]{GOTWY}}]\label{ulrich}
Let $I$ be an equimultiple ideal of $A$ with $s=\height_AI$. 
We say that $I$ is an {\it Ulrich ideal} of $A$ if the following conditions are satisfied. 
\begin{enumerate}
\item[$(1)$] $\gr_I(A)$ is a Cohen-Macaulay ring with $\rma(\gr_I(A)) = 1-s$.
\item[$(2)$] $I/I^2$ is a free $A/I$-module.
\end{enumerate}
When $I$ contains a minimal reduction $Q=(a_1, a_2, \ldots, a_s)$, the condition $(1)$ of Definition \ref{ulrich} is equivalent to saying that $I \ne Q$, $I^2 = QI$, and $A/I$ is Cohen-Macaulay (see e.g., \cite[Remark (3.1.6)]{GW}, \cite[Corollary 2.7]{VV}).
In addition, if $I^2 =QI$, the condition $(2)$ holds if and only if $I/Q$ is free as an $A/I$-module (the proof of \cite[Lemma 2.3 (2)]{GOTWY} works for equimultiple ideals). 
The existence of minimal reductions is automatically satisfied, if the residue class field 
$A/\m$ of $A$ is infinite, or if $A$ is analytically irreducible and $\dim A = 1$.
Besides, if $A/\m$ is infinite, all the minimal reductions of $I$ are minimally generated by $\ell(I)$ elements. 
\end{defn}



In what follows, let $I$ be an equimultiple ideals of $A$ with $\height_AI=1$ and assume that $I$ contains a principal ideal $Q=(a_1)$ as a minimal reduction. 

\begin{prop}
Suppose that $I$ is an Ulrich ideal of $A$. Then the following assertions hold true. 
\begin{enumerate}
\item[$(1)$] The equality $A:I = I:I$ holds. 
\item[$(2)$] For every system $x_1, x_2, \ldots, x_{d-1}$ of parameters of $A/I$ which is a subsystem of parameters of $A$, we set $\fkq = (x_1, x_2, \ldots, x_{d-1})$ and $J=I+\fkq$. Then $J/\fkq$ is an Ulrich ideal of $A/\fkq$.
\end{enumerate}
\end{prop}

\begin{proof}
$(1)$ Since $I/Q$ is free as an $A/I$-module, we get $I= Q:_AI =a(A:I)$. Hence the equalities
$$
A:I = \frac{I}{a} = I:I
$$
hold in $\rmQ(A)$, where the last follows from $I^2 = QI$.

$(2)$ Let $\overline{A} = A/\fkq$, $\overline{I} = I\overline{A}$, and $\overline{Q} = Q\overline{A}$. Then $\overline{I} = J/\fkq \cong I/\fkq \cap I = I/\fkq I$. We have $\overline{I}^2 = \overline{Q} \cdot \overline{I}$. If $\overline{I} = \overline{Q}$, then $I+\fkq = Q + \fkq$, so that $I = (Q + \fkq) \cap I = Q + \fkq I$. This shows $I=Q$, a contradiction. Thus $\overline{I} \ne \overline{Q}$. By setting $n= \mu_A(I)$,  we have $I/Q \cong (A/I)^{\oplus (n-1)}$. Therefore
$$
\overline{I}/\overline{Q}  \cong I/[(Q+\fkq)\cap I] = I/(Q + \fkq I) \cong A/\fkq \otimes_A I/Q \cong  (A/[I+\fkq])^{\oplus (n-1)} \cong (\overline{A}/\overline{I})^{\oplus (n-1)}
$$
and hence $\overline{I} = J/\fkq$ is an Ulrich ideal of $\overline{A} =A/\fkq$.
\end{proof}

We need a general lemma below. 

\begin{lem}\label{10.3}
Let $(A, \m)$ be a Cohen-Macaulay local ring with $d=\dim A$, $I$ an ideal of $A$ with $n=\dim A/I>0$, and $f_1, f_2, \ldots, f_n$ a system of parameters of $A/I$. Then there exists a subsystem $a_1, a_2, \ldots, a_n$ of parameters of $A$ such that the images of $f_i$ and $a_i$ in $A/I$ coincides for all $1 \le i \le n$. 
\end{lem}

\begin{proof}
We denote by $\Assh A$ the set of all prime ideals $\p$ in $A$ such that $\dim A/\fkp =d$. Since $\height_A((f_1) + I) = d-n+1 \ge 1$, we have $(f_1) + I \not\subseteq \bigcup_{\p \in \Assh A}\p$. Choose $x \in I$ such that $f_1 + x \not\in \p$ for every $\p \in \Assh A$. We set $a_1 = f_1 + x$. The classes of $f_1$ and $a_1$ in $A/I$ coincides. We are now assuming that there exists a subsystem $a_1, a_2, \ldots, a_i~(i<n)$ of parameters of $A$ such that the images of $f_j$ and $a_j$ in $A/I$ coincides for all $1 \le j \le i$. 
Then, because $\height_A((f_1, f_2, \ldots, f_{i+1}) + I) = d-n+i+1 \ge i+1$ and $\height_A \p = i$ for every $\p \in \Assh_A A/(a_1, a_2, \ldots, a_i)$, we get
$$
(f_{i+1}) + I \not\subseteq \bigcup_{\p \in \Assh_A A/(a_1, a_2, \ldots, a_i)}\p
$$
where $\Assh_A A/(a_1, a_2, \ldots, a_i)$ denotes the set of all  $\p \in \Supp_AA/(a_1, a_2, \ldots, a_i)$ such that $\dim A/\p = d-i$. Hence, we can choose $y \in I$ which satisfies $a_{i+1} = f_{i+1} + y \not\in \p$ for every $\p \in \Assh_A A/(a_1, a_2, \ldots, a_i)$. This completes the proof. 
\end{proof}

The goal of this section is stated as follows.

\begin{thm}\label{10.4}
Let $(A, \m)$ be a Gorenstein local ring with $d=\dim A>0$ and $I$ an equimultiple Ulrich ideal of $A$ with $\height_AI=1$ admitting its minimal reduction $Q=(a_1)$. Then the quasi-trivial extension $A(\alpha) = A \overset{\alpha}{\ltimes} I$ is an $n$-Goto ring for every $\alpha \in A$, where the sequence $a_2, a_2, \ldots, a_d$ is a system of parameters of $A/I$, $J=I + (a_2, a_2, \ldots, a_d)$, and $2n = \rme_0(J)$. In particular, the idealization $A \ltimes I$ and the fiber product $A\times_{A/I} A$ are $n$-Goto rings. 
\end{thm}

\begin{proof}
Suppose $d=1$. Then $J=I$. As $\mu_A(I) = 2$, we have $I/Q \cong A/I$ as an $A$-module. This yields that $\rme_0(I) = \ell_A(A/Q) = 2\cdot \ell_A(A/I)$ (\cite[Lemma 2.1]{GMP}). We set $T = A:I$. Note that $T$ is a birational module-finite extension of $A$, $A \ne T$, and $I = A:T$. Thanks to Theorem \ref{5.2}, the quasi-trivial extension $A(\alpha) = A \overset{\alpha}{\ltimes} I$ is $n$-Goto for every $\alpha \in A$, where $n$ is a half of $\rme_0(I)$. 

Suppose that $d \ge 2$. By Lemma \ref{10.3}, we may assume the sequence $a_2, a_2, \ldots, a_d$ forms a subsystem of parameters of $A$. Let $\overline{A}=A/\fkq$, $\overline{I} = I\overline{A}$, and $\overline{Q} = Q\overline{A}$. Then, for each $\alpha \in A$, we have the canonical isomorphisms
$$
A(\alpha)/\fkq A(\alpha) \cong \overline{A} \otimes_A (A \overset{\alpha}{\ltimes} I) \cong \overline{A} \overset{\overline{\alpha}}{\ltimes} \overline{I}
$$
where $\overline{\alpha}$ denotes the image of $\alpha$ in $\overline{A}$. We consider $A(\alpha)$ as an $A$-algebra via the homomorphism $\xi : A \to A(\alpha), a \mapsto (a, 0)$. Thus, the regular sequence $a_2, a_3, \ldots, a_d$ on $A$ forms a regular sequence on $A(\alpha)$ as well. We set $Q_1=Q + \fkq$. Then, because $\overline{I}^2 = \overline{Q} \cdot \overline{I}$, we have
$$
J^2 = [Q_1J + \fkq] \cap J^2 = Q_1 J + \fkq \cap J^2 = Q_1J + \fkq J = Q_1J
$$
where the third equality follows from $\fkq \cap I = \fkq I$ (remember that $\fkq$ is a parameter ideal of $A/I$). This induces that $\fkq \cap J^{m+1}=\fkq J^m$ for every $m \in \Bbb Z$; hence $a_2, a_3, \ldots, a_d$ forms a super-regular sequence of $A$ with respect to $J$. 
Since $A(\alpha)/\fkq A(\alpha) \cong \overline{A} \overset{\overline{\alpha}}{\ltimes} \overline{I}$ is an $n$-Goto ring with $2n = \rme_0(\overline{I})$ and $\rme_0(J) = \rme_0(\overline{I})$, by Theorem \ref{3.5} we conclude that $A(\alpha)$ is an $n$-Goto ring and $n$ is the half of $\rme_0(J)$. 
\end{proof}


We note an example. 

\begin{ex}\label{9.5e}
Let $d\ge 2$, $n \ge 1$, and $\ell \ge 2$ be integers. 
Let $(S, \n)$ be a regular local ring with $\dim S =d+1$ admitting its regular system $X_1, X_2, \ldots, X_d, Y$ of parameters. We set
$$
A = S/(Y^{2n}-X_1^{\ell}X_2\cdots X_d) \ \ \ \text{and} \ \ \ I=(x_1, y^n)
$$
where $x_i~(1 \le i \le d)$ and $y$ denote the images in $A$, respectively. Then $I$ is an Ulrich ideal of $A$ with $\height_AI=1$. Therefore the quasi-trivial extension $A(\alpha) = A \overset{\alpha}{\ltimes} I$ is an $n$-Goto ring for every $\alpha \in A$. 
\end{ex}

\begin{proof}
The ring $A/I \cong S/(X_1, Y^n)$ is Gorenstein and of dimension $d-1$. We set $Q=(x_1)$. Since $y^{2n} = x_1^{\ell}x_2\cdots x_d$ in $A$, we have 
$I^2 = QI + (y^{2n}) = QI$. If $I=Q$, then $y^n \in Q$, so that $Y^n \in (X_1, Y^{2n})$. This is impossible. Hence $I \ne Q$. As $I/Q$ is a cyclic $A/I$-module generated by the class of $y^{\ell}$, there is a surjective homomorphism $\varphi: A/I \to I/Q$. Let $X = \Ker \varphi$. We set $\fkq = (x_2, x_3, \ldots, x_d)$, $Q_1 = Q + \fkq$, and $J=I+\fkq$. By taking the functor $A/\fkq \otimes_A (-)$ to the surjective map, we obtain the exact sequence
$$
X/\fkq X \to A/J \to J/Q_1 \to 0
$$
of $A$-modules, because $I/(\fkq I + Q) \cong J/Q_1$. Note that $A/Q \cong S/(X_1, X_2, \ldots, X_d, Y^{2n})$ and $A/J \cong S/(X_1, X_2, \ldots, X_d, Y^{n})$. This shows that $\ell_A(J/Q_1) = n$. Therefore $X/\fkq X = (0)$ and hence $X=(0)$. Consequently, $I=(x_1, y^n)$ is an Ulrich ideal of $A$ with $\height_AI=1$. Besides, because $x_2, x_3, \ldots, x_d$ is a super-regular sequence of $A$ with respect to $J$, by passing through the ring $A/\fkq$ we get $\rme_0(J) = \ell_A(A/Q_1) = 2n$. Hence $A(\alpha)= A \overset{\alpha}{\ltimes} I$ is an $n$-Goto ring for every $\alpha \in A$. 
\end{proof}

When $\ell=1$, the ideal $I=(x_1, y^n)$ in Example \ref{9.5e} is not necessarily an Ulrich ideal. 

\begin{ex}
Let $S=k[[X, Y, Z]]$ be the formal power series ring over a field $k$. We set $A=S/(Z^2 - XY)$ and $\fkp =(x, z)$, where $x$ and $z$ stand for the images of $X$ and $Z$, respectively. Note that $A$ is a normal domain whose divisor class group $\Cl(A) \cong {\Bbb Z}/2 {\Bbb Z}$ is generated by the class $\cl(\fkp)$ of $\fkp$ (\cite[Theorem]{G}). Thus $\cl(\fkp) \ne 0$. If $\p$ has a principal reduction $Q=(a)$, we have $\p^{r+1} = a\p^r$ for some $r \ge 0$. Then $(r+1)\cl(\p) = r\cl(\p)$, so that $\cl(\p) = 0$. This makes a contradiction. Hence $\fkp=(x, z)$ is not an Ulrich ideal of $A$. 
\end{ex}



\section{Sally modules and Goto rings}\label{sec10}


In this section, unless otherwise specified, let $(A, \m)$ be a Cohen-Macaulay local ring with $d=\dim A>0$ and $I$ a canonical ideal of $A$. 

We start by characterizing $0$-Goto rings. 



\begin{thm}\label{11.1}
Consider the following conditions. 
\begin{enumerate}
\item[$(1)$] $A$ is a $0$-Goto ring. 
\item[$(2)$] $A$ is a Gorenstein ring. 
\item[$(3)$] There exists a parameter ideal $Q=(a_1, a_2, \dots, a_d)$ of $A$ such that $a_1 \in I$ and $I+Q = I$. 
\item[$(4)$] There exists a parameter ideal $Q=(a_1, a_2, \dots, a_d)$ of $A$ such that $a_1 \in I$ and $\calS_Q(J)=(0)$, where $J=I+Q$. 
\item[$(5)$] There exists a parameter ideal $Q=(a_1, a_2, \dots, a_d)$ of $A$ such that $a_1 \in I$, $Q$ is a reduction of $J$, and $\rme_1(J) < \rmr(A)$, where $J=I+Q$. 
\item[$(6)$] There exists a parameter ideal $Q=(a_1, a_2, \dots, a_d)$ of $A$ such that $a_1 \in I$, $Q$ is a reduction of $J$, and $\rme_1(J) = \rmr(A) -1$, where $J=I+Q$. 
\item[$(7)$] There exists a parameter ideal $Q=(a_1, a_2, \dots, a_d)$ of $A$ such that $a_1 \in I$, $Q$ is a reduction of $J$, and $\rme_1(J) =0$, where $J=I+Q$.  
\item[$(8)$] There exists a parameter ideal $Q=(a_1, a_2, \dots, a_d)$ of $A$ such that $a_1 \in I$, $Q$ is a reduction of $J$, and $\gr_J(A)$ is a Cohen-Macaulay ring, where $J=I+Q$. 
\end{enumerate}
Then the conditions $(1)$, $(4)$, $(5)$, $(6)$ are equivalent and the implications $(3) \Rightarrow (1)\Rightarrow (2)$, $(3) \Rightarrow (8) \Rightarrow (2)$, $(7) \Rightarrow (5)$ hold. If $A$ has an infinite residue class field $A/\m$, then the implications $(2) \Rightarrow (3)$, $(2) \Rightarrow (7)$ hold, so that all the conditions are equivalent. 
\end{thm}

\begin{proof}
Thanks to Proposition \ref{2.9}, we may assume the existence of a parameter ideal $Q=(a_1, a_2, \ldots, a_d)$ of $A$ satisfying the condition $(\sharp)$. We set $J=I+Q$ and $\fkq = (a_2, a_3, \ldots, a_d)$. 

$(3) \Rightarrow (1)$ Since $J=I$, we have $J^3 =QJ^2$ and $\ell_A(J^2/QJ) = 0$, i.e., $A$ is $0$-Goto. 

$(1) \Rightarrow (2)$ Suppose $d=1$. By setting $K=\frac{I}{a_1}$ in $\rmQ(A)$, the equality $K^2 = K$ holds. Hence $A=K$, because $A \subseteq K \subseteq K:K = A$. Thus $A$ is Gorenstein. We assume $d \ge 2$ and the condition $(2)$ holds for $d-1$. As $a_2 \in \fkq$ is a super-regular element of $A$ with respect to $J$, by Theorem \ref{3.5} the ring $A/(a_2)$ is $0$-Goto. The hypothesis of induction guarantees that $A/(a_2)$ is a Gorenstein ring. Hence $A$ is Gorenstein as well. 

$(1) \Leftrightarrow (4)$ This follows from the fact that the equality $J^2 = QJ$ holds if and only if $\calS_Q(J) = (0)$. See \cite[Lemma 2.1 (3)]{GNO} for the proof. 

$(4) \Rightarrow (6)$ By \cite[Theorem 3.7]{GMP}, we may assume $d \ge 2$. Since $J^2 = QJ$, we see that the sequence $a_2, a_3, \ldots, a_d$ of $A$ forms super-regular with respect to $J$. Let $\overline{A} = A/\fkq$ and $\overline{J} = J \overline{A}$. Then $\dim \overline{A} = 1$ and $\overline{J} = (I+\fkq)/\fkq \cong I/\fkq \cap I = I/\fkq I \cong \rmK_{\overline{A}}$. Hence the equalities
$$
\rme_1(J) = \rme_1(J/(a_2, a_3, \ldots, a_{d-1})) = \rme_1(\overline{J}) + \ell_A\left(\left[(0):_{\overline{A}}a_d\right]\right) = \rme_1(\overline{J}) = \rmr(\overline{A})-1 = \rmr(A) -1
$$
hold. 

$(5) \Rightarrow (4)$ By \cite[Proposition 3.6, Theorem 3.7]{GMP}, it suffices to assume $d \ge 2$. Let $\overline{A} = A/\fkq$, $\overline{J} = J \overline{A}$, and $\overline{Q} = Q\overline{A}$. 
Note that $\dim \overline{A} = 1$ and $\overline{Q}$ is a minimal reduction of $\overline{J} \cong \rmK_{\overline{A}}$. Since $\mu_A(J/Q) = \mu_A(\overline{J}/\overline{Q}) = \mu_A(\overline{J})-1 = \rmr(\overline{A}) - 1 = \rmr(A) -1$, we then have
\begin{eqnarray*}
\rank \calS_Q(J) + \ell_A(J/Q) \!\! &=& \!\! \rank \calS_Q(J) + \ell_A(A/Q) - \ell_A(A/J) \\
\!\! &=& \!\! \rank \calS_Q(J) + \rme_0(J) -  \ell_A(A/J) \\
\!\! &=& \!\! \rme_1(J) \\
\!\! &\le& \!\! \rmr(A) - 1 = \mu_A(J/Q) \le \ell_A(J/Q)
\end{eqnarray*} 
where the second equality follows because $J$ contains the parameter ideal $Q$ as a reduction. This shows $\rank \calS_Q(J) \le 0$ and hence $\calS_Q(J) = (0)$.

$(6) \Rightarrow (5), (7) \Rightarrow (5)$ These are obvious. 

$(3) \Rightarrow (8)$ The polynomial ring $\gr_J(A) \cong (A/J)[X_1, X_2, \ldots, X_d]$ is Cohen-Macaulay, because $J=Q$ is a parameter ideal of $A$.

$(8) \Rightarrow (2)$ Suppose $d = 1$. Let $\overline{A}$ be the integral closure of $A$ in $\rmQ(A)$.
We consider the fractional canonical ideal $K=\frac{I}{a_1}$ in $\rmQ(A)$. Suppose the contrary and seek a contradiction. Indeed, since $K \subsetneq K^2$, we have $A:\m \subseteq K:\m \subseteq K^2$. By \cite[Corollary 2.7]{VV}, the equality $Q \cap I^2 = QI$ holds. Let $\varphi \in A:\m$. Then $a_1^2 \varphi \in (a_1) \cap a_1^2K^2 = (a_1) \cap I^2 = QI = a_1 I = a_1^2K$, so that $\varphi \in K$. Thus $A:\m \subseteq K$. Thanks to \cite[Bemerkung 2.5]{HK}, we have
$$
A =K:K \subseteq K:(A:\m) = K:\left((K:K):\m\right) = K:(K:\m K) = \m K \subseteq \m \overline{A}
$$
whence $1 \in \m \overline{A} \cap A = \m$. This is absurd. Hence $A$ is Gorenstein. Suppose $d \ge 2$. Since $\gr_J(A)$ is Cohen-Macaulay, by \cite[Corollary 2.7]{VV} $\gr_{J/\fkq}(A/\fkq)$ is a Cohen-Macaulay ring. This implies $A/\fkq$ is Gorenstein, then so is $A$.

In the rest of this proof, we assume the field $A/\fkm$ is infinite. 

$(2) \Rightarrow (3)$ Since $A$ is Gorenstein, the ideal $I$ is principal. Then $\mu_A(J) = d$ because $J=(a_2, a_3, \ldots, a_d)+I $ and $Q$ is a minimal reduction of $J$. Hence $J=Q$.

$(2) \Rightarrow (7)$ The equivalence $(2)\Leftrightarrow (6)$ ensures that $\rme_1(J) = \rmr(A) -1 = 0$. 
\end{proof}

\begin{cor}
Let $n \ge 1$ be an integer and $Q$ a parameter ideal of $A$ with condition $(\sharp)$. If $A$ is an $n$-Goto ring with respect to $Q$, then 
 $\depth \gr_J(A) = d-1$, where $J=I+Q$.
\end{cor}

\begin{proof}
Since $\calS_Q(J) = \calT\left[\calS_Q(J)\right]_1$, by Lemma \ref{2.4} (1) we have $\depth \gr_J(A) \ge d-1$. Suppose $\depth \gr_J(A) = d$. Then Theorem \ref{11.1} $(8) \Rightarrow (2)$ yields $A$ is Gorenstein. This makes a contradiction, because every $n$-Goto ring with $n \ge 1$ is not Gorenstein; see Remark \ref{3.4r}. 
\end{proof}

Next, we consider the relation between $1$-Goto and almost Gorenstein properties. We first prepare an auxiliary lemma below. 


\begin{lem}\label{11.2}
Let $(A, \m)$ be a Noetherian local ring with $d = \dim A \ge 2$ and $J$ an $\m$-primary ideal of $A$. If $A$ has an infinite residue class field $A/\m$, the ideal $J$ contains  a superficial sequence $g_1, g_2, \ldots, g_d$ of $A$ with respect to $J$ such that $g_2$ forms a superficial element of $A$ with respect to $J$.  
\end{lem}

\begin{proof}
Let $\calR(J) = A[Jt] = \bigoplus_{i \ge 0}J^it^i$ be the Rees algebra of $J$, where $t$ is an indeterminate over $A$. We denote by $\calF$ the set of all associated prime ideals $\p \in \Spec \calR(J)$ of $\gr_J(A)$ such that $\p \not\supseteq \calR(J)_+$, where $\calR(J)_+ = \bigoplus_{i > 0}J^it^i$. We consider the map $\varphi : J \to \calR(J)$ defined by $\varphi(a) =at$ for each $a \in J$. Since $A/\m$ is infinite, we can choose $g_1 \in J$ such that $g_1 \not\in \bigcup_{\p \in \calF}\varphi^{-1}(\p)$. Let $\overline{A} = A/(g_1)$ and $\overline{J} = J\overline{A}$. We define $\overline{\calF}$ the set of all associated prime ideals $\q \in \Spec \calR(\overline{J})$ of $\gr_{\overline{J}}(\overline{A})$ such that $\q \not\supseteq \calR(\overline{J})_+$ and consider the map $\psi : \overline{J} \to \calR(\overline{J})$ with $\psi(a) =at$ for each $a \in \overline{J}$. Then, because $A/\m$ is infinite, we can choose $g_2 \in J$ which satisfies
$$
g_2 \not\in \left[ \bigcup_{\q \in \overline{\calF}}\left(\psi^{-1}(\q) \cap A\right)\right] \bigcup \left[ \bigcup_{\p \in \calF}\varphi^{-1}(\p)\right].
$$
This shows $g_1, g_2$ is a superficial sequence of $A$ with respect to $J$ and $g_2$ itself forms a superficial element of $A$ with respect to $J$.  
Then all you have to do is to choose the remaining superficial sequence $g_3, g_4, \ldots, g_d$ as usual.
\end{proof}

We state the following characterization of $1$-Goto rings. The condition $(5)$ of Theorem \ref{11.3} provides a natural generalization of \cite[Definition 3.1]{GMP}. The equivalence $(2) \Leftrightarrow (3)$ is essentially due to \cite[Theorem 5.1]{GTT}, and the implication $(5) \Rightarrow (2)$ was pointed out by S. Goto.


\begin{thm}\label{11.3}
Consider the following conditions. 
\begin{enumerate}
\item[$(1)$] $A$ is a $1$-Goto ring. 
\item[$(2)$] $A$ is an almost Gorenstein ring but not a Gorenstein ring. 
\item[$(3)$] There exists a parameter ideal $Q=(a_1, a_2, \dots, a_d)$ of $A$ such that $a_1 \in I$, $J \ne Q$, and $\m J = \m Q$, where $J=I+Q$. 
\item[$(4)$] There exists a parameter ideal $Q=(a_1, a_2, \dots, a_d)$ of $A$ such that $a_1 \in I$ and $\calS_Q(J) \cong \calB(-1)$ as a graded $\calT$-module, where $J=I+Q$ and $\calB = \calT/\m \calT$. 
\item[$(5)$] There exists a parameter ideal $Q=(a_1, a_2, \dots, a_d)$ of $A$ such that $a_1 \in I$, $Q$ is a reduction of $J$, and $\rme_1(J) = \rmr(A)$, where $J=I+Q$. 
\end{enumerate}
Then the conditions $(1)$ and $(4)$ are equivalent, and the implications $(3) \Rightarrow (1) \Rightarrow (2)$, $(1) \Rightarrow (5)$ hold. If $A$ has an infinite residue class field $A/\m$, then the implications $(2) \Rightarrow (3)$, $(5) \Rightarrow (2)$ hold, so that all the conditions are equivalent. 
\end{thm}

\begin{proof}
$(3) \Rightarrow (1)$ If $d=1$, the ideal $I$ contains a parameter ideal $Q$ such that $I \ne Q$ and $\m I = \m Q$. Since $\m$ is faithful as an $A$-module, the parameter ideal $Q$ is a reduction of $I$. By \cite[Theorems 3.11, 3.16]{GMP}, we conclude that $I^3 = QI^2$ and $\ell_A(I^2/QI) = 1$. Thus $A$ is $1$-Goto. When $d \ge 2$, the condition $(1)$ follows from \cite[Corollary 5.3]{GTT} and its proof. 

$(1) \Rightarrow (2)$ We first consider the case where $d = 1$. Choose a parameter ideal $Q$ of $A$ which satisfies $I^3 = QI^2$ and $\ell_A(I^2/QI) = 1$. 
Thanks to \cite[Theorem 3.16]{GMP}, the ring $A$ is almost Gorenstein, but not a Gorenstein ring. Assume $d \ge 2$ and the condition $(2)$ holds for $d-1$. Let $Q=(a_1, a_2, \ldots, a_d)$ be a parameter ideal of $A$ such that $a_1 \in I$, $J^3 =QJ^2$, and $\ell_A(J^2/QJ) = 1$, where $J=I+Q$. Then $a_2 \in J$ is a super-regular element of $A$ with respect to $J$. Hence, by Theorem \ref{3.5} $A/(a_2)$ is a $1$-Goto ring, so that the induction argument asserts that the ring $A/(a_2)$ is almost Gorenstein, but not Gorenstein. Therefore, we get the condition $(2)$ by \cite[Theorem 3.7 (1)]{GTT}.

$(1) \Leftrightarrow (4)$ We may assume the existence of a parameter ideal $Q=(a_1, a_2, \ldots, a_d)$ of $A$ such that $Q$ satisfies the condition $(\sharp)$. We set $J=I+Q$, $\calT =\calR(Q)$, and $\calB = \calT/\m \calT$. 
Suppose that $\calS_Q(J) \cong \calB(-1)$. Then, because $\calB$ is a homogeneous ring over $A/\m$ and $\left[\calS_Q(J)\right]_1 = J^2/QJ$, we have $J^3 = QJ^2$ and $\ell_A(J^2/QJ) = 1$. Hence $A$ is $1$-Goto. Conversely, we assume that $A$ is a $1$-Goto ring. Since $\calS_Q(J) = \calT \left[\calS_Q(J)\right]_1$ and $J^2/QJ \cong A/\m$, we see that $\calS_Q(J)$ is a cyclic $\calB$-module. Therefore we have an surjective homomorphism
$\calB(-1) \to \calS_Q(J)$ of graded $\calT$-modules. Hence $\calB(-1) \cong \calS_Q(J)$, because $\calB$ is an integral domain with $\dim \calB = \dim \calS_Q(J) = d$.

$(1) \Rightarrow (5)$ We choose a parameter ideal $Q = (a_1, a_2, \ldots, a_d)$ of $A$ satisfying $a_1 \in I$, $J^3 = QJ^2$, and $\ell_A(J^2/QJ) = 1$, where $J=I+Q$. If $d=1$, then $\rme_1(I) = \rmr(A)$ by \cite[Theorem 3.16]{GMP}. When $d \ge 2$, we set $\fkq = (a_2, a_3, \ldots, a_d)$. Note that $A/\fkq$ is a one-dimensional $1$-Goto ring. Since $a_2, a_3, \ldots, a_d$ forms a super-regular, and hence superficial, sequence of $A$ with respect to $J$, we conclude that the equalities
$$
\rme_1(J) = \rme_1(J/\fkq) + \ell_A\left([(0):_{A/\fkq}a_d]\right) = \rme_1(J/\fkq) = \rmr(A/\fkq) = \rmr(A)
$$
hold, as desired. 

We assume that $A/\m$ is infinite and show the implications $(2) \Rightarrow (3)$ and $(5) \Rightarrow (2)$.

$(2) \Rightarrow (3)$ This follows from \cite[Theorem 5.1]{GTT} and Theorem \ref{11.1}. Although \cite[Theorem 5.1]{GTT} is assumed the Gorensteinness of $\rmQ(A)$, we only need this assumption to get the existence of canonical ideals. 

$(5) \Rightarrow (2)$ We choose a parameter ideal $Q=(a_1, a_2, \dots, a_d)$ of $A$ such that $a_1 \in I$, $Q$ is a reduction of $J$, and $\rme_1(J) = \rmr(A)$, where $J=I+Q$. By \cite[Theorem 3.16]{GMP}, we assume that $d \ge 2$ and the condition $(2)$ holds for $d-1$. Then, by Lemma \ref{11.2} the ideal $J$ contains a superficial sequence $g_1, g_2, \ldots, g_d$ of $A$ with respect to $J$; while $g_2$ is a superficial element of $A$ with respect to $J$. Let $Q_1 = (g_1, g_2, \ldots, g_d)$. Lemma \ref{2.9} guarantees that $Q_1$ is a minimal reduction of $J$. Let $\overline{(*)}$ be the image in $A/I$. Then $(\overline{g_1}, \overline{g_2}, \ldots, \overline{g_d})$ is a reduction of $J/I$. Since $J/I = (\overline{a_2}, \overline{a_3}, \ldots, \overline{a_d})$ forms a parameter ideal of $A/I$, we get the equality
$$
(\overline{g_1}, \overline{g_2}, \ldots, \overline{g_d}) = (\overline{a_2}, \overline{a_3}, \ldots, \overline{a_d})
$$
inside of the ring $A/I$. Thus there is an integer $1 \le i \le d$ such that 
$$
J=I+ \left(g_1, g_2, \ldots, g_{i-1}, g_{i+1}\ldots, g_d\right).
$$
Without loss of generality, we may assume $a_2 \in J$ is a superficial element of $A$ with respect to $J$. We set $\overline{A} = A/(a_2)$, $\overline{J} = J\overline{A}$, and $\overline{I} = I \overline{A}$. Note that $\overline{I} \cong I/a_2I \cong \rmK_{\overline{A}}$ and $\overline{J} = \overline{I} + (a_3, a_4, \ldots, a_d)\overline{A}$ contains a parameter ideal $(a_1, a_3, a_4, \ldots, a_d)\overline{A}$ as a reduction. Besides, because $\rme_1(\overline{J}) = \rme_1(J) = \rmr(A) = \rmr(\overline{A})$, the hypothesis on $d$ shows $\overline{A} = A/(a_2)$ is almost Gorenstein, but not Gorenstein. Therefore $A$ is a non-Gorenstein almost Gorenstein ring. This completes the proof. 
\end{proof}

\begin{rem}
When $\dim R=1$, if $R$ is an almost Gorenstein local ring in the sense of \cite{GTT}, then $R$ is almost Gorenstein in the sense of \cite{GMP}. The converse does not hold in general (\cite[Remark 3.5]{GTT}, see also \cite[Remark 2.10]{GMP}), but it does when the canonical ideal contains a parameter ideal as a reduction, e.g., the residue field $A/\m$ is infinite, or $A$ is analytically irreducible. Hence, in Theorem \ref{11.3}, there is no need to distinguish between these two definitions of almost Gorenstein rings.
\end{rem}

On the other hand, there are no implications between Goto rings, and nearly Gorenstein rings defined in \cite{HHS}.

\begin{ex}\label{10.5e}
Let $k[[t]]$ be the formal power series ring over a field $k$. We set $A = k[[t^4, t^5, t^{11}]]$ and $K = A + At$. Then $K$ is a fractional canonical ideal of $A$ and $K^2 \ne K^3$. Thus $A$ is not a Goto ring. Note that
$$
A \cong k[[X, Y, Z]]/
\rmI_2
\begin{pmatrix}
X & Y^2 & Z \\
Y & Z & X^3
\end{pmatrix}
$$
where $k[[X, Y, Z]]$ denotes the formal power series ring over $k$. By \cite[Proposition 6.3]{HHS}, the ring $A$ is nearly Gorenstein. Conversely, let us consider the ring
$$
A = k[[X, Y, Z]]/
\rmI_2
\begin{pmatrix}
X^n & Y & Z \\
Y^2 & Z^2 & X^n
\end{pmatrix}
$$
with $n \ge 2$. Then $A$ is $n$-Goto, but it is not nearly Gorenstein. 
\end{ex}

We explore the characterization of $n$-Goto rings with $n \ge 2$ in terms of the structure of Sally modules of extended canonical ideals. 
Remember that $\m^{\ell}\calS_Q(J) = (0)$ for all $\ell \gg 0$. 



\begin{lem}
Let $n \ge 1$ be an integer. Let $Q$ be a parameter ideal of $A$ which satisfies the condition $(\sharp)$. Set $J=I+Q$. If $A$ is an $n$-Goto ring with respect to $Q$, then $\m^p \calS_Q(J) = (0)$ for every $p \ge n$. 
\end{lem}

\begin{proof}
We write $Q=(a_1, a_2, \ldots, a_d)$ with $a_1 \in I$ and set $\calT = \calR(Q)$. 
Suppose that $d=1$. By setting $K = \frac{I}{a_1}$ in $\rmQ(A)$, we get
$$
\left[\m^n\calS_Q(I)\right]_1 = \m^n(I^2/QI) \cong \m^n(K^2/K). 
$$
If $\m^n\calS_Q(I) \ne (0)$, then $\m^n K^2 \not\subseteq K$, i.e., $(\m/\fkc)^n \ne (0)$ in $A/\fkc$ where $\fkc = A:A[K] = A:K$. For each $1 \le i < n$, we have $(\m/\fkc)^i \ne (\m/\fkc)^{i+1}$. This makes a contradiction because $\ell_A(A/\fkc) = n$. Hence $\m^p \calS_Q(I) = (0)$ for every $p \ge n$. 

We assume $d \ge 2$ and our assertion holds for $d-1$. Let $\fkq = (a_2, a_3, \ldots, a_d)$. Since $\calS_Q(I) = \calT\left[\calS_Q(I)\right]_1$, the element $a_2 \in \fkq \setminus \m \fkq$ is super-regular of $A$ with respect to $J$. Let $\overline{A} = A/(a_2)$, $\overline{J} = J \overline{A}$, and $\overline{Q} = Q\overline{A}$. Then the parameter ideal $\overline{Q}$ of $\overline{A}$ satisfies the condition $(\sharp)$ and $\overline{A} =A/(a_2)$ is an $n$-Goto ring with respect to $\overline{Q}$. The induction hypothesis on $d$ shows that $\overline{\m}^p \calS_{\overline{Q}}(\overline{J}) = (0)$ for every $p \ge n$. Hence $\m^p \calS_Q(J) = (0)$. 
\end{proof}


The following generalizes both of \cite[Corollary 5.3 (2)]{GTT} and \cite[Theorem 3.7]{CGKM}. 

\begin{thm}\label{11.8}
Let $n \ge 1$ be an integer. Let $Q$ be a parameter ideal of $A$ which satisfies the condition $(\sharp)$. We set $J=I+Q$, $\calT=\calR(Q)$, and $\calB=\calT/\m \calT$.
Then the following conditions are equivalent.
\begin{enumerate}
\item[$(1)$] $A$ is an $n$-Goto ring with respect to $Q$. 
\item[$(2)$] There exist integers $0 \le \ell < n$ and $s_i \ge 1~(0 \le i \le \ell)$ such that $n = \sum_{i=0}^{\ell}s_i$ and $\m^{\ell} \calS_Q(J) \cong \calB(-1)^{\oplus s_0}$, and if $\ell>0$, there exist exact sequences
\begin{eqnarray*}
 0 \to \m^{\ell} \calS_Q(J) \to \!\! &\m^{\ell-1} \calS_Q(J) & \!\! \to \calB(-1)^{\oplus s_1} \to 0 \\
 0 \to \m^{\ell-1} \calS_Q(J) \to \!\! &\m^{\ell-2} \calS_Q(J) & \!\! \to \calB(-1)^{\oplus s_2} \to 0  \\
  &\vdots& \\
 0 \to \m \calS_Q(J) \to \!\!\!\!\!\!\! &\calS_Q(J)& \!\!\!\!\!\!\! \to \calB(-1)^{\oplus s_{\ell}} \to 0 
\end{eqnarray*}
of graded $\calT$-modules.
\end{enumerate}
When this is the case, if $\ell > 0$, we have a non-split exact sequence 
$$
 0 \to \calB(-1)^{\oplus s_0} \to \m^{\ell-1} \calS_Q(J) \to \calB(-1)^{\oplus s_1} \to 0
$$
of graded $\calT$-modules. 
\end{thm}

\begin{proof}
$(2) \Rightarrow (1)$ If $\ell =0$, then $\calS_Q(J) \cong \calB(-1)^{\oplus s_0}$. 
As $\calB \cong (A/\m)\left[\calB_1\right]$, we have $\calS_Q(J) = \calT \left[\calS_Q(J)\right]_1$. Localizing the isomorphism at $\p = \m \calT \in \Spec \calT$ induces that $[\calS_Q(J)]_\p \cong (\calT/\p)_\p^{\oplus s_0} \cong (\calT_{\p}/\p \calT_\p)^{\oplus s_0}$. This shows $\rank \calS_Q(J) = \ell_{\calT_{\p}}(\left[\calS_Q(J)\right]_{\p}) = s_0 = n$, and hence $A$ is an $n$-Goto ring with respect to $Q$. 
We assume $\ell >0$. Because of $\m^{\ell} \calS_Q(J) \cong \calB(-1)^{\oplus s_0}$, the exact sequence 
$$
 0 \to \m^{\ell} \calS_Q(J) \to \m^{\ell-1} \calS_Q(J)  \to \calB(-1)^{\oplus s_1} \to 0
$$
yields that the graded $\calT$-module $\m^{\ell-1} \calS_Q(J)$ is generated by the homogeneous elements of degree one, while by the sequence  
$$
 0 \to \m^{\ell-1} \calS_Q(J) \to \m^{\ell-2} \calS_Q(J)  \to \calB(-1)^{\oplus s_2} \to 0 
$$
of graded $\calT$-modules, the degree one elements generate the $\calT$-module 
$\m^{\ell-2} \calS_Q(J)$, as well. 
Repeating the same process for $\m^i \calS_Q(J)$ recursively guarantees that $\calS_Q(J) = \calT \left[\calS_Q(J)\right]_1$. 
By localizing the exact sequences at $\p=\m \calT$, we get the sequences
\begin{eqnarray*}
 0 \to \left[\m^{\ell} \calS_Q(J)\right]_\p \to \!\! & \left[\m^{\ell-1} \calS_Q(J)\right]_\p & \!\! \to \left(\calT_\p/\p\calT_\p\right)^{\oplus s_1} \to 0 \\
 0 \to \left[\m^{\ell-1} \calS_Q(J)\right]_\p \to \!\! & \left[\m^{\ell-2} \calS_Q(J)\right]_\p & \!\! \to \left(\calT_\p/\p\calT_\p\right)^{\oplus s_2}\to 0  \\
  &\vdots& \\
 0 \to \left[\m \calS_Q(J)\right]_\p \to \!\!\!\!\!\!\! &\left[\calS_Q(J)\right]_\p& \!\!\!\!\!\!\! \to \left(\calT_\p/\p\calT_\p\right)^{\oplus s_{\ell}} \to 0 
\end{eqnarray*}
of $\calT_{\p}$-modules. Therefore we get 
\begin{eqnarray*}
\rank \calS_Q(J) &=& \ell_{\calT_\p}(\left[\calS_Q(J)\right]_\p) \\
&=& \ell_{\calT_\p}(\left[\m \calS_Q(J)\right]_\p) + s_{\ell} \\ 
&=&  \ell_{\calT_\p}(\left[\m^2 \calS_Q(J)\right]_\p) + s_{\ell - 1} + s_{\ell} \\
&=& \cdots =   \ell_{\calT_\p}(\left[\m^{\ell} \calS_Q(J)\right]_\p) +  \sum_{i=1}^\ell s_i = n 
\end{eqnarray*}
where the last equality follows from $\m^{\ell} \calS_Q(J) \cong \calB(-1)^{\oplus s_0}$. Therefore $A$ is an $n$-Goto ring with respect to $Q$. 

$(1) \Rightarrow (2)$ Since $n \ge 1$, we note that $\calS_Q(J) \ne (0)$. We define the integer $\ell =\max\{0 \le k \in \Bbb Z \mid \m^k \calS_Q(J) \ne (0)\}$. Then $0 \le \ell < n$. As $\m^{\ell}\calS_Q(J) \ne (0)$ and $\calS_Q(J) = \calT \left[\calS_Q(J)\right]_1$, we get $\left[\m^i\calS_Q(J)\right]_1 \ne (0)$ for every $0 \le i \le \ell$. Besides $\m^{\ell+1}J^2 \subseteq QJ$ because $\m^{\ell + 1} \calS_Q(J) = (0)$. For each $0 \le i \le \ell$, we set $s_i = \mu_A(\left[\m^{\ell-i}\calS_Q(J)\right]_1)$. Then $s_i \ge 1$ and $\sum_{i=0}^\ell s_i = n$. It is straightforward to check that the equality
$$
\mu_A(\left[\m^i\calS_Q(J)\right]_1) = \mu_{\calT_\p}(\left[\m^i \calS_Q(J)\right]_{\p})
$$
holds, where $\p = \m \calT \in \Spec \calT$. 
Since $\m^i \calS_Q(J)/\m^{i+1} \calS_Q(J)$ is a graded $\calB$-module generated by the homogeneous elements of degree one, we get a surjective homomorphism $\varphi: \calB(-1)^{\oplus s_{\ell-i}} \to \m^i \calS_Q(J)/\m^{i+1} \calS_Q(J)$. Let $K = \Ker \varphi$ and assume that $K \ne (0)$. 
Then $\Ass_{\calT}K =\{\p\}$, and the exact sequence
$$
0 \to K_{\p} \to \left(\calT_{\p}/\p\calT_{\p}\right)^{\oplus s_{\ell-i}} \to \left[\m^i \calS_Q(J)/\m^{i+1} \calS_Q(J)\right]_{\p} \to 0
$$
yields that $s_{\ell-i} = \ell_{\calT_{\p}}(K_{\p}) + \mu_{\calT_{\p}}(\left[\m^i \calS_Q(J)\right]_{\p}) = \ell_{\calT_{\p}}(K_{\p}) + s_{\ell-i}$. 
This shows $K_{\p}=(0)$, which is impossible. Thus $K=(0)$, and hence we have an isomorphism
$$
\m^i\calS_Q(J)/\m^{i+1}\calS_Q(J) \cong \calB(-1)^{\oplus s_{\ell-i}}
$$
of graded $\calT$-modules for every $0 \le i \le \ell$. In particular, $\m^{\ell} \calS_Q(J) \cong \calB(-1)^{\oplus s_0}$. 
When we have the equivalent conditions, if $\ell>0$, the exact sequence 
$$
 0 \to \calB(-1)^{\oplus s_0} \to \m^{\ell-1} \calS_Q(J) \to \calB(-1)^{\oplus s_1} \to 0
$$
is not split, because $\m^{\ell} \calS_Q(J) \ne (0)$.
\end{proof}

The following is a generalization of \cite[Corollary 5.3 (3)]{GTT}.

\begin{cor}\label{10.8ab}
Let $n \ge 1$ be an integer. Let $Q$ be a parameter ideal of $A$ which satisfies the condition $(\sharp)$. We set $J=I+Q$. Suppose that $A$ is an $n$-Goto ring with respect to $Q$. Then the Hilbert function of $A$ with respect to $J$ is given by
$$
\ell_A(A/J^{\ell+1}) =
\rme_0(J)\cdot\binom{\ell+d}{d} -\left[\rme_0(J) - \ell_A(A/J) + n\right]\cdot\binom{\ell+d-1}{d-1} + n\cdot\binom{\ell+d-2}{d-2}
$$
for all $\ell \ge 0$. Hence, $\rme_2(J) = n$ if $d \ge 2$, and $\rme_i(J) = 0$ for all $3 \le i \le d$ if $d \ge 3$.
\end{cor}

\begin{proof}
By Fact \ref{2.3} (3), we have
$$
\ell_A(A/J^{\ell+1}) = \rme_0(J)\cdot\binom{\ell+d}{d} - \left[\rme_0(J) - \ell_A(A/J)\right]\cdot\binom{\ell+d-1}{d-1} - \ell_A(\left[\calS_Q(J)\right]_\ell)
$$ for all $\ell \ge 0$. 
We set $\calT=\calR(Q)$ and $\calB=\calT/\m \calT$.
Since $A$ is an $n$-Goto ring with respect to $Q$, we can choose integers $0 \le t < n$ and $s_i \ge 1~(0 \le i \le t)$ such that $n = \sum_{i=0}^{t}s_i$ and $\m^{t} \calS_Q(J) \cong \calB(-1)^{\oplus s_0}$, and if $t>0$, there exist exact sequences
\begin{eqnarray*}
 0 \to \m^{t} \calS_Q(J) \to \!\! &\m^{t-1} \calS_Q(J) & \!\! \to \calB(-1)^{\oplus s_1} \to 0 \\
 0 \to \m^{t-1} \calS_Q(J) \to \!\! &\m^{t-2} \calS_Q(J) & \!\! \to \calB(-1)^{\oplus s_2} \to 0  \\
  &\vdots& \\
 0 \to \m \calS_Q(J) \to \!\!\!\!\!\!\! &\calS_Q(J)& \!\!\!\!\!\!\! \to \calB(-1)^{\oplus s_{\ell}} \to 0 
\end{eqnarray*}
of graded $\calT$-modules. Hence, for each $\ell \ge 0$, we get
$\ell_A(\left[\calS_Q(J)\right]_\ell) = \sum_{i=0}^t s_i \cdot \ell_A(\calB_{\ell-1}) = n \cdot \left[\binom{\ell + d-1}{d-1} - \binom{\ell + d-2}{d-2}\right]$, 
because $\calB \cong (A/\m)[X_1, X_2, \ldots, X_d]$ is the polynomial ring over $A/\m$. Therefore the equality
$$
\ell_A(A/J^{\ell+1}) =
\rme_0(J)\cdot\binom{\ell+d}{d} -\left[\rme_0(J) - \ell_A(A/J) + n\right]\cdot\binom{\ell+d-1}{d-1} + n\cdot\binom{\ell+d-2}{d-2}
$$
holds. 
\end{proof}

We summarize some consequences of Theorem \ref{11.8}.


\begin{cor}
Let $n \ge 1$ be an integer. 
Then the following conditions are equivalent. 
\begin{enumerate}
\item[$(1)$] $A$ is an $n$-Goto ring.
\item[$(2)$] There exist a parameter ideal $Q=(a_1, a_2, \dots, a_d)$ of $A$ and integers $0 \le \ell < n$, $s_i \ge 1~(0 \le i \le \ell)$ such that $a_1 \in I$, $n = \sum_{i=0}^{\ell}s_i$, and $\m^{\ell} \calS_Q(J) \cong \calB(-1)^{\oplus s_0}$, and if $\ell>0$, there exist exact sequences
\begin{eqnarray*}
 0 \to \m^{\ell} \calS_Q(J) \to \!\! &\m^{\ell-1} \calS_Q(J) & \!\! \to \calB(-1)^{\oplus s_1} \to 0 \\
 0 \to \m^{\ell-1} \calS_Q(J) \to \!\! &\m^{\ell-2} \calS_Q(J) & \!\! \to \calB(-1)^{\oplus s_2} \to 0  \\
  &\vdots& \\
 0 \to \m \calS_Q(J) \to \!\!\!\!\!\!\! &\calS_Q(J)& \!\!\!\!\!\!\! \to \calB(-1)^{\oplus s_{\ell}} \to 0 
\end{eqnarray*}
of graded $\calT$-modules, where $J=I+Q$, $\calT=\calR(Q)$, and $\calB=\calT/\m \calT$.
\end{enumerate}
When this is the case, if $\ell > 0$, we have a non-split exact sequence 
$$
 0 \to \calB(-1)^{\oplus s_0} \to \m^{\ell-1} \calS_Q(J) \to \calB(-1)^{\oplus s_1} \to 0
$$
of graded $\calT$-modules. 
\end{cor}

\begin{cor}\label{11.9}
Let $n \ge 1$ be an integer. Let $Q$ be a parameter ideal of $A$ which satisfies the condition $(\sharp)$. We set $J=I+Q$, $\calT=\calR(Q)$, and $\calB = \calT/\m \calT$. 
Then the following conditions are equivalent.
\begin{enumerate}
\item[$(1)$] $A$ is an $n$-Goto ring with respect to $Q$ and $\m^{n-1} \calS_Q(J) \ne (0)$.
\item[$(2)$] There exists an isomorphism $\m^{n-1} \calS_Q(J) \cong \calB(-1)$, and if $n \ge 2$, there exist exact sequences
\begin{eqnarray*}
 0 \to \m^{n-1} \calS_Q(J) \to \!\! &\m^{n-2} \calS_Q(J) & \!\! \to \calB(-1) \to 0 \\
 0 \to \m^{n-2} \calS_Q(J) \to \!\! &\m^{n-3} \calS_Q(J) & \!\! \to \calB(-1) \to 0  \\
  &\vdots& \\
 0 \to \m \calS_Q(J) \to \!\!\!\!\!\!\! &\calS_Q(J)& \!\!\!\!\!\!\! \to \calB(-1) \to 0 
\end{eqnarray*}
of graded $\calT$-modules.
\end{enumerate}
When this is the case, if $n \ge 2$, we have a non-split exact sequence 
$$
 0 \to \calB(-1) \to \m^{n-2} \calS_Q(J) \to \calB(-1) \to 0
$$
of graded $\calT$-modules. 
\end{cor}

\begin{proof}
We maintain the notation as in the proof of Theorem \ref{11.8}. 
Since $\m^{n-1} \calS_Q(J) \ne (0)$, we have $\ell = n-1$, so that $s_i = 1$ for every $0 \le i \le \ell$. 
\end{proof}

For an integer $n \ge 2$, we assume $A$ is an $n$-Goto ring with $\dim A=1$. 
Then the canonical ideal $I$ contains a parameter ideal $Q=(a)$ as a reduction, and  $\m^{n-1}\calS_Q(I) \ne (0)$ if and only if $v(A/\fkc) = 1$, where $K = \frac{I}{a}$ in $\rmQ(A)$. Since the equivalent conditions hold for $n=2$, the next provides a natural generalization of \cite[Theorem 3.7 $(1) \Leftrightarrow (2)$]{CGKM}.


\begin{cor}
Let $n \ge 1$ be an integer. Suppose that $d=1$ and $I$ contains a parameter ideal $Q=(a)$ as a reduction. We set $K=\frac{I}{a}$, $\fkc = A:A[K]$, $\calT=\calR(Q)$, and $\calB=\calT/\m \calT$.
Then the following conditions are equivalent.
\begin{enumerate}
\item[$(1)$] $A$ is an $n$-Goto ring and $v(A/\fkc) = 1$.
\item[$(2)$] There exists an isomorphism $\m^{n-1} \calS_Q(J) \cong \calB(-1)$, and if $n \ge 2$, there exist exact sequences
\begin{eqnarray*}
 0 \to \m^{n-1} \calS_Q(J) \to \!\! &\m^{n-2} \calS_Q(J) & \!\! \to \calB(-1) \to 0 \\
 0 \to \m^{n-2} \calS_Q(J) \to \!\! &\m^{n-3} \calS_Q(J) & \!\! \to \calB(-1) \to 0  \\
  &\vdots& \\
 0 \to \m \calS_Q(J) \to \!\!\!\!\!\!\! &\calS_Q(J)& \!\!\!\!\!\!\! \to \calB(-1) \to 0 
\end{eqnarray*}
of graded $\calT$-modules.
\end{enumerate}
When this is the case, if $n \ge 2$, we have a non-split exact sequence 
$$
 0 \to \calB(-1) \to \m^{n-2} \calS_Q(J) \to \calB(-1) \to 0
$$
of graded $\calT$-modules. 
\end{cor}

\section{Exact sequences and Goto rings}\label{sec11}


We investigate the relation between Goto rings and the existence of a certain embedding of the ring into its canonical module. We begin with the setup. 

\begin{setup}
Let $(A, \m)$ be a Cohen-Macaulay local ring with $d=\dim A>0$ admitting a canonical module $\rmK_A$. Let $I$ be a canonical ideal of $A$ and $Q=(a_1, a_2, \ldots, a_d)$ a parameter ideal of $A$ with condition $(\sharp)$. We set $J = I+Q$ the extended canonical ideal of $A$.  
\end{setup}

Recall that, for a finitely generated $A$-module $M$ and an $\m$-primary ideal $\fka$ of $A$, we say that $M$ is an {\it Ulrich $A$-module with respect to $\fka$}, if $M$ is a Cohen-Macaulay $A$-module, $\rme^0_{\fka}(M) = \ell_A(M/\fka M)$, and $M/\fka M$ is free as an $A/\fka$-module, where $\rme^0_{\fka}(M)$ denotes the $0$-th Hilbert coefficient of $M$ with respect to $\fka$ (see \cite[Definition 2.1]{GK}, \cite[Definition 1.2]{GOTWY}, \cite[Definition 2.1]{GTT}).

With this notation we have the following. 

\begin{thm}\label{11.2a}
Let $n \ge 1$ be an integer.
Suppose that $A$ is an $n$-Goto ring with respect to $Q$ and $\m^{n-1} \calS_Q(J) \ne (0)$. Then there exist an exact sequence 
$$
0 \to A \to \rmK_A \to C \to 0
$$
of $A$-modules and, for each $1 \le i \le n$, an $\m$-primary ideal $\fka_i$ of $A$ with $\ell_A(A/\fka_i)=i$ such that $C \cong \bigoplus_{i=1}^n M_i$ and $M_n \ne (0)$, where $M_i$ denotes an Ulrich $A$-module with respect to $\fka_i$ for each $1 \le i \le n$. 
\end{thm}

\begin{proof}
We may assume $n \ge 2$. 
Suppose $d=1$. We set $K = \frac{I}{a_1}$ inside of the ring $\rmQ(A)$. Note that $K$ is a fractional canonical ideal of $A$. Since $A$ is $n$-Goto, we get $K^2 = K^3$ and $Q:_AJ = Q:_AI = A:K = A:A[K]$. By setting $\fkc = A:A[K]$, we see that $v(A/\fkc) = 1$. As $\ell_A(A/\fkc) =n$, we can choose a minimal system $x_1, x_2, \ldots, x_{\ell}$ of generators of $\m$ such that $\fkc =(x_1^n, x_2, \ldots, x_{\ell})$, where $\ell = v(A)$. For each $1 \le i \le n$, we define
$$
I_i = (x_1^i, x_2, \ldots, x_{\ell}).
$$
Thanks to Theorem \ref{4.8}, there is an exact sequence
$$
0 \to A \to K \to C \to 0
$$
of $A$-modules such that 
$$
C\cong K/A \cong \bigoplus_{i=1}^{n} \left(A/I_i \right)^{\oplus \ell_i}
$$
where $\ell_n >0$, $\ell_i \ge 0~(1 \le i < n)$, and $\sum_{i=1}^{n} \ell_i = \rmr(A)-1$. The $A$-module $M_i= \left(A/I_i\right)^{\oplus \ell_i}$ is Ulrich with respect to $I_i$ and $\ell_A(A/I_i) = i$. Moreover, as $\ell_n >0$, we get $M_n \ne (0)$. 

It suffices to assume that $d \ge 2$. Let $\fkq = (a_2, a_3, \ldots, a_d)$ and $\overline{A} = A/\fkq$. We set $\overline{\m} = \m \overline{A}$, $\overline{J} = J\overline{A}$, and $\overline{Q} = Q\overline{A}$. Then $\overline{A}$ is $n$-Goto with respect to $\overline{Q}$ and $\overline{\m}^{n-1} \calS_{\overline{Q}}(\overline{J}) \ne (0)$. Let $x_1, x_2, \ldots, x_{\ell}$ be a minimal system of generators of $\overline{\m}$ such that $\overline{Q}:_{\overline{A}}\overline{J}=(x_1^n, x_2, \ldots, x_{\ell})$, where $\ell = v(\overline{A})$. For each $1 \le i \le n$, we consider the ideal $I_i = (x_1^i, x_2, \ldots, x_{\ell})$ in $\overline{A}$. Choose an element $X_i \in \m$ so that $x_i$ is the image of $X_i$ in $\overline{A}$. We set 
$$
J_i = (X_1^i, X_2, \ldots, X_{\ell}) \ \ \text{and} \  \ \fka_i = J_i + \fkq
$$
for each $1 \le i \le n$. Then 
$$
Q:_AJ = \fka_n \subsetneq \fka_{n-1} \subsetneq \cdots \subsetneq \fka_1 = \m
$$
forms a chain of $\m$-primary ideals in $A$ and $\ell_A(A/\fka_i) = i$. 
Remember that $\rmK_A/\fkq \rmK_A \cong \rmK_{\overline{A}}$ (\cite[Korollar 6.3]{HK}). Thus there is an exact sequence
$$
0 \to \overline{A} \overset{\psi}{\longrightarrow} \rmK_{\overline{A}} \to D \to 0
$$
of $\overline{A}$-modules such that 
$$
D \cong \bigoplus_{i=1}^n \left(\overline{A}/I_i\right)^{\oplus \ell_i} \cong \bigoplus_{i=1}^n \left(A/\fka_i\right)^{\oplus \ell_i}
$$
where $\ell_n >0$, $\ell_i \ge 0~(1 \le i < n)$, and $\sum_{i=1}^{n} \ell_i = \rmr(A)-1$. Let $\xi \in \rmK_A$ such that $\psi(1) = \overline{\xi}$, where $\overline{\xi}$ denotes the image of $\xi$ in $\rmK_A/\fkq \rmK_A \cong \rmK_{\overline{A}}$. We now consider the exact sequence 
$$
A \overset{\varphi}{\longrightarrow} \rmK_A \to C \to 0
$$
of $A$-modules with $\varphi(1) = \xi$. Then, because $\psi = \overline{A} \otimes \varphi$, we get $D \cong C/\fkq C$; whence $\dim_A C \le d-1 <d$. By \cite[Lemma 3.1 (1)]{GTT}, the homomorphism $\varphi$ is injective. Hence $C$ is a Cohen-Macaulay $A$-module with $\dim_AC = d-1$ (\cite[Lemma 3.1 (2)]{GTT}). This yields that $a_2, a_3, \ldots, a_d \in \fkq$ forms a regular sequence on $C$. Let 
$$
C = \bigoplus_{i=1}^t\left[\bigoplus_{j=1}^{s_i} N_j^{(i)}\right]
$$
be the indecomposable decomposition of the $A$-module $C$, where $t \ge 1$, $s_i \ge 0~(1 \le i <t)$, $s_t>0$, and $N_j^{(i)}$ is an indecomposable $A$-submodule of $C$. We then have the isomorphisms
$$
\bigoplus_{i=1}^n \left(A/\fka_i\right)^{\oplus \ell_i} \cong D \cong C/\fkq C \cong \bigoplus_{i=1}^t\left[\bigoplus_{j=1}^{s_i} N_j^{(i)}/\fkq N_j^{(i)}\right].
$$
As $C/\fkq C$ has finite length, after renumbering the indices if necessary, we have that
\begin{center}
$n=t$, \ $\ell_i = s_i$ for all $1 \le i \le n$, \ and  \ $N_j^{(i)}/\fkq N_j^{(i)} \cong A/\fka_i$ for all $1 \le j \le \ell_i$.
\end{center}
For each $1 \le i \le n$, the $\overline{A}$-module 
$$
M_i/\fkq M_i \cong \bigoplus_{j=1}^{s_i} N_j^{(i)}/\fkq N_j^{(i)} \cong \left(\overline{A}/I_i\right)^{\oplus \ell_i} 
$$
is Ulrich with respect to $I_i$, where $M_i = \bigoplus_{j=1}^{s_i} N_j^{(i)}$; hence $M_i$ is an Ulrich $A$-module with respect to $\fka_i = J_i + \fkq$ (\cite[Proposition 2.6 (4)]{GK}, \cite[Proposition 2.2 (5)]{GTT}). In particular, since $\ell_n >0$, we see that $M_n \ne (0)$. This completes the proof. 
\end{proof}

We close this paper by stating an example.

\begin{ex}
Let $T=k[[X, Y, Z, W]]$ stands for the formal power series ring over a field $k$. We set $A = T/\rmI_2
\left(\begin{smallmatrix}
X^2 & Y+W & Z \\
Y & Z & X^2
\end{smallmatrix}\right)
$ and denote by $x, y, w$ the images of $X, Y, W$ in $A$, respectively. The ring $A$ admits a canonical ideal $I=(x^2, y)$ and a parameter ideal $Q = (x^2, w)$. Then, by setting $J=I+Q$, we get $J^3 = QJ^2$ and $\ell_A(J^2/QJ) = 2$. This implies that $A$ is a $2$-Goto ring with respect to $Q$ and $\m \calS_Q(J) \ne (0)$. Let 
$$
0  \to A \overset{\varphi}{\longrightarrow} I \to C \to 0
$$
be an exact sequence of $A$-modules with $\varphi(1) = y$. Then $\fka C = w I$ and $C/\fka C \cong A/\fka$, where $\fka = Q:_AJ = (x^2, y, z, w)$. Hence $C$ is an Ulrich $A$-module with respect to $\fka$ and $\dim_AC=1$. 
\end{ex}


\begin{ac}
The author would like to thank S. Goto for his helpful and simultaneous discussions during this research. The author is grateful also to B. Ulrich, N. Matsuoka, K. Ozeki, R. Isobe, and S. Kumashiro for their valuable advices and comments.
\end{ac}




\end{document}